\pdfoutput=1
\documentclass[a4paper]{amsart}
\usepackage{color} 

\usepackage{graphicx}
\usepackage{amssymb}

\usepackage[mathscr]{eucal} 
\usepackage{bbm} 

\usepackage{bm} 

\usepackage{textcomp} 

\usepackage[all]{xy}

\newcommand{\stacks}[1]{[\href{https://stacks.math.columbia.edu/tag/#1}{St:~#1}]}
\newcommand{\arxiv}[1]{\href{https://arxiv.org/abs/#1}{\texttt{arXiv:\allowbreak #1}}}

\newcommand{\xhat}[1]{\bm\hat{#1}}
\newcommand{\xtilde}[1]{\bm\tilde{#1}}
\newcommand{\xbar}[1]{\bm\bar{#1}}

\newcommand{\et}{\textup{\'et}}
\newcommand{\nop}{\textup{op}}
\newcommand{\npf}{\textit{pf}}
\newcommand{\nan}{\textup{an}}

\newcounter{ctr} \numberwithin{ctr}{subsection}

\theoremstyle{plain}
\newtheorem{thm}[ctr]{Theorem}
\newtheorem{lem}[ctr]{Lemma}
\newtheorem{prp}[ctr]{Proposition}
\newtheorem{cor}[ctr]{Corollary}

\theoremstyle{definition}
\newtheorem{dfn}[ctr]{Definition}
\newtheorem{rmk}[ctr]{Remark}
\newtheorem{exa}[ctr]{Example}

\newtheorem*{conv}{Convention}
\newtheorem*{rmk*}{Remark}

\newtheorem*{intrormk*}{Remark}

\usepackage{hyperref}

\DeclareMathOperator{\GL}{GL}
\DeclareMathOperator{\coker}{coker}

\DeclareMathOperator{\Tor}{Tor}

\newcommand{\bF}{\mathbb{F}}
\newcommand{\bG}{\mathbb{G}}
\newcommand{\bQ}{\mathbb{Q}}
\newcommand{\bR}{\mathbb{R}}
\newcommand{\bZ}{\mathbb{Z}}

\newcommand{\cA}{\mathcal{A}}
\newcommand{\cB}{\mathcal{B}}
\newcommand{\cC}{\mathcal{C}}

\newcommand{\cF}{\mathcal{F}}

\newcommand{\cM}{\mathcal{M}}
\newcommand{\cN}{\mathcal{N}}
\newcommand{\cO}{\mathcal{O}}

\newcommand{\fm}{\mathfrak{m}}

\newcommand{\fp}{\mathfrak{p}}

\newcommand{\Fq}{\bF_{\!q}}
\newcommand{\Fqbar}{\smash{\xbar{\bF}}}

\newcommand{\uH}{\mathrm{H}}

\newcommand{\rF}[1]{#1(\!(z)\!)}
\newcommand{\rO}[1]{#1[[z]]}
\newcommand{\rFt}[1]{#1(\!(t^{-1})\!)}

\DeclareMathOperator{\ord}{ord}

\newcommand{\app}[2]{{}^{#1}{#2}}
\newcommand{\abar}{\bm\bar{a}}

\newcommand{\nell}{\fp}
\newcommand{\nellab}{p}

\newcommand{\npF}[1]{F_{#1}}

\newcommand{\nFinf}{\npF{\infty}}

\newcommand{\nFell}{\npF{\nell}}

\newcommand{\nKc}{K}

\newcommand{\nFloc}{\smash{\bm\hat{F}}}
\newcommand{\nOloc}{\cO}
\newcommand{\nlocm}{\fm}
\newcommand{\nUK}[1]{Y_{K,#1}}
\newcommand{\nFF}[1]{X_{K,#1}}

\newcommand{\nlUK}{\nUK{\nFloc}}
\newcommand{\nlFF}{\nFF{\nFloc}}

\newcommand{\nsigma}{\sigma}
\newcommand{\nlin}{\textit{lin}}
\newcommand{\taulin}{\tau^{\nlin}}

\newcommand{\unit}{\mathbbm{1}} 
\newcommand{\unitob}[1]{\mathbbm{1}_{#1}} 

\newcommand{\ncoef}{F}
\newcommand{\ncoint}{\cO}
\newcommand{\ncomax}{\fm}
\newcommand{\ncons}{\kappa}
\newcommand{\nresdeg}{d}
\newcommand{\ninfdeg}{d}

\newcommand{\nresf}{k}
\newcommand{\nadic}{\ncomax}

\newcommand{\nSig}[1]{#1^\nan}

\newcommand{\qcshet}[1]{#1_\et}
\newcommand{\shet}[1]{#1_\et^\tau}
\newcommand{\shetx}[1]{(#1)_\et^\tau}
\newcommand{\etsite}[1]{(\Spec #1)_\et}

\newcommand{\nslope}{\lambda}
\newcommand{\nslopealt}{\mu}

\newcommand{\nerank}{r}
\newcommand{\neshift}{h}

\newcommand{\nval}{v}
\newcommand{\nDval}{v_\nDN}

\newcommand{\nvaltau}{v_{\tau^{-1}}}
\newcommand{\nvalinf}{v_{\infty}}
\newcommand{\nOK}{R}

\newcommand{\nur}{u}
\newcommand{\nKur}{K^\nur}

\newcommand{\nKcur}{K^c}
\newcommand{\nRcur}{\nOK^c}

\newcommand{\symnB}{\cA}
\newcommand{\symnBz}{\cB}

\newcommand{\nVarB}{\symnBz}

\newcommand{\nBR}[1]{\symnB_{#1}}
\newcommand{\nB}{\nBR{\nOK}}
\newcommand{\nBK}{\nBR{K}}

\newcommand{\nFBR}[2]{\nBR{#2,#1}}
\newcommand{\nFBRz}[2]{\nBRz{#2,#1}}

\newcommand{\nFBK}[1]{\nFBR{#1}{K}}
\newcommand{\nFBKz}[1]{\nFBRz{#1}{K}}

\newcommand{\ncBRz}[1]{\nFBRz{\nFloc}{#1}}

\newcommand{\nQq}{\bQ_q}

\newcommand{\nqBRz}[1]{\nFBRz{\nQq}{#1}}
\newcommand{\nqBK}{\nFBK{\nQq}}
\newcommand{\nqBKz}{\nFBKz{\nQq}}

\newcommand{\npBRz}[2]{\nFBRz{\npF{#1}}{#2}}

\newcommand{\npBKz}[1]{\npBRz{#1}{K}}

\newcommand{\nIC}{M}
\newcommand{\nICb}{M}

\newcommand{\nBRz}[1]{\symnBz_{#1}}
\newcommand{\nBz}{\nBRz{\nOK}}
\newcommand{\nBKz}{\nBRz{K}}
\newcommand{\nBkz}{\nBRz{k}}
\newcommand{\nBzc}{\xhat{\symnBz}}

\newcommand{\nboun}[1]{#1^{\diamond}}
\newcommand{\nBboun}[1]{#1^{\,\diamond}}
\newcommand{\nBzboun}[1]{#1^{\,\diamond}}

\newcommand{\nBl}{\nBboun{\nBR{\nOK}}}
\newcommand{\nBzl}{\nBzboun{\nBRz{\nOK}}}

\newcommand{\nMloc}[1]{\nboun{#1}}
\newcommand{\nBllat}{\nboun{\nBlat}}
\newcommand{\rMloc}{\nboun{M}_\infty}

\newcommand{\nMzl}[1]{K\otimes_{\nOK} #1}

\newcommand{\nL}{L}
\newcommand{\nOL}{S}
\newcommand{\nMLzl}[1]{\nL\otimes_{\nOL} #1}
\newcommand{\nBLz}{\nBRz{\nOL}}
\newcommand{\nBLzl}{\nBzboun{\nBLz}}
\newcommand{\rMLloc}{\nboun{M}_{\nL,\infty}}
\newcommand{\roML}{\roM_{\nL}}
\newcommand{\roMLinf}{\ro{M}_{\nL,\infty}}
\newcommand{\rML}{M_{\nL}}

\newcommand{\notim}[1]{\otimes}

\newcommand{\nfrob}{\delta}

\newcommand{\ngamma}{\phi}

\newcommand{\nFgamma}[2]{[#1]^*#2}

\newcommand{\nShift}{\Delta}

\newcommand{\nbackm}{r}

\newcommand{\nlatlift}{U}
\newcommand{\nlatfam}{\mathcal{U}}

\newcommand{\nKo}{K^{\ngamma}}
\newcommand{\nsigmao}{\widetilde{\sigma}}
\newcommand{\tauo}{\widetilde{\tau}}

\newcommand{\nBRo}[1]{\widetilde{\symnB}_{#1}}
\newcommand{\nBRzo}[1]{\widetilde{\symnBz}_{#1}}

\newcommand{\nBo}{\nBRo{\nKo}}
\newcommand{\nBzo}{\nBRzo{\nKo}}
\newcommand{\nkprojo}{\nkproj{\ngamma}}
\newcommand{\nkProj}{\nkprojo^*}

\newcommand{\nKi}[1]{K^{(#1)}}
\newcommand{\nBi}[1]{\nBRo{\nKi{#1}}}
\newcommand{\nBzi}[1]{\nBRzo{\nKi{#1}}}
\newcommand{\nkmap}[1]{\ngamma\nfrob^{#1}}

\newcommand{\nkproj}[1]{[#1]}
\newcommand{\nkfold}[2]{#1^{[#2]}}

\newcommand{\nlatio}[2]{\ngen{\tauo^{#1}#2}}

\newcommand{\nBlat}{T}
\newcommand{\ngen}[1]{\langle{#1}\rangle}
\newcommand{\nlati}[2]{\ngen{\tau^{#1}#2}}
\newcommand{\nBKlat}{T_K}

\newcommand{\nGalZero}{G_0}

\newcommand{\nWCl}{\overline{W}_{\!K}}
\newcommand{\nWInv}{K^w}
\newcommand{\nZ}{\gamma_1^\bZ}
\newcommand{\nZCl}{\Gamma}

\newcommand{\ro}[1]{\xtilde{#1}}
\newcommand{\roM}{\ro{M}}
\newcommand{\roN}{\ro{N}}

\newcommand{\rolat}{\ro{\nBlat}}

\newcommand{\nAR}[1]{A_{#1}}
\newcommand{\nAK}{\nAR{K}}

\newcommand{\raMinf}[1]{\nBlat_{K,#1}}
\newcommand{\raoMinf}{\nBlat}
\newcommand{\rrM}{\xtilde{M}_k}
\newcommand{\rrMinf}{\xtilde{M}_{k,\infty}}

\newcommand{\roMa}[1]{#1_{\infty}}
\newcommand{\rrMa}[1]{#1_{k}}
\newcommand{\rrMinfa}[1]{#1_{k,\infty}}

\newcommand{\nDK}{\cC_K}
\newcommand{\nDKcur}{\cC_{\nKcur}}

\newcommand{\rtauO}[1]{#1[[\tau^{-1}]]}
\newcommand{\rtauF}[1]{#1(\!(\tau^{-1})\!)}

\newcommand{\nER}[1]{#1\langle\!\langle t\rangle\!\rangle}
\newcommand{\nEK}{\nER{K\!}}

\newcommand{\ncurve}{C}
\newcommand{\nDMOK}{\xtilde{E}}

\newcommand{\nev}{\textup{ev}}
\newcommand{\ncoev}{\mu}

\newcommand{\nKyu}{K}
\newcommand{\nKyualt}{{L}}
\newcommand{\nKyual}{\xbar{K}}
\newcommand{\nV}{V}

\newcommand{\nv}{v}
\newcommand{\naV}{T}
\newcommand{\nW}{W}

\newcommand{\nht}{r}

\newcommand{\nYuV}{V}
\newcommand{\nYuW}{J}
\newcommand{\nYuVal}{\xbar{V}}
\newcommand{\nYuWal}{\xbar{J}}
\newcommand{\nphi}{\phi}

\newcommand{\nfnel}{y}
\newcommand{\nyel}{u}
\newcommand{\ntauind}[1]{#1^\natural}

\newcommand{\nDN}{D}
\newcommand{\nOD}{\cO_D}

\newcommand{\naideal}{I}
\newcommand{\naring}[1]{\nBR{#1}/\naideal\nBR{#1}}

\newcommand{\nCentN}{D_\iota}
\newcommand{\neCentN}{\alpha}

\newcommand{\nTate}[1]{T(#1)}
\newcommand{\nTinf}[1]{\nTate{#1_\infty}}

\newcommand{\nlatN}{\nBlat_N}

\newcommand{\navg}[1]{#1_{\textup{av}}}
\newcommand{\navi}[1]{\nBlat_{#1}}

\newcommand{\Zhat}{\smash{\widehat{\bZ}}}

\newcommand{\isochar}{\ensuremath{\sim}}
\newcommand{\shortisochar}{\kern-1pt\ensuremath{\sim}\kern-0.5pt}
\newcommand{\genisosign}[1]{\smash{\raisebox{-0.65ex}{#1}}}
\newcommand{\isosign}{\genisosign{\isochar}} 
\newcommand{\shortisosign}{{\kern-1pt}\xrightarrow{\genisosign{\shortisochar}}} 

\newcommand{\antitr}{\alpha}
\DeclareMathOperator{\End}{End}

\DeclareMathOperator{\Gal}{Gal}
\DeclareMathOperator{\Hom}{Hom}
\DeclareMathOperator{\iHom}{\mathscr{H}\mathrm{om}}
\DeclareMathOperator{\Lie}{Lie}
\DeclareMathOperator{\rank}{rank}

\DeclareMathOperator{\Spec}{Spec}
\DeclareMathOperator{\tr}{tr}
\DeclareMathOperator{\Gr}{Gr}
\DeclareMathOperator{\Proj}{Proj}

\newcommand{\ngcoc}[1]{C(G,\,#1)}

\newcommand{\nArt}{\textup{Art}}

\DeclareMathOperator{\mlim}{\lim\nolimits}

\newcommand{\spr}{\fp}

\begin{document}
\title[Tate modules of isocrystals]{Tate modules of isocrystals and \\ good reduction of Drinfeld modules}
\author{M. Mornev}
\address{ETH Z\"urich, D-MATH, R\"amistrasse 101, 8092 Z\"urich, Switzerland}
\begin{abstract}%
A Drinfeld module has a $\nell$-adic Tate module not only for every finite place $\nell$
of the coefficient ring but also for $\nell = \infty$.
This was discovered by J.-K.~Yu in the form of a representation of the Weil group.

Following an insight of Taelman we construct the $\infty$-adic Tate module
by means of the theory of isocrystals. 
This applies more generally to pure $A$-motives and to pure $F$-isocrystals of $p$-adic cohomology theory.

We demonstrate that a Drinfeld module has good reduction if and only if
its $\infty$-adic Tate module is unramified.
The key to the proof is the theory of Hartl and Pink which 
gives an analytic classification of
vector bundles on the Fargues-Fontaine curve in equal characteristic.%
\end{abstract}%
\maketitle%
\setcounter{tocdepth}{1}%
\tableofcontents%
\section{Introduction}

\subsection{The good reduction criterion}\label{ss:introcrit}

Let $\Fq$ be a finite field of cardinality $q$.
Fix a function field $F$ over $\Fq$ and a place~$\infty$ of~$F$.
Let $\nFinf$ be the completion of $F$ at $\infty$
and $A \subset F$ the ring of elements which are integral outside~$\infty$.
Let $K$ be a field over $\Fq$
and $E$ a Drinfeld $A$-module of rank $r$ over $\Spec K$.
In this setting J.-K.~Yu~\cite{yu} constructed a representation $\rho_\infty\colon W_K \to \nCentN^\times$ where
\begin{itemize}
\item $W_K$ is the Weil group of $K$, see Definition~\ref{dfnweilintro},

\item $\nCentN$ is a central division algebra over $\nFinf$ of Hasse invariant $-\frac{1}{r}$.
\end{itemize}
In this text we shall define the $\infty$-adic Tate module of the motive of $E$.
We shall see that the representation $\rho_\infty$ comes from the action of $W_K$ on this object.
 
Suppose that we are given a valuation $\nval\colon K^\times \twoheadrightarrow \bZ$
with the ring of integers $\nOK$.
Let $\iota_E\colon A \to K$ be the morphism induced by the action of $A$
on the tangent space of $E$ at~$0$.
We assume that $\iota_E(A) \subset \nOK$.
It then makes sense to ask whether $E$ has good reduction over $\Spec\nOK$.
We shall establish the following criterion:

\begin{thm}\label{introcrit}%
The Drinfeld module $E$ has good reduction over $\Spec\nOK$ if and only if
the representation
$\rho_\infty$ is unramified at $\nval$.%
\end{thm}%
The morphism $\iota_E$ is not assumed to be injective
so the Drinfeld module $E$ may have special characteristic.
The field $K$ need not be $\nval$-adically complete and
the residue field of $\nval$ may be infinite.

\begin{cor}%
The image of inertia under $\rho_\infty$ is infinite for every Drinfeld module
which does not have potential good reduction over $\Spec\nOK$.\qed%
\end{cor}

There are good reduction criteria in terms of every
finite place of $F$, that is,
a nonzero prime ideal $\nell \subset A$.
\begin{enumerate}%
\item\label{takacrit}%
Suppose that $\iota_E(\nell)$ does not lie in the maximal ideal of $\nOK$.
Then $E$ has good reduction
if and only if
the $\nell$-adic Tate module $T_\nell E$ is unramified.

\item\label{hhcrit}%
More generally,
$E$ has good reduction
if and only if
the completion of the motive of $E$ at $\nell$
arises from a local shtuka over $\nOK$.
\end{enumerate}%
The first is a theorem of Takahashi~\cite{takahashi}
and the second is a result of Hartl and H\"usken \cite{hartl-huesken}.
Theorem~\ref{introcrit} extends this list to $\nell = \infty$.

In the theory of abelian varieties
the result \eqref{takacrit}
corresponds to the criterion of N\'eron-\hspace{0pt}Ogg-\hspace{0pt}Shafarevich and
\eqref{hhcrit} to the $p$-adic criterion of Grothendieck and de~Jong.
Theorem~\ref{introcrit} seems to lack a counterpart.
For abelian varieties over any base the role of the ring $A$ 
is played by the ring of integers $\bZ$.
There is no $\nellab$-adic Tate module for the place $\nellab = \infty$ of this ring.

\subsection{Tate modules for pure isocrystals}\label{ss:introtate}
The first result of this paper is a construction of Tate modules for pure isocrystals
developed in Section~\ref{sec:tate}.
This construction underlies the definition of the $\infty$-adic Tate module 
and is essential to the proof of Theorem~\ref{introcrit}.

Let us discard the choices of $F$, $K$ and $\nval$ above. 
Let $\nFloc$ be a local field over $\Fq$
with the ring of integers $\nOloc$ and the maximal ideal $\nlocm$.
For each field $\nKc/\Fq$ we set
\begin{equation*}
\ncBRz{\nKc} = \big(\mlim_{n>0} \nKc\otimes_{\Fq} \ncoint/\ncomax^n\big) \otimes_{\ncoint} \nFloc.
\end{equation*}
This is the completed tensor product of $K$ with the discrete topology and
$\nFloc$ with the $\nadic$-adic topology.
We denote this ring
by $\nBRz{\nKc}$ when the choice of $\nFloc$ is clear.

The $q$-Frobenius of $K$ induces an endomorphism
$\nsigma\colon \nBRz{\nKc}\to\nBRz{\nKc}$
by functoriality of the construction $\nKc \mapsto \nBRz{\nKc}$.
Let $\nsigma^*$ be the corresponding functor of extension of scalars.

In the context of this introduction an~\emph{isocrystal} is
a finitely generated $\nBKz$-module $M$ equipped with an isomorphism
$\nsigma^* M \xrightarrow{\isosign} M$.
A morphism of isocrystals is a $\nsigma$-equivariant morphism of the underlying modules.
When the field $K$ is algebraically closed the category of isocrystals is semisimple
\cite[Theorem~2.4.5]{laumon}.
Its simple objects are classified up to isomorphism by a single invariant,
a rational number called the \emph{slope}.
An isocrystal is \emph{pure of slope $\nslope$} if 
it decomposes into a direct sum of simple isocrystals of slope $\nslope$
over an algebraic closure of $K$.

Fix a field $K$ over $\Fq$ and $\nslope \in \bQ$.
The Tate module functor for pure isocrystals of slope $\nslope$
depends on the choice of two objects:
\begin{enumerate}
\renewcommand{\theenumi}{\roman{enumi}}%
\item%
A separable closure $K^s$ of $K$.
Let $\Fqbar$ be the algebraic closure of~$\Fq$ in~$K^s$.

\item%
A pure and simple isocrystal $N$ of slope $\nslope$ over $\nBRz{\Fqbar}$.
\end{enumerate}%
The isocrystal $N$ is unique up to a non-unique isomorphism
and is similar in this respect to the separable closure~$K^s$.

\begin{dfn}%
The \emph{Tate module} of a pure isocrystal $\nIC$ of slope $\nslope$ is the $\Hom$~set
\begin{equation*}
\nTate{\nIC} = \Hom(\nBRz{K^s}\otimes_{\nBRz{\Fqbar}} N,\:\nBRz{K^s\!}\!\otimes_{\nBKz} \nIC).
\end{equation*}
computed in the category of $\nBRz{K^s}$-isocrystals.%
\end{dfn}
The endomorphism ring $\nDN$ of $N$ is a central division algebra over $\nFloc$.
Its Hasse invariant is the residue class of $-\nslope$ in $\bQ/\bZ$.
The Tate module $\nTate{\nIC}$ is naturally a right $\nDN$-vector space. 
Its dimension is finite and is determined by the rank of $\nIC$ over $\nBKz$. 

We shall equip $\nTate{\nIC}$ with an action of a subgroup of $G_K = \Gal(K^s/K)$.
Let us identify the group $\Gal(\Fqbar/\Fq)$ with $\Zhat$
using the \emph{geometric} Frobenius as a generator.
Restriction to $\Fqbar$ defines a homomorphism
$\ord\colon G_K \to \Zhat$.

\begin{dfn}\label{dfnweilintro}%
The \emph{Weil group} $W_K$ is $\ord^{-1}(\bZ)$ with the topology
in which $\ord^{-1}\{0\}$ is an open subgroup with its 
profinite topology.
\end{dfn}

The structure isomorphism $\taulin_N\colon \nsigma^* N \xrightarrow{\isosign} N$ 
induces an additive map
\begin{equation*}
\tau_N\colon N \to N, \quad x \mapsto \taulin_N(1\otimes x)
\end{equation*}
where $1\otimes x$ denotes the element of $\nsigma^* N = \nBRz{\Fqbar}\otimes_{\nBRz{\Fqbar}} N$
corresponding to $x$.
This map is bijective and commutes with every endomorphism of $N$.


\begin{dfn}\label{dfnweilactintro}%
For every $\gamma \in W_K$ and every $f \in \nTate{\nIC}$
we define a morphism
$\gamma(f)\colon \nBRz{K^s}\otimes_{\nBRz{\Fqbar}} N \to \nBRz{K^s}\otimes_{\nBKz} \nIC$
by the formula
\begin{equation*}
1 \otimes x \mapsto \gamma\big(f(1\otimes \tau_N^{\ord\gamma} x)\big). 
\end{equation*}%
Here $G_K$ acts on $\nBRz{K^s}\otimes_{\nBKz} M$ by functoriality
of $\nBRz{K^s}$.
\end{dfn}
Lemma~\ref{actwd} shows that $\gamma(f)$ is a well-defined morphism of isocrystals.
%
The resulting action of $W_K$ on $\nTate{\nIC}$
is $\nDN$-linear by construction. One can prove that this action is continuous
in the $\nadic$-adic topology of $\nTate{\nIC}$.

\begin{rmk}\label{introtate0}%
In the case $\nslope = 0$ the algebra $\nDN$ is equal to $\nFloc$.
The Tate module $\nTate{\nIC}$ is naturally isomorphic to the space of invariants
\begin{equation*}
\{ m \in \nBRz{K^s}\otimes_{\nBKz} \nIC \mid \tau_M(m) = m \}
\end{equation*}
where $\tau_M$ is the additive endomorphism 
induced by the isomorphism $\nsigma^* M \xrightarrow{\isosign} M$.
The Galois group acts on this space via $K^s$,
and the induced action of $W_K$ 
is exactly the one of Definition~\ref{dfnweilactintro}.
This construction of the Tate module is well-known.
\end{rmk}

The main result on the construction $\nIC \mapsto \nTate{\nIC}$ in this paper is Theorem~\ref{tate}.
Under assumption that $W_K$ is dense in $G_K$
this theorem shows that 
the Tate module functor is fully faithful
and describes its essential image.

The isocrystals of this text are related to $F$-isocrystals of $p$-adic cohomology theory.
The construction of Tate modules 
and Theorem~\ref{tate}
extend naturally
to pure $F$-isocrystals over perfect fields.
See Remark~\ref{tatefisoc} for details.%

The Tate module $\nTate{M}$
is a hybrid of the Tate module
introduced by Gross~\cite{gross} for $p$-divisible groups and a construction of J.-K.~Yu~\cite{yu}
for Drinfeld modules.
The underlying $D$-vector space is the one of \cite[\S2]{gross} 
but the Galois action 
is modelled on \cite[Construction~2.5]{yu}.
The key idea is to define
an action of the Weil group rather than the full Galois group.
This is due entirely to J.-K.~Yu.%

Let us explain the role of the Weil group.
Suppose that $\nFloc = \Fq(\!(z)\!)$.
Then there is a canonical pure and simple isocrystal $N$ of slope $\nslope$ over~$\nBRz{\Fq}$,
defined as follows.
Write $\nslope = \tfrac{s}{r}$ in lowest terms,
set
$N = \bigoplus_{i=1}^r \nBRz{\Fq} e_i$
and define an isomorphism $\taulin_N\colon \nsigma^* N \xrightarrow{\isosign} N$
by the formula
\begin{equation*}
\taulin_N(1 \otimes e_1) = e_{2}, \quad \dotsc,
\quad\taulin_N(1 \otimes e_{r-1}) = e_r,
\quad\taulin_N(1 \otimes e_r) = z^s \cdot e_1.
\end{equation*}
%
One can try to define the Tate module by the formula
\begin{equation*}
\nTate{\nIC} = \Hom(\nBRz{K^s}\otimes_{\nBRz{\Fq}} N, \:\nBRz{K^s}\otimes_{\nBKz} \nIC).
\end{equation*}
The full Galois group $G_K$ acts on this by functoriality of~$\nBRz{K^s}$.
However such a definition has a hidden defect:
The isomorphism type of 
the $G_K$-module $\nTate{\nIC}$
depends on the choice of the uniformizer $z\in\nFloc$.
The reason is that a change of $z$ may alter the isomorphism type of $N$. 

Definition~\ref{dfnweilactintro} removes this non-canonicity
at the expense of restricting to the Weil group $W_K$.
This group is not always dense in $G_K$,
and in such cases our Tate modules for slope $\nslope = 0$
give less information than the classical construction of
Remark~\ref{introtate0}.
Nonetheless $W_K$ is dense in $G_K$
in many situations of interest to number theory, such as
the cases when $K$ is finitely generated over $\Fq$
or is local. Whenever the density condition holds
Theorem~\ref{tate} shows that pure isocrystals can be reconstructed
functorially from their Tate modules.

\subsection{Motives and isocrystals}\label{ss:motiso}
To explain 
the relation with the work of J.-K.~Yu \cite{yu}
we return to the setting of Section~\ref{ss:introcrit}
where
$E$ is a Drinfeld $A$-module of rank $r$ over $\Spec K$.

We shall use the motive of $E$ as introduced by Anderson~\cite{anderson}.
Set $\nAK = K\otimes_{\Fq} A$.
Let $\nsigma\colon \nAK \to \nAK$ be the homomorphism $x \otimes a \mapsto x^q \otimes a$
and let $\nsigma^*$ be the functor of extension of scalars along~$\nsigma$.
Consdier the $\nAK$-module
\begin{equation*}
M = \Hom_{\Fq}(E,\,\bG_{a,K})
\end{equation*}
of $\Fq$-linear group scheme morphisms
on which $K$ acts via $\bG_{a,K}$ and $A$ acts via $E$.
The \emph{motive} of $E$ is the pair consisting of $M$ and the morphism
$\nsigma^* M \to M$ induced by the $q$-Frobenius of $\bG_{a,K}$,
see Definition~\ref{dfnmot}.
It is known from the work of Drinfeld~\cite{drinfeld-commrgs} that $M$ is locally free of rank $r$ over $\nAK$.

For every place $\nell$ of $F$ we denote the corresponding local field by $\nFell$
and set
\begin{equation*}
M_\nell = \npBKz{\nell} \otimes_{A_K} M.
\end{equation*}
This may be viewed as the $\nell$-adic completion of the rational motive $F\otimes_A M$.
The morphism $\nsigma^* M \to M$ gives rise to an isomorphism $\nsigma^* M_\nell \xrightarrow{\isosign} M_\nell$.
So we obtain an isocrystal $M_\nell$ for each place~$\nell$ including $\nell = \infty$.

As before we denote by $\iota_E\colon A \to K$ the morphism induced by the action of $A$
on the tangent space of $E$ at~$0$.
It is known that the isocrystal $M_\nell$ is pure for each $\nell \ne \ker(\iota_E)$.
Invoking the construction of Section~\ref{ss:introtate} we get a Tate module $\nTate{M_\nell}$
for every such place $\nell$. In particluar 
we obtain the Tate module $\nTinf{M}$.

If $\nell$ is different from $\ker(\iota_E)$ and $\infty$ then the slope of $M_\nell$ is~$0$.
By Remark~\ref{introtate0}
the action of the Weil group on $\nTate{M_\nell}$ extends canonically
to a continuous action of the full Galois group. 
It follows from \cite[Proposition~1.8.3]{anderson} that
the resulting representation is dual to the rational Tate module
$V_\nell E$: There~is a natural Galois-equivariant isomorphism
\begin{equation*}
\nTate{M_\nell} \xrightarrow{\isosign} \Hom_{\nFell}(V_\nell E, \:\Omega_\nell) 
\end{equation*}
where $\Omega_\nell = \nFell\otimes_F \Omega_{F/\Fq}$
is the $\nell$-adic completion of the space of rational K\"ahler differentials $\Omega_{F/\Fq}$.

Drinfeld proved in \cite{drinfeld-commrgs} that the isocrystal $M_\infty$ is pure of slope~$-\frac{1}{r}$.
So the Tate module $\nTinf{M}$ is naturally a right vector space 
over a central division algebra $\nDN$ of invariant~$\frac{1}{r}$ over $\nFinf$.
Since $M$ is of rank $r$ over $\nAK$
it follows that $\nTinf{M}$ is of dimension~$1$ over $\nDN$.
Theorem~\ref{tateyu} says that, up to technical details, $\nTinf{M}$ is
the representation introduced by J.-K.~Yu~\cite{yu}.

The construction of~\cite{yu} has nothing to do with torsion points of $E$.
Nonetheless we see that $\nTinf{M}$ and the Tate modules at finite places arise in the same way from the motive of $E$.
This puts the construction of J.-K.~Yu on equal footing with 
the other Tate modules.

One must add that the details of Theorem~\ref{tateyu} hide a difference:
The representation of J.-K.~Yu depends on a choice of extra parameters
but the Tate module $\nTinf{M}$ does not.
See Remark~\ref{yudeps} and the discussion before it.

In the same way we get the notion of the $\infty$-adic Tate module for pure Anderson motives
and, more generally, pure $\tau$-sheaves, see Section~\ref{ss:andpure}.
The main theorem of \cite{yu} extends readily to this setting.
A related result of Zywina \cite[Corollary~4.9]{zywina-satotate}
implies the full faithfulness part of an $\infty$-adic Tate conjecture for pure $\tau$-sheaves.
I plan to discuss all this in a forthcoming paper.

\subsection{Tate modules, \texorpdfstring{$\sigma$}{sigma}-bundles and the curve}

To prove the good reduction criterion 
we shall study ramification properties of Tate modules. 
Fix a valuation
$\nval\colon K^\times \twoheadrightarrow \bZ$
and let $\nOK$ be its ring of integers.
In Section~\ref{ss:red} we introduce the notion of good reduction over $R$ for isocrystals.
We show that the Tate module of a pure isocrystal is unramified at $\nval$
if and only if the isocrystal has good reduction 
(Theorem~\ref{tatered}).

The key to the proof of Theorem~\ref{introcrit} is the fact that every
isocrystal of good reduction gives rise to a $\sigma$-bundle of Hartl-Pink
\cite{hartl-pink}.
Let us explain how this works.
%
For expository purposes we consider a valuation $\nval\colon K^\times \to \bR$ which is not necessarily discrete.
We assume that $K$ is complete with respect to $\nval$ and work with the norm $|x| = q^{-\nval(x)}$.
We denote by $\nFloc$ the local field $\rF{\Fq}$.

A $\sigma$-bundle of Hartl and Pink
\cite[Definition~2.1]{hartl-pink}
is a $\nsigma$-equivariant vector bundle on the punctured open unit disk 
\begin{equation*}
\nlUK = \{ z \in K \mid 0 < |z| < 1 \}
\end{equation*}
where $\nsigma\colon \nlUK\to\nlUK$ is the $q$-Frobenius.
The field $\nFloc$ appears in this setting as the subring
of $\sigma$-invariant analytic functions on $\nlUK$.
Hartl and Pink obtained a complete classification of $\sigma$-bundles for algebraically closed
fields~$K$ \cite[Theorem~11.1]{hartl-pink}. 

If $K$ is perfect then $\nsigma$ is an automorphism of $\nlUK$.
Following \cite[Remark 2.2]{hartl-pink}
one can then think of $\sigma$-bundles as vector bundles on the quotient
\begin{equation*}
\nlFF = \nlUK / \nsigma^{\bZ}.
\end{equation*}
Fargues and Fontaine 
\cite{ff-vbrep}, \cite[Sections~6.5 and 7.3]{lacourbe}
constructed $\nlFF$ as a scheme over $\Spec\nFloc$.
This scheme is not of finite type, 
but is separated integral noetherian regular of dimension~$1$ and has genus~$0$
in the sense that $H^1(\nlFF) = 0$. It also has a completeness property.
The classification of vector bundles on $\nlFF$ is similar to Grothendieck's classification of vector bundles on the projective line.
This parallel was observed already by Hartl and Pink and plays a major role in their theory.
\begin{rmk*}%
Fargues and Fontaine introduced a complementary $p$-adic construction: 
there is a punctured open unit disk $\nUK{\bQ_p}$ and its quotient
$\nFF{\bQ_p}$ is a scheme over $\Spec\bQ_p$ which behaves as a complete curve of genus~$0$.
This scheme has a deep connection to $p$-adic Hodge theory.%
\end{rmk*}

By assumption $\nFloc = \rF{\Fq}$ and so $\nFBKz{\nFloc}= \rF{K}$
is the ring of Laurent series with finite principal parts.
Isocrystals of good reduction 
have models over $\rF{\nOK}$. 
All the elements of $\rF{\nOK}$ converge on the punctured open unit disk. 
This allows us to define a base change functor $\nICb \mapsto \nSig{\nICb}$ from
pure isocrystals of good reduction to $\sigma$-bundles.

The main result in this direction is Theorem~\ref{hpzero}:
if the valuation $\nval$ is discrete then the functor
$\nICb\mapsto \nSig{\nICb}$ is fully faithful. 
%
To prove this theorem we show that the Fargues-Fontaine construction of $\nlFF$
yields an empty scheme whenever the valuation $\nval$ is discrete (Remark~\ref{flop}).
In contrast, if $K$ is perfect then $\nlFF$ is not empty and the functor $\nICb \mapsto \nSig{\nICb}$ is not full.%

\subsection{Good reduction of Drinfeld modules}
Let us explain how Theorem~\ref{hpzero} implies the good reduction criterion.
As before we denote by $M$ the motive of the Drinfeld module $E$ over $\Spec K$.
For simplicity we assume that $\nFinf = \rF{\Fq}$.

To speak of the reduction type of $E$ over $\Spec\nOK$
we must assume that $\iota_E(A)$ is contained in $\nOK$.
The motive $M$ then gives rise to a $\nsigma$-bundle $\cM$ on $\nUK{\nFinf}$.
Let us further assume that $E$ has stable reduction.
The results of Drinfeld \cite[\S7]{drinfeld-ell} and Gardeyn \cite[Theorem~1.2]{gardeyn-anmor} imply that
there is a saturated subbundle $\cN \subset \cM$
with the following properties:
\begin{itemize}
\item
$\cM/\cN$ arises from a pure isocrystal of good reduction,

\item
$\cN$ arises from a pure isocrystal $N$ of slope $0$ and good reduction,

\item
$\cN = 0$ if and only if $E$ has good reduction.
\end{itemize}
If $M_\infty$ has good reduction then Theorem~\ref{hpzero} implies that
the inclusion $\cN \hookrightarrow \cM$
arises from a morphism of isocrystals $N \to M_\infty$.
Such a morphism is zero since the slopes of $N$ and $M_\infty$ are different.
Hence $\cN = 0$ and $E$ has good reduction.

 
\subsection{Good reduction of \texorpdfstring{$A$}{A}-motives}
The good reduction criterion of Hartl and H\"usken \cite{hartl-huesken} applies to arbitrary $A$-motives.
The criterion of Takahashi~\cite{takahashi} was extended to $A$-motives by Gardeyn \cite{gardeyn-goodred}.
However the criterion of this paper does not yet work in such generality.

Let $\cM$ be the $\nsigma$-bundle over $\nUK{\nFinf}$ associated to $M$.
The analytic semistablity theorem of Gardeyn \cite[Theorem~4.13]{gardeyn-anstruct}
shows that after a finite separable extension of $K$ the $\nsigma$-bundle $\cM$
acquires a decreasing filtration
\begin{equation*}
\cM = \cM_0 \supset \cM_1 \supset \cdots \supset \cM_{n-1} \supset \cM_n = 0
\end{equation*}
where every subquotient $\cM_i/\cM_{i+1}$ is not zero and arises from an isocrystal $M_i$ of good reduction.
The motive $M$ has potential good reduction if and only if $\cM_1 = 0$.
The filtration on a motive of a Drinfeld module has $\cM_2 = 0$
and $\cM_1 = \cN$. 

Assume that the isocrystal $M_\infty$ is pure and that its Tate module $\nTinf{M}$ is unramified.
Using the results and the methods of this paper one should be able to prove
that the motive $M$ has good reduction
provided that 
one of the isocrystals $M_0$, $M_{n-1}$ enjoys the following properties:
\begin{enumerate}
\item\label{critnewton}
The isocrystal $M_i$ 
is pure,

\item\label{critineq}
The slope of $M_i$ is different from the slope of $M_\infty$ when $n > 0$.
\end{enumerate}
Apart from the case of Drinfeld modules
it is not known if the filtration of Gardeyn's theorem
has the properties \eqref{critnewton}, \eqref{critineq}.
The general case is a subject of current research.

\subsection*{Acknowledgements}
This paper was motivated by an insight of Lenny Taelman.
He predicted that the representation of J.-K.~Yu~\cite{yu} should
arise by a functorial construction 
from the isocrystal $M_\infty$. 

Theorem~\ref{tate} is a result of a discussion with Richard Pink.
He sketched the statement and guessed an inverse to the Tate module functor.

Ambrus P\'al attracted my attention to the work of Watson~\cite{watson}.
Conversations with Ambrus led to the discovery of Example~\ref{bzlmixed}.
Urs Hartl aided me in understanding the details of his paper~\cite{hartl-annals}.

I would like to thank Urs Hartl, Ambrus P\'al, Richard Pink, Lenny Taelman
and the anonymous reviewer of this paper.
My research is supported by ETH Z\"urich Postdoctoral Fellowship Program
and Marie Sk\l{}odowska-Curie Actions COFUND Program.

\section{Preliminaries}

\subsection{Notation and conventions}
We shall use The Stacks Project \cite{stacks} as a source for commutative algebra and the theory of schemes.
References to The Stacks Project
have the form [St:~$wxyz$] where ``$wxyz$'' is a combination of four letters and numbers.
The corresponding item is located at \url{https://stacks.math.columbia.edu/tag/}$wxyz$.

We follow the convention of \cite{stacks} with respect to limits and colimits.
So in this paper a \emph{limit} is what is otherwise known as a \emph{projective} (or \emph{inverse}) limit.
The limits will be denoted by the symbol ``$\lim$''.

All the rings in this text are associative and unital.
The rings are also commutative by default.
When the assumption of commutativity is not made we speak explicitly of associative unital rings.

Given a ring $R$ we denote by $\rF{R}$ the ring $\rO{R}[z^{-1}]$ of Laurent series over $R$ with finite principal parts.

\begin{dfn}%
Let $f\colon R \to S$ be a homomorphism of associative unital rings.
Given a left $R$-module $M$ we denote by $f^\ast M$ the left $S$-module $S\otimes_R M$.
Given a left $S$-module $N$ we denote by $f_\ast N$ the left $R$-module obtained
by restriction of scalars along $f$.%
\end{dfn}

With this choice of notation the functor $f^*$ is left adjoint to $f_*$.

\subsection{Difference rings}
\begin{dfn}[\cite{kedlaya-diff}, Definition~14.1.1]%
A \emph{difference ring} is a pair $(R,\nsigma)$ consisting of a ring $R$
and an endomorphism $\nsigma\colon R \to R$.%
\end{dfn}

\begin{conv}%
We shall work with a single difference ring structure on every suitable ring~$R$.
So we speak of $R$ itself as a difference ring
and reserve the letter $\nsigma$ to denote the corresponding endomorphism.%
\end{conv}

\begin{dfn}%
A \emph{morphism of difference rings} $R \to S$ is a ring homomorphism
which commutes with $\nsigma$.%
\end{dfn}

\begin{dfn}[\cite{kedlaya-diff}, Definition~14.2.1]\label{taupoly}%
Let $R$ be a difference ring.
The \emph{twisted polynomial ring} $R\{\tau\}$
is the ring of polynomials in a formal variable $\tau$
with coefficients in $R$. The multiplication in $R\{\tau\}$ is
subject to the identity
$\tau \cdot x = \nsigma(x) \cdot \tau$ for all $x\in R$.%
\end{dfn}


A left $R\{\tau\}$-module is a pair $(M,\tau)$ where $M$ is an $R$-module and
$\tau\colon M \to M$ an additive map satisfying $\tau(x \cdot m) = \nsigma(x)
\cdot \tau(m)$ for all $(x,m) \in R\times M$. In other words $\tau$ is an
$R$-linear morphism from $M$ to $\nsigma_\ast M$.
\begin{dfn}%
The \emph{structure morphism}
$\taulin\colon \nsigma^* M \to M$
is the adjoint of $\tau\colon M \to \nsigma_* M$.
We write $\taulin_M$ when the dependence 
on $M$ needs an emphasis.%
\end{dfn}

From time to time we shall use an ``adjoint'' description of the category
of left $R\{\tau\}$-modules:

\begin{lem}\label{adjdesc}%
The functor $M \mapsto (M, \taulin_M)$ is an equivalence
of categories of left $R\{\tau\}$-modules and 
pairs $(M,a_M)$ where
\begin{itemize}
\item
$M$ is an $R$-module and
$a_M\colon \nsigma^* M \to M$ is a morphism of $R$-modules,

\item
a morphism of pairs $(M,a_M) \to (N,a_N)$ is a morphism of $R$-modules
$\mu\colon M \to N$ such that $a_N \circ \nsigma^*(\mu) = \mu \circ a_M$.\qed%
\end{itemize}%
\end{lem}

Let $f\colon R \to S$ be a morphism of difference rings.
For each left $R\{\tau\}$-module $M$
the natural map $S\otimes_R M \shortisosign S\{\tau\}\otimes_{R\{\tau\}} M$ is an $S$-linear isomorphism.
The resulting action of $\tau$ on $S\otimes_R M$ is given by
the formula $\tau(s\otimes m) = \nsigma(s)\otimes \tau(m)$.

\begin{lem}%
The pullback $f^*(\taulin_M)$ is naturally isomorphic to $\taulin_{f^*M}$.\qed%
\end{lem}

\subsection{\texorpdfstring{$\nsigma$}{Sigma}-modules}

Fix a difference ring $R$.
We review
the notion of a $\nsigma$-module over~$R$ and discuss some
constructions with $\nsigma$-modules which make sense for
arbitrary difference rings.

\pagebreak
\begin{dfn}%
A left $R\{\tau\}$-module $M$ is called a \emph{$\nsigma$-module} if
\begin{itemize}
\item $M$ is locally free of finite type as an $R$-module,

\item the structure morphism of $M$ is an isomorphism.%
\end{itemize}%
\end{dfn}

These objects are known under many names in the literature.
Our terminology is borrowed from \cite{hartl-annals} where it is used for a restricted class of difference rings.

\begin{dfn}%
The \emph{unit $\nsigma$-module}
$\unitob{R}$ is 
$R$ with the action of $\tau$ given by $\nsigma$.%
\end{dfn}%

\begin{dfn}%
Let $M$, $N$ be $\nsigma$-modules.
The \emph{tensor product $\nsigma$-module} $M \notim{R} N$ is
$M \otimes_{R} N$
with the action of $\tau$ given by the formula
$\tau(m \otimes n) = \tau(m) \otimes \tau(n)$.%
\end{dfn}

\begin{dfn}%
Let $M$ and $N$ be $\nsigma$-modules.
The \emph{Hom $\nsigma$-module} $\iHom(M,N)$ is
$\Hom_R(M,N)$ with the 
action of $\tau$ given by the formula
\begin{equation*}
\tau(f) = \taulin_N \circ \nsigma^\ast(f) \circ (\taulin_M)^{-1}.
\end{equation*}%
The \emph{dual $\nsigma$-module} $\iHom(M, \unitob{R})$ is denoted by $M^*$.%
\end{dfn}

\begin{lem}\label{fdmhomtensor}%
For all $\nsigma$-modules $M$, $N$
the natural map $M^* \notim{R} N \xrightarrow{\isosign} \iHom(M,N)$ is an isomorphism
of $\nsigma$-modules.\qed\end{lem}

\begin{lem}\label{doubledual}%
For every $\nsigma$-module $M$ the natural map
$M \xrightarrow{\isosign} M^{**}$ is an isomorphism of $\nsigma$-modules.\qed%
\end{lem}

Given a left $R\{\tau\}$-module $M$
we set $M^\tau = \{ m \in M \mid \tau m = m \}$. 

\begin{lem}\label{dmhominv}%
For all $\nsigma$-modules $M$ and $N$ we have
\begin{equation*}
\Hom_{R\{\tau\}}(M,N) = \iHom(M,N)^{\tau}
\end{equation*}
as submodules of $\Hom_R(M,N)$.\qed%
\end{lem}

\begin{conv}%
Whenever the difference ring $R$ is clear from the context
we shall write $\unit$ for $\unit_R$ and
$\Hom(M,N)$ for $\Hom_{R\{\tau\}}(M,N)$.%
\end{conv}

\section{Isocrystals}\label{sec:isoc}\label{sec:coef}

Fix a finite field $\Fq$ of cardinality $q$.
From now on all morphisms are assumed to be $\Fq$-linear.

\subsection{Locally constant sheaves}
Let $R$ be an $\Fq$-algebra and $\nsigma\colon R\to R$
the $q$-Frobenius. We shall review
some results on $\nsigma$-modules over $R$.
Unfortunately this does not seem
to appear in the literature in the form we need.

A left $R\{\tau\}$-module $M$
determines a quasi-\hspace{0pt}coherent sheaf $\qcshet{M}$
on the small \'etale site $\etsite{R}$
together with an $\Fq$-linear endomorphism $\tau\colon \qcshet{M} \to\qcshet{M}$.
We set
\begin{equation*}
\shet{M} = \ker\!\big(\qcshet{M} \xrightarrow{\:1-\tau\:} \qcshet{M}\big).
\end{equation*}
For each \'etale $R$-algebra $S$ one has
\begin{equation*}
\qcshet{M}(S) = S\otimes_R M, \quad
\tau\colon s \otimes m \mapsto s^q \otimes \tau(m), \quad (s,m) \in S\times M.
\end{equation*}
As a consequence $\shet{M}(S) = (S\otimes_R M)^\tau$.

\begin{thm}[Katz]\label{wk-equiv}%
The functor $M \mapsto \shet{M}$ is an equivalence of categories
of $\nsigma$-modules over $R$ and locally constant sheaves
of finite-dimensional $\Fq$-vector spaces on $\etsite{R}$.
This equivalence is compatible with arbitrary
base change $R \to S$.%
\end{thm}%
The original result of Katz
\cite[Chapter~4, Proposition~4.1.1]{katz-modforms}
is stated for a normal integral scheme. 
The general case follows from Lang's isogeny theorem for $\GL_{n,\Fq}$
reproved by Katz in
\cite[Corollaire~1.1.2]{katz-sga7}.%
\begin{proof}[Proof of Theorem~\ref{wk-equiv}]%
Drinfeld
\cite[Proposition~2.1]{drinfeld-fmod} constructed a \emph{contravariant}
functor $M \mapsto \Gr M$ from the category of $\nsigma$-modules over $R$ to
the category of finite \'etale $\Fq$-vector space schemes over $\Spec R$
and proved that this is an anti-equivalence.
It follows from the description of $\Gr$ in the proof 
that 
$\shet{M}$ is the \'etale sheaf defined
by the group scheme $\Gr(M^*)$
and that $\Gr$ is compatible with arbitrary base change.
Now the functor $M \mapsto M^*$ is compatible with arbitrary
base change by definition and is an anti-\hspace{0pt}equivalence by Lemma~\ref{doubledual}.
The result follows since every locally constant sheaf of
finite-dimensional $\Fq$-vector spaces on $\etsite{R}$ is represented
by a finite \'etale scheme \stacks{03RV}.%
\end{proof}

Theorem~\ref{wk-equiv} implies that the category of $\sigma$-modules over~$R$ is abelian.
The following result shows that exact sequences of $\sigma$-modules
remain exact in the category left $R\{\tau\}$-modules.

\begin{prp}\label{wk-abelian}%
The subcategory of $\nsigma$-modules over $R$ is closed
under kernels and cokernels in the category of left $R\{\tau\}$-modules.
\end{prp}%
Remarkably, this holds with no restriction on $R$ whatsoever.
\begin{proof}[Proof of Proposition~\ref{wk-abelian}]%
Let $f\colon M \to N$ be a morphism of $\nsigma$-modules
and let $K$, $Q$ be its kernel and cokernel computed in the category
of left $R\{\tau\}$-modules. We need to show that $K$ and $Q$
are locally free $R$-modules of finite type and that
their structure morphisms are isomorphisms.
It is enough to do so \'etale-locally on $\Spec R$.
In view of Theorem~\ref{wk-equiv} we are free to assume
that $f$ comes from a constant morphism of constant sheaves.
In other words $f$ has the form
\begin{equation*}
f_0 \otimes_{\Fq} \textup{id}_\unit \colon U \otimes_{\Fq} \unit \to V \otimes_{\Fq} \unit
\end{equation*}
where $f_0\colon U \to V$
is a morphism of finite-\hspace{0pt}dimensional $\Fq$-vector spaces.
Then $K = \ker f_0 \otimes_{\Fq} \unit$
and $Q = \coker f_0 \otimes_{\Fq} \unit$, and the result follows.%
\end{proof}

\subsection{The coefficient field}

For the duration of Sections~\ref{sec:isoc} -- \ref{sec:isocdvr}
we fix a local field $\ncoef$ over $\Fq$, the \emph{coefficient field}.
We use the following notation:
\begin{itemize}
\item $\ncoint \subset \ncoef$ is the ring of integers,

\item $\ncomax \subset \ncoint$ is the maximal ideal,

\item $\ncons$ is the algebraic closure of $\Fq$ in $\ncoef$,

\item $\nresdeg = [\ncons : \Fq]$.
\end{itemize}
The field $\ncons$ is naturally isomorphic to the residue field $\ncoint/\ncomax$.

We also fix a uniformizer $z\in\ncoef$.
Save for 
Definition~\ref{dfnsimplebun}
this choice does not affect the definitions and the results which follow.

\subsection{Lisse sheaves}

Let $R$ be an $\Fq$-algebra.
Recall that a \emph{lisse $\ncomax$-adic sheaf}
on the small \'etale site $\etsite{R}$
is an inverse system $\cF = \{\cF_n\}_{n\geqslant 1}$ where
\begin{itemize}
\item
each $\cF_n$ is a locally constant
sheaf of finitely generated $\ncoint/\ncomax^n$-modules,

\item the transition maps are $\ncoint$-linear and
induce isomorphisms
\begin{equation*}%
\cF_{n+1}/\ncomax^n\cF_{n+1} \xrightarrow{\isosign} \cF_n.%
\end{equation*}%
\end{itemize}%
The stalk of $\cF$ at a geometric point $\bm\bar{x}$
is the $\ncoint$-module $\mlim_{n>0} (\cF_n)_{\bm\bar{x}}$.
This is automatically finitely generated.
A \emph{morphism} of lisse $\ncomax$-adic sheaves is an $\ncoint$-linear morphism of inverse systems.

Normally one works with lisse sheaves on noetherian schemes only.
This restriction is not relevant in our setting.

\begin{dfn}\label{defetring}%
The ring
$\nFBR{\ncoint}{R}$ is the completion of 
$R \otimes_{\Fq} \ncoint$ at the ideal
$R \otimes_{\Fq} \ncomax$.
We denote by $\nsigma\colon \nFBR{\ncoint}{R} \to \nFBR{\ncoint}{R}$
the completion of the endomorphism 
which acts 
as the $q$-Frobenius on $R$
and
as the identity on $\ncoint$.
\end{dfn}%

If the ring $\ncoint$ is $\Fq[[z]]$ then $\nFBR{\ncoint}{R} = R[[z]]$.
The endomorphism $\sigma$ acts on $R[[z]]$ by the formula
$\sum_{n\geqslant 0} \alpha_n z^n \mapsto \sum_{n\geqslant 0} \alpha_n^q z^n$.

\begin{conv}%
For brevity we shall 
omit the subscript $\ncoint$ in $\nFBR{\ncoint}{R}$
when the choice of $\ncoint$ is clear.
\end{conv}

\begin{thm}\label{wellknown}%
The functor $M \mapsto \{\shetx{M/\ncomax^n M}\}_{n\geqslant 1}$
is an equivalence of categories of
$\nsigma$-modules over $\nBR{R}$
and lisse $\ncomax$-adic sheaves on $\etsite{R}$
all of whose stalks are free. 
This equivalence is compatible with arbitrary base change $R \to S$.%
\end{thm}%

This theorem is related to \cite[Proposition~6.1]{tw}.
Although the theorem is widely known, I have not found a reference. 
The following proof is included for the readers' convenience.
We shall only need the cases where $R$ is a field or a discrete valuation ring,
but the latter case is not much easier than the one of arbitrary $R$.
%
We shall deduce Theorem~\ref{wellknown} from a sequence of lemmas.

\begin{lem}\label{wk-flat}%
Let $n > 0$ be an integer and let $S$ be an $\Fq$-algebra.
Set $A = \ncoint/\ncomax^n$.
For every $S\otimes_{\Fq} A$-module $M$
the following are equivalent:
\begin{enumerate}
\item\label{wkflat-seq}
$M/zM$ is locally free over $S$ and
the sequence $M \xrightarrow{\:z^{n-1}\:} M \xrightarrow{\:z\:} M$
is exact.

\item $M$ is locally free over $S\otimes_{\Fq} A$.%
\end{enumerate}%
\end{lem}%
We shall apply this to $S = R$ and $S = \Fq$.
\begin{proof}[Proof of Lemma~\ref{wk-flat}]%
The $S\otimes_{\Fq} A$-module $S\otimes_{\Fq} A/zA$ admits a free resolution
\begin{equation*}
\cdots \to 
S\otimes_{\Fq} A \xrightarrow{\:z\:}
S\otimes_{\Fq} A \xrightarrow{\:z^{n-1}\:}
S\otimes_{\Fq} A \xrightarrow{\:z\:}
S\otimes_{\Fq} A.
\end{equation*}
Thus the sequence in \eqref{wkflat-seq} is exact if and only if $\Tor_1(S\otimes_{\Fq} A/zA,\,M) = 0$.
Without loss of generality we assume that $M/zM$ is $S$-free.
Then $\Tor_1(S\otimes_{\Fq} A/zA,\,M) = 0$ 
if and only if $M$ is free over $S\otimes_{\Fq} A$ by \stacks{051H}.%
\end{proof}


\begin{lem}\label{wk-equiv-coeff}%
Let $n > 0$ be an integer. Set $A = \ncoint/\ncomax^n$
and define an endomorphism $\nsigma$ of $R\otimes_{\Fq}A$ by
the formula $x\otimes a \mapsto x^q\otimes a$.
Then the functor $M \mapsto \shet{M}$ is an equivalence of categories
of $\nsigma$-modules over $R\otimes_{\Fq}A$ and locally constant sheaves
of finitely generated free $A$-modules on $\etsite{R}$.%
\end{lem}%
\begin{proof}%
Theorem~\ref{wk-equiv} identifies the category of left $(R\otimes_{\Fq}A)\{\tau\}$-modules
which are $\nsigma$-modules over $R$ with the category of locally constant sheaves
of finitely generated $A$-modules on $\etsite{R}$.
We need to prove that such a sheaf has free stalks if and only if
it corresponds to a $\nsigma$-module over $R\otimes_{\Fq}A$.

Let $M$ be a left $(R\otimes_{\Fq}A)\{\tau\}$-module which is a $\nsigma$-module over $R$.
Consider the sequence of left $R\{\tau\}$-modules
\begin{equation}\label{wkcoeff-mod}
M \xrightarrow{\:z^{n-1}\:} M \xrightarrow{\:z\:} M
\end{equation}
and the induced sequence of \'etale sheaves
\begin{equation}\label{wkcoeff-sh}
\shet{M} \xrightarrow{\:z^{n-1}\:} \shet{M} \xrightarrow{\:z\:} \shet{M}.
\end{equation}
We shall prove that the following are equivalent:
\begin{enumerate}
\renewcommand{\theenumi}{\roman{enumi}}%
\item\label{wkcoeff-sigma}
$M$ is a $\nsigma$-module over $R\otimes_{\Fq} A$.

\item\label{wkcoeff-modexact}
The sequence of $R\otimes_{\Fq}A$-modules \eqref{wkcoeff-mod} is exact.

\item\label{wkcoeff-shexact}
The sequence of sheaves \eqref{wkcoeff-sh} is exact.

\item\label{wkcoeff-shfree}
Every stalk of $\shet{M}$ is a free $A$-module.
\end{enumerate}

\eqref{wkcoeff-sigma} $\Leftrightarrow$ \eqref{wkcoeff-modexact}
The structure morphism of $M$ over $R\otimes_{\Fq} A$
coincides with the one over $R$, and so is an isomorphism.
Proposition~\ref{wk-abelian} implies that $M/zM$ is locally free over $R$.
Applying Lemma~\ref{wk-flat} with $S = R$
we conclude that $M$ is locally free over $R\otimes_{\Fq} A$
if and only if the sequence \eqref{wkcoeff-mod} is exact.

\eqref{wkcoeff-modexact} $\Leftrightarrow$ \eqref{wkcoeff-shexact}
Proposition~\ref{wk-abelian} shows that the sequence \eqref{wkcoeff-mod} is exact
in the category of left $R\{\tau\}$-modules
if and only if it is exact in the category of $\nsigma$-modules over $R$.
Hence the claim follows from Theorem~\ref{wk-equiv}.

\eqref{wkcoeff-shexact} $\Leftrightarrow$ \eqref{wkcoeff-shfree}
follows by applying Lemma~\ref{wk-flat} with $S = \Fq$
to the stalks of $\shet{M}$.%
\end{proof}

\begin{lem}\label{complete}%
Let $M$ be a locally free $\nBR{R}$-module of finite type.
\begin{enumerate}
\item
The natural morphism
$M/\fm^n M \xrightarrow{\isosign} (R\otimes_{\Fq} \ncoint/\ncomax^n) \otimes_{\nBR{R}} M$
is an isomorphism for all $n > 0$.

\item
The natural morphism $M \xrightarrow{\isosign} \mlim_{n>0} M/\fm^n M$
is an isomorphism.
In other words $M$ is $\nadic$-adically complete.%
\end{enumerate}%
\end{lem}%
\begin{proof}%
By a direct sum argument we reduce to the case $M = \nBR{R}$.
The ring $\nBR{R}$ is the completion of $R\otimes_{\Fq}\ncoint$ 
at the principal ideal $R\otimes_{\Fq}\ncomax$.
So the claim follows by \stacks{05GG}.%
\end{proof}

\begin{lem}\label{bunsys}%
The functor $M \mapsto \{ M/\fm^n M\}_{n\geqslant 1}$
is an equivalence of categories
of finitely generated locally free $\nBR{R}$-modules 
and inverse systems $\{M_n\}_{n\geqslant 1}$ where
\begin{itemize}
\item
$M_n$ is a finitely generated locally free $R\otimes_{\Fq} \ncoint/\ncomax^n$-module,

\item
the transition maps are $\nBR{R}$-linear and induce isomorphisms
\begin{equation*}
M_{n+1}/\ncomax^n M_{n+1} \xrightarrow{\isosign} M_n.
\end{equation*}

\item
The morphisms of inverse systems are $\nBR{R}$-linear.%
\end{itemize}%
\end{lem}%
\begin{proof}%
Let $\{M_n\}_{n\geqslant 1}$ be such an inverse system.
According to \stacks{0D4B} the limit $M = \mlim_{n>0} M_n$ is a finitely generated locally free $\nBR{R}$-module 
and the natural maps $M/\fm^n M \to M_n$ are isomorphisms.
Conversely, a finitely generated locally free $\nBR{R}$-module
is $\nadic$-adically complete by Lemma~\ref{complete}.%
\end{proof}

\begin{proof}[Proof of Theorem~\ref{wellknown}]%
We view $\nsigma$-modules over $\nBR{R}$ as pairs $(M,a_M)$
where $M$ is a locally free $\nBR{R}$-module of finite type
and $a_M\colon \nsigma^* M \xrightarrow{\isosign} M$ an isomorphism
(see Lemma~\ref{adjdesc}).
The same applies to $\nsigma$-modules over the rings
$R\otimes_{\Fq}\ncoint/\ncomax^n$.
The equivalence of Lemma~\ref{bunsys} is tautologically
compatible with arbitrary base change, and in particular with
the base change $\nsigma\colon \nBR{R} \to \nBR{R}$.
Therefore the functor $M \mapsto \{M/\ncomax^n M\}_{n>0}$
is an equivalence of categories of $\nsigma$-modules over $\nBR{R}$
and inverse systems $\{M_n\}_{n>0}$ where
\begin{itemize}
\item
each $M_n$ is a $\nsigma$-module over $R\otimes_{\Fq}\ncoint/\ncomax^n$,

\item
the transition maps are $\nBR{R}\{\tau\}$-linear
and induce isomorphisms
\begin{equation*}
M_{n+1}/\ncomax^n M_{n+1} \xrightarrow{\isosign} M_n.
\end{equation*}
\end{itemize}
Lemma~\ref{wk-equiv-coeff} implies that
the functor $\{M_n\}_{n>0} \mapsto \{ \shet{(M_n)} \}_{n>0}$ is an equivalence
of this category and the category of lisse $\nadic$-adic sheaves $\cF = \{\cF_n\}_{n>0}$
such that the stalks of every $\cF_n$ are free over $\ncoint/\ncomax^n$.
The stalk of $\cF$ at a geometric point $\bm\bar{x}$
is $\mlim_{n>0} (\cF_n)_{\bm\bar{x}}$. This is free over $\ncoint$
if and only if every term $(\cF_n)_{\bm\bar{x}}$ is free over $\ncoint/\ncomax^n$.
It remains to note that all the intermediate equivalences 
are compatible with arbitrary base change $R\to S$.%
\end{proof}



\subsection{Isocrystals}\label{ss:isocrys}

\begin{dfn}\label{defdieuring}%
We set $\nFBRz{\ncoef}{R} = \nFBR{\ncoint}{R} \otimes_{\ncoint}\ncoef$
and equip this ring with the unique endomorphism $\nsigma$
extending the one of $\nFBR{\ncoint}{R}$.%
\end{dfn}%
\begin{conv}%
For brevity we shall write $\nBRz{R}$ in place of $\nFBRz{\ncoef}{R}$
when the choice of the coefficient field $\ncoef$ is clear.
\end{conv}

If the field $\ncoef$ is $\rF{\Fq}$ then $\nBRz{R} = \rF{R}$ and
$\nsigma\big(\sum_n \alpha_n z^n\big) = \sum_n\alpha_n^q z^n$.

\begin{dfn}\label{dfnisoc}%
A \emph{$\nBRz{R}$-isocrystal} is a $\nsigma$-module over $\nBRz{R}$.%
\end{dfn}%
\begin{conv}%
We omit the prefix ``$\nBRz{R}$-'' when the ring $R$ is clear from the context.%
\end{conv}%

These objects are related to $F$-isocrystals
from Berthelot's theory of rigid cohomology, 
see \cite{kedlaya-notes}.
Informally speaking,
an $F$-isocrystal is an isocrystal with the coefficient field $\ncoef=\bQ_p$
equipped with a $\bQ_p$-linear connection.
Although connections do not appear in the theory of $\nBRz{R}$-isocrystals,
almost all the results of this article have analogs for $F$-isocrystals.

\begin{dfn}%
Let $M$ be a $\nBRz{R}$-module.
Given a subset $X \subset M$ we denote by $\ngen{X}$ the $\nBR{R}$-submodule generated by $X$.%
\end{dfn}

\begin{dfn}\label{defblat}%
Let $M$ be a locally free $\nBRz{R}$-module of finite type.
A \emph{lattice} $\nBlat$ in $M$ is a locally free $\nBR{R}$-submodule
of finite type such that $\nBlat\otimes_{\ncoint}\ncoef = M$.%
\end{dfn}

\begin{lem}\label{latshift}%
Let $M$ be an isocrystal and $\nBlat\subset M$ a lattice.
Then $\nlati{n}{\nBlat}$ is a lattice for every $n\geqslant 0$.%
\end{lem}%
\begin{proof}%
By assumption the structure morphism $\taulin\colon \nsigma^* M \xrightarrow{\isosign} M$ is an isomorphism.
Therefore $\nlati{}{\nBlat} = \taulin(\nsigma^* \nBlat)$ is a lattice.
The claim follows by induction since $\nlati{n+1}{\nBlat} = \nlati{}{\nlati{n}{\nBlat}}$ for all $n\geqslant 0$%
\end{proof}


Recall that $\nresdeg = [\ncoint/\ncomax : \Fq]$ is the degree of the residue
field of $\ncoef$ over $\Fq$.

\begin{dfn}\label{dfnpure}%
An isocrystal $M$ is \emph{pure}
if
there are integers $s$ and $r$, $r > 0$,
and a lattice $\nBlat$ in $M$
such that 
$\nlati{r\nresdeg}{\nBlat} = z^s\nBlat$.
The rational number $\frac{s}{r}$ is called the \emph{slope} of $M$.
\end{dfn}

The notion of purity is central to this work.
In a moment we shall see that the slope of a nonzero pure isocrystal
does not depend on the choice of $s$, $r$ and $\nBlat$.
Note that in Definition~\ref{dfnpure} the exponent of $\tau$ is a multiple of $\nresdeg$.
This ensures that for every field $K$ the slopes of $\nBRz{K}$-isocrystals of rank $1$ are integral.

\begin{exa}%
Every Drinfeld module of rank $r$ gives rise
a pure isocrystal of slope $-\frac{1}{r}$,
see Proposition~\ref{infslope}.
Even though this proposition only covers the case of a field~$R$
the underlying construction of Drinfeld~\cite{drinfeld-commrgs}
works for arbitrary rings.
More generally,
pure Anderson motives of weight $w$
give rise to
pure isocrystals of slope $-w$, see Section~\ref{ss:andpure}.%
\end{exa}

\begin{exa}
An isocrystal which is not pure is called \emph{mixed}.
A natural source of such isocrystals is provided by 
Drinfeld modules of special characteristic.

Let $E$ be a Drinfeld module of characteristic $\spr$ over a field $K$.
We denote its rank by $r$ and its height by $h$. 
Consider the rational $\spr$-adic completion $M_\spr$ of the motive of $E$ as described in~\S\ref{ss:motiso}.
This isocrystal sits in a canonical short exact sequence
\begin{equation*}
0 \to (M_\spr)^0 \to M_\spr \to M_\spr/(M_\spr)^0 \to 0
\end{equation*}
where
\begin{itemize}
\item
the sub-isocrystal $(M_\spr)^0$ is pure of slope~$0$ and rank $r - h$,

\item
the quotient isocrystal $M_\spr / (M_\spr)^0$ is pure of slope $\tfrac{1}{h}$
and rank $h$.
\end{itemize}
Recall that $h$ is an integer from the interval $[1,r]$.
Therefore $M_\spr$ is mixed of slopes $0$ and $\tfrac{1}{h}$
whenever $E$ is not supersingular ($h < r$).

Although I am not aware of a reference for these results,
the existence of such a short exact sequence can be deduced
from the theory of Drinfeld modules of special characteristic.
\end{exa}

\begin{rmk}%
Assuming that $R$ is noetherian
one can show that every pure isocrystal of slope~$0$ arises from a $\nsigma$-module over $\nBR{R}$.
Unfortunately I do not know an easy proof.%
\end{rmk}


\pagebreak
\begin{prp}\label{fdmtensor}\label{fdmhom}%
Let $M$, $N$ be pure isocrystals of slopes $\nslope$, $\nslopealt$.%
\begin{enumerate}%
\item\label{fdmcons-tensor}%
$M\notim{K} N$ is pure of slope $\nslope + \nslopealt$.

\item\label{fdmcons-hom}%
$\iHom(M,N)$ is pure of slope $\nslopealt - \nslope$.%
\end{enumerate}%
\end{prp}%
\begin{proof}%
\eqref{fdmcons-tensor} follows 
by taking
the tensor product of lattices.
Similarly, taking the dual of a lattice in $M$
we deduce that $M^*$ is pure of slope $-\nslope$.
As $\iHom(M,N) = M^*\notim{K} N$ the claim
\eqref{fdmcons-hom} follows from \eqref{fdmcons-tensor}.%
\end{proof}

\begin{lem}\label{injective}%
Let $M$ be a locally free $\nBRz{R}$-module
and $f\colon M \to V$ an $\ncoef$-linear morphism
to an $\ncoef$-vector space.
Suppose that there is a lattice $\nBlat \subset M$
and an $\ncoint$-submodule $U \subset V$
such that $f$ maps $\nBlat$ to $U$ and
induces an injection modulo $\ncomax$.
Then $f$ is injective.%
\end{lem}%
\begin{proof}%
Since $M$ is the union of $z^{-n}\nBlat$, $n\geqslant 0$,
it is enough to prove that $\ker(f|_\nBlat) = 0$.
For every $n \geqslant 0$
the morphism $f$ sends $z^n\nBlat$ to $z^n U$ and induces
an injection modulo~$z$.
Hence $\ker(f|_\nBlat) \subset \bigcap_{n>0} z^n\nBlat$.
The latter submodule is zero
since $\nBlat$ is $\nadic$-adically complete (Lemma~\ref{complete}).%
\end{proof}

\begin{prp}\label{slopeunique}%
Let $M$, $N$ be pure isocrystals of slopes $\nslope$, $\nslopealt$.
If $\nslope < \nslopealt$
then
$\Hom(M,\,N) = 0$.%
\end{prp}%

\begin{proof}%
According to Lemma~\ref{dmhominv} the isocrystal $H = \iHom(M,N)$
has the property that
$H^\tau = \Hom(M,\,N)$.
We shall prove that $H^\tau = 0$.

The isocrystal $H$ is pure of a strictly positive slope by Proposition~\ref{fdmhom}.
So there is an integer $r > 0$ and
a lattice $\nBlat$ in $H$ such that 
$\tau^{r\nresdeg}\nBlat \subset z\nBlat$.
As a consequence the $\ncoef$-linear morphism $1-\tau^{r\nresdeg}\colon H \to H$
sends $\nBlat$ to $\nBlat$ and
reduces to the identity modulo~$z$.
Lemma~\ref{injective} shows that $1 - \tau^{r\nresdeg}$ is injective.
In other words $H$ contains no $\tau^{r\nresdeg}$-invariant elements.
It then contains no $\tau$-invariant elements either.%
\end{proof}

\begin{prp}\label{slopewd}%
The slope of a nonzero pure isocrystal is well-defined.%
\end{prp}%
\begin{proof}%
Let $M$ be such an isocrystal and let $\nslope$, $\nslopealt$ be its slopes.
The identity morphism $M \to M$
is a nonzero morphism from a pure isocrystal of slope $\nslope$
to a pure isocrystal of slope $\nslopealt$.
Proposition~\ref{slopeunique} implies that $\nslope \geqslant \nslopealt$.
Interchanging $\nslope$ and $\nslopealt$ we conclude that $\nslope = \nslopealt$.%
\end{proof}

\begin{prp}\label{slopehom}%
Assume that $R$ is reduced.
Let $M$, $N$ be pure isocrystals of slopes $\nslope$, $\nslopealt$.
If $\nslope \ne \nslopealt$ then $\Hom(M,\,N) = 0$.%
\end{prp}%
\begin{rmk}%
More generally this holds if the nilradical of $R$ is nilpotent.%
\end{rmk}
\begin{proof}[Proof of Proposition~\ref{slopehom}]%
As before we consider the isocrystal
$H = \iHom(M,N)$ and aim to prove that $H^\tau = 0$.
In view of Proposition~\ref{slopeunique} we are free to assume that $\nslope > \nslopealt$.
Then there is a lattice $\nBlat \subset H$ and integers $s < 0$, $r > 0$
such that
$z^s \nBlat$ is generated by $\tau^{r\nresdeg}\nBlat$.
%
The $\ncoef$-linear morphism $1-\tau^{r\nresdeg}$
sends $\nBlat$ to $z^s\nBlat$.
Let $f$ 
be its reduction modulo~$z$.
We shall prove that $f$ is injective.
The result then follows from Lemma~\ref{injective}.

Let $\nsigma\colon R\to R$ be the $q$-Frobenius.
The morphism $f$ is $\nsigma^{r\nresdeg}$-linear since $s < 0$.
Its linearization
$f^\nlin\colon (\nsigma^{r\nresdeg})^*(\nBlat/z\nBlat) \to z^s\nBlat/z^{s+1}\nBlat$
is surjective since $\nlati{r\nresdeg}{\nBlat} = z^s\nBlat$.
The locally free $R$-modules 
$\nBlat/z\nBlat$ and $z^s\nBlat/z^{s+1}\nBlat$ are isomorphic by construction.
In particular their rank functions coincide.
The rank function of $(\nsigma^{r\nresdeg})^*(\nBlat/z\nBlat)$
is the same as the one of $\nBlat/z\nBlat$
since $\nsigma$ is the identity 
on the topological space $\Spec R$.
Hence $f^\nlin$ is an isomorphism.
As $R$ is reduced it follows that $f$ is injective.
\end{proof}

\begin{exa}%
Proposition~\ref{slopehom} can fail for a non-reduced ring $R$.
Let $\ncoef = \rF{\Fq}$ and
$R = \Fq[\varepsilon^{1/q^n}:n\geqslant 0]/(\varepsilon)$.
The isocrystal
$M = \rF{R}$, $\tau \cdot x = z\nsigma(x)$,
is pure of slope $1$ by construction.
The equation $\nsigma(y) = z y$ has a solution
\begin{equation*}
y = \varepsilon^{1/q} z + \varepsilon^{1/q^2} z^2 + \dotsc
\end{equation*}
This defines a nonzero morphism $M \to \unit$, $x \mapsto xy$
to a pure isocrystal of slope~$0$.%
\end{exa}

\begin{prp}\label{pureexist}%
For every pair of integers $(s,r)$ with $r > 0$
there is a pure isocrystal of slope $\frac{s}{r}$
which is free of rank $r\nresdeg$ over $\nBRz{R}$.%
\end{prp}%
\begin{proof}%
Set $M = \nBRz{R}\{\tau\}/I$ where
$I$ is the left ideal generated by the element $\tau^{r\nresdeg} - z^s$.
The elements $1, \tau, \dotsc, \tau^{r\nresdeg-1}$ form a basis of $M$ over $\nBRz{R}$.
Let $\nBlat$ be the free $\nBR{R}$-submodule with this basis. 
By construction $\ngen{\tau^{r\nresdeg}\nBlat} = z^s\nBlat$.
So the structure morphism $\taulin\colon \nsigma^*M \to M$ is surjective.
Since $\nsigma^* M$ and $M$ are free $\nBRz{R}$-modules of the same rank
it follows that $\taulin$ is an isomorphism.
Hence $M$ is a pure isocrystal of slope $\frac{s}{r}$.%
\end{proof}

\begin{rmk}
In general there may be no pure isocrystals of slope $\frac{s}{r}$ and rank~$r$,
see Proposition~\ref{cor-embed}.%
\end{rmk}

\subsection{Rank}

For the moment we assume that \emph{$\Spec R$ is connected.}
\begin{lem}\label{bRidem}%
The spectrum of $\nBRz{R}$ has finitely many connected
components and
$\nsigma$ acts transitively on them.%
\end{lem}

\begin{cor}\label{bRrank}%
Every isocrystal is locally free of constant rank over $\nBRz{R}$.\qed%
\end{cor}

\begin{dfn}%
The \emph{rank} of an isocrystal
is the rank of the underlying $\nBRz{R}$-module.%
\end{dfn}

\begin{proof}[Proof of Lemma~\ref{bRidem}]%
As usual we denote by $\ncons$ the algebraic closure of $\Fq$ in $\ncoef$.
Observe that
$\nBRz{R} = \rF{(R \otimes_{\Fq} \ncons)}$
by construction.
Consider the scheme
\begin{equation*}
X = \Spec \Fq[x]/(x^{q^{\nresdeg}} - x).
\end{equation*}

The set $X(R)$ is naturally an $\Fq$-subalgebra of $R$.
%
We claim that the natural map $X(R)\otimes_{\Fq} \ncons \to X(\nBRz{R})$ is bijective.
%
The natural map $X(R)\otimes_{\Fq}\ncons\to X(R\otimes_{\Fq}\ncons)$
is a bijection since the $q^\nresdeg$-Frobenius of $R\otimes_{\Fq}\ncons$
is the tensor product of the $q^\nresdeg$-Frobenius of $R$ and the identity endomorphism of $\ncons$.
We thus need to show that $X(R\otimes_{\Fq}\ncons) = X(\nBRz{R})$.
Write an element $f \in X(\nBRz{R})$ in the form $f_n + f_0 + f_p$
where $f_n$ has only negative powers of $z$,
$f_p$ only the positive ones and $f_0$ is the constant coefficient.
Since
$f_n^{q^\nresdeg} + f_0^{q^\nresdeg} + f_p^{q^\nresdeg} = f_n + f_0 + f_p$
we conclude that $f_n = 0 = f_p$.
As
$\nBRz{R} = \rF{(R \otimes_{\Fq} \ncons)}$
we conclude that $X(R\otimes_{\Fq}\ncons) = X(\nBRz{R})$.
Therefore the natural map
$X(R)\otimes_{\Fq} \ncons \xrightarrow{\isosign} X(\nBRz{R})$ is a bijection.

All the idempotents of $\nBRz{R}$ belong to the subset $X(\nBRz{R})$.
The endomorphism $\nsigma$ of $\nBRz{R}$ restricts to an endomorphism of
$X(R) \otimes_{\Fq} \ncons$ which acts as the $q$-Frobenius on $X(R)$
and as the identity on $\ncons$. Hence it is enough to study the action
of $\nsigma$ on the idempotents of $X(R) \otimes_{\Fq} \ncons$.

Pick an $R$-algebra $K$ which is a separably closed field.
The natural map $X(R) \to X(K)$ is injective 
since $\Spec R$ is connected and $X$ is finite \'etale over $\Spec\Fq$. 
As $X(K) \cong \ncons$ we conclude that $X(R)$ admits an embedding to $\ncons$. 
In particular $X(R)$ is a finite extension of $\Fq$.

As a consequence 
$X(R) \otimes_{\Fq} \ncons$ decomposes to a direct product of $[X(R):\Fq]$ copies of $\ncons$.
Under this decomposition
the endomorphism $\nsigma$ acts by a cyclic permutation of the coordinates.
Hence $\nsigma$ acts transitively on the set of minimal
idempotents of $X(R) \otimes_{\Fq} \ncons$.%
\end{proof}

\subsection{Morita equivalence}
Fix an isocrystal $N$.
We shall see that under a natural assumption on $N$
the twisting operation $M \mapsto N \otimes M$ admits an inverse.
This will let us to reduce some questions on pure isocrystals
to pure isocrystals of slope~$0$.

The isocrystal $S = \iHom(N,N)$
is an associative unital
$\nBRz{R}$-algebra under composition.
The maps $S \notim{R} S \to S$ and $\unit \to S$ induced by 
the multiplication and the unit of $S$ are morphisms of isocrystals.

\begin{dfn}\label{dfnisomod}%
A \emph{left $S$-isocrystal} is an isocrystal $P$
equipped with a morphism of isocrystals $S \notim{R} P \to P$
which makes $P$ into a left $S$-module. A \emph{morphism} of left 
$S$-isocrystals is an $S$-linear morphism of the underlying isocrystals.%
\end{dfn}

The notion of a right $S$-isocrystal is defined in the same manner.
If $M$ is an isocrystal then $N \notim{R} M$ is a left
$S$-isocrystal in a natural way. If $P$ is a left $S$-isocrystal
then $P^*$ is naturally a right $S$-isocrystal.

\begin{dfn}%
Let $P$ be a left $S$-isocrystal.
We equip $N^* \otimes_S P$ with the structure of a left $\nBRz{R}\{\tau\}$-module
given by the formula $\tau(f \otimes p) = \tau(f) \otimes \tau(p)$.%
\end{dfn}

\begin{prp}\label{generalmorita}%
Assume that $\Spec R$ is connected.
If $N \ne 0$ then the functor $M \mapsto N \notim{R} M$ is an equivalence of categories
of isocrystals and left $S$-isocrystals.
An inverse is given by the functor $P \mapsto N^* \otimes_S P$.%
\end{prp}%
\begin{proof}%
We denote $\antitr\colon N \otimes_{\nBRz{R}} N^* \to S$
the natural isomorphism.
First let us show that
for every left $S$-isocrystal $P$
the left $\nBRz{R}\{\tau\}$-module
$N^* \otimes_S P$ is an isocrystal and the natural map
$\mu\colon N \otimes (N^* \otimes_S P) \to P$, $n \otimes f \otimes p \mapsto \antitr(n \otimes f) \cdot p$
is an isomorphism of left $S$-isocrystals.

The map $\mu$ 
is an $S$-linear isomorphism by construction. It follows 
that $N^* \otimes_S P$ is a locally free $\nBRz{R}$-module of finite type.
Since the isomorphism $\mu$ 
is also $\nBRz{R}\{\tau\}$-linear 
we conclude that the structure morphism
of 
$N^* \otimes_S P$
is an isomorphism.
Therefore $N^* \otimes_S P$ is an isocrystal.%

Next let us show that for every isocrystal $M$ the $\nBRz{R}\{\tau\}$-linear
map
\begin{equation*}
N^* \otimes_S (N \notim{R} M) \to M, \quad f \otimes n \otimes m \mapsto f(n) \cdot m
\end{equation*}
is an isomorphism.
According to Corollary~\ref{bRrank}
the $\nBRz{R}$-module $N$ is locally free
of constant rank greater than zero. As a consequence the functor
$\Hom_{\nBRz{R}}(N, -)$ is faithful.
Morita theory \cite[Chapter 7, Proposition 18.17 (2a), p. 486]{lam}
shows that the map
$\tr\colon N^* \otimes_S N \to \nBRz{R}$, $f \otimes n \mapsto f(n)$,
is an isomorphism and the result follows.%
\end{proof}

\section{Isocrystals over a field}\label{sec:isocfield}

Fix a field $K$ over $\Fq$.
In this section we develop a theory of $\nBKz$-isocrystals.

\subsection{Generalities}


\begin{prp}\label{ab}%
The category of isocrystals is closed under kernels and cokernels
in the category of left $\nBKz\{\tau\}$-modules.%
\end{prp}%
\begin{proof}%
Let $f\colon M \to N$ be a morphism of isocrystals.
Let $P$ and $Q$ be the kernel and the cokernel of $f$ in the category of left $\nBKz\{\tau\}$-modules.
The functor $\nsigma^*$ is exact on $\nBKz$-modules.
So five lemma shows that the structure morphisms of $P$ and $Q$ are isomorphisms.
The ring $\nBKz$ is a finite product of fields.
Hence $P$ and $Q$ are locally free $\nBKz$-modules of finite type.%
%
\end{proof}

%
%
%
The following technical result
will help us to relate the theory of isocrystals with the works of Drinfeld~\cite{drinfeld-commrgs} and J.-K.~Yu~\cite{yu}.
%
\begin{prp}\label{embed}%
An isocrystal $M$ of rank $r > 0$
is pure of slope $\frac{1}{r}$
if and only if
there is a decreasing family $\{\navi{n}\}_{n\in\bZ}$
of lattices in $M$ such that for all $n$
\begin{enumerate}%
\item\label{embed-gen}%
$\navi{n+1} = \nlati{}{\navi{n}}$,

\item\label{embed-shift}%
$\navi{n+\nerank\nresdeg} = z\navi{n}$,%

\item\label{embed-dim}%
$\dim_K \navi{n}/\navi{n+1} = 1$.
\end{enumerate}%
\end{prp}%
%
To prove this we need a lemma:
\begin{lem}\label{nilpcrit}%
Let $K\{\tau\}$ be the twisted polynomial ring defined by the $q$-Frobenius of $K$.
Let $V$ be a left $K\{\tau\}$-module of dimension $n < \infty$ over $K$.
If $\tau$ is nilpotent on $V$ then $V$ is $\tau^n$-torsion.
%
\end{lem}%
\begin{proof}%
We argue by induction on $n$.
If $n > 0$ then there is a nonzero $v \in V$ killed by $\tau$.
The subspace $V' = K \cdot v$ is a $K\{\tau\}$-submodule of $V$.
By induction $V/V'$ is killed by $\tau^{n-1}$ and so $V$ is killed by $\tau^n$.%
\end{proof}

\begin{proof}[Proof of Proposition~\ref{embed}]%
%
An isocrystal $M$ containing a family of lattices with the properties
\eqref{embed-gen}, \eqref{embed-shift} is 
pure of slope~$\frac{1}{r}$ since
$\langle \tau^{\nerank\nresdeg} \nBlat_0 \rangle = \nBlat_{\nerank\nresdeg} = z^\nerank \nBlat_0$.
Conversely, suppose that $M$
is pure of rank $\nerank>0$ and slope~$\frac{1}{\nerank}$.
By assumption there are integers $s$ and $\neshift$, $\neshift > 0$,
and a lattice $\nBlat$ in $M$ such that $\frac{s}{\neshift} = \frac{1}{\nerank}$
and $z^s \nBlat = \nlati{\neshift\nresdeg}{\nBlat}$.

Let $\navg{\nBlat}$ be the $\nBK$-submodule of $M$ generated by
$\nBlat, \tau\nBlat, \dotsc, \tau^{\neshift\nresdeg-1} \nBlat$. 
The ring $\nBK$ is a finite product of discrete valuation rings
and $\nBKz$ is the product of the corresponding fraction fields.
Hence $\navg{\nBlat}$ is a locally free $\nBK$-module of finite type.
By construction $z^s\navg{\nBlat} = \nlati{\neshift\nresdeg}{\navg{\nBlat}}$.
So we are free to replace $\nBlat$ with $\navg{\nBlat}$.
Now $\nBlat$ is closed under the action of $\tau$.
%
For each $n \in \bZ$
we set $\navi{n} = z^{sa}\nlati{b}{\nBlat}$
where $a \in \bZ$ and $b \in \{0,\dotsc,\neshift\nresdeg-1\}$ 
are unique integers such that $n = a \neshift\nresdeg + b$.
The family of lattices $\{\navi{n}\}_{n\in\bZ}$
has the property (1) 
by construction.
Furthermore $\navi{n+\neshift\nresdeg} = z^s\navi{n}$
for all $n \in \bZ$.

As $M$ is a free $\nBKz$-module of rank $\nerank$ it follows that each $\navi{n}$ is 
a free $\nBK$-module of rank $\nerank$. So the quotient
$\navi{n}/\navi{n+\neshift\nresdeg} = \navi{n}/z^s\navi{n}$
is a free module of rank $\nerank$ over the ring $K\otimes_{\Fq} \ncoint/\ncomax^s$.
Since $\nerank = \frac{\neshift}{s}$ we conclude that
\begin{equation*}
\dim_K \navi{n}/\navi{n+\neshift\nresdeg} = \nerank \cdot s \nresdeg = \neshift\nresdeg.
\end{equation*}
Every subquotient $\navi{n}/\navi{n+1}$ is nonzero.
Indeed if $\navi{n}/\navi{n+1} = 0$ 
then $\navi{n}$ is generated by $\tau\navi{n}$
and so $M$ is a pure isocrystal of slope~$0$, a contradiction to Proposition~\ref{slopewd}.
Now the equality
\begin{equation*}
\sum_{i=0}^{\neshift\nresdeg-1} \dim_K \navi{n+i}/\navi{n+i+1} =
\dim_K \navi{n}/\navi{n+\neshift\nresdeg} = \neshift\nresdeg
\end{equation*}
implies that $\dim_K \navi{n}/\navi{n+1} = 1$.
We thus get (3). 

Consider the quotient $V_n = \navi{n}/z\navi{n}$.
This is a left $K\{\tau\}$-module of dimension $\nerank\nresdeg$ over $K$.
By assumption $\tau^{\neshift\nresdeg}$ acts by zero on $V_n$.
Lemma~\ref{nilpcrit} shows that the endomorphism $\tau^{\nerank\nresdeg}\colon V_n \to V_n$ is zero.
Hence $\navi{n+\nerank\nresdeg} \subset z\navi{n}$.
We thus have a surjection $\navi{n}/\navi{n+\nerank\nresdeg} \to \navi{n}/z\navi{n}$.
Since its source and target are of dimension $\nerank\nresdeg$
it follows that $\navi{n+\nerank\nresdeg} = z\navi{n}$.
So we obtain (2).
\end{proof}

\subsection{The field of constants}
\label{subsec:nocoeff}

Recall that the field of constants $\ncons$
is the algebraic closure of $\Fq$ in the coefficient field $\ncoef$.

\begin{prp}\label{cor-embed}%
Suppose that there exists a $\nBKz$-isocrystal of rank $\nerank > 0$ and slope~$\frac{1}{\nerank}$.
Then the field of constants $\ncons$ admits an embedding to $K$.
\end{prp}%
%
So there are no pure $\nBRz{\Fq}$-isocrystals of rank~$1$ and slope~$1$ whenever $\nresdeg > 1$.
\begin{proof}[Proof of Proposition~\ref{cor-embed}]%
Let $M$ be an isocrystal of rank $\nerank$ and slope $\frac{1}{\nerank}$.
Consider a family of lattices $\{\navi{n}\}_{n\in\bZ}$ in $M$
produced by Proposition~\ref{embed}.
Every quotient $\navi{n}/\navi{n+1}$ is an $\nBK$-module and so carries
a $K$-linear action of $\ncons$.
This determines a homomorphism $\ncons\to K$
since $\dim_K \navi{n}/\navi{n+1} = 1$.%
\end{proof}

Let us assume that there is 
a morphism $\ngamma\colon\ncons\to K$.
We denote by $\nKo$ the field $K$ viewed as a $\ncons$-algebra
via $\ngamma$.
The field $\ncons$ is finite, so we have a theory of isocrystals with $\ncons$ in place of $\Fq$.
We shall use the tilde accent to differentiate the objects of this theory from the usual $\Fq$-theory.
We have difference rings
$\nBo = \rO{\nKo}$, $\nBzo = \rF{\nKo}$ with endomorphisms $\nsigmao$,
the twisted polynomial ring $\nBzo\{\tauo\}$,
the notions of purity and slope.
Note the following facts:
\begin{itemize}%
\item%
The endomorphism $\nsigmao$ acts as the $q^\nresdeg$-Frobenius on the subfield $K \subset \nBzo$.

\item%
The parameter $\nresdeg$ is equal to $1$ in the theory of $\nBzo$-isocrystals.%
\end{itemize}%

We shall prove that the categories of $\nBKz$-isocrystals and $\nBzo$-isocrystals are equivalent,
and that this equivalence respects purity and slopes.
So computations with $\nBKz$-isocrystals can be done under extra assumption
$\nresdeg = 1$ whenever the field $K$ contains a copy of $\ncons$.

In the following we view left $\nBKz\{\tau\}$-modules
as pairs $(M,a)$ where $M$ is a $\nBKz$-module
and $a\colon\nsigma^* M \to M$ is a morphism (see Lemma~\ref{adjdesc}).
The same applies to left $\nBzo\{\tauo\}$-modules.

The morphism $K\otimes_{\Fq}\ncoint \to \nKo\otimes_{\ncons} \ncoint$
induces a morphism $\nkprojo\colon \nBK \to \nBo$ by completion.
This extends uniquely to a morphism $\nkprojo\colon \nBKz\to\nBzo$.
For each left $\nBKz\{\tau\}$-module $(M,a)$ we denote by $\nkfold{a}{\nresdeg}$ the composition
\begin{equation*}
(\nsigma^{\nresdeg})^* M \xrightarrow{\:(\nsigma^{\nresdeg-1})^* a\:}
(\nsigma^{\nresdeg-1})^* M \to \cdots \to \nsigma^* M \xrightarrow{\:a\:} M.
\end{equation*}
Thanks to the identity $\nkprojo \circ \nsigma^\nresdeg = \nsigmao \circ \nkprojo$
we can view $\nkprojo^* \nkfold{a}{\nresdeg}$ as a morphism
from $\nsigmao^* \nkprojo^* M$ to $\nkprojo^* M$.
We set $\nkProj(M,a) = (\nkprojo^*M,\,\nkprojo^*\nkfold{a}{\nresdeg})$.

\begin{prp}\label{traditio}%
The functor $(M,a) \mapsto \nkProj(M,a)$
is an equivalence of categories of $\nBKz$-isocrystals and $\nBzo$-isocrystals.
Moreover a $\nBKz$-isocrystal $(M,a)$ is pure of slope $\nslope$
if and only if $\nkProj(M,a)$ is pure of the same slope.%
\end{prp}

We split the proof into a sequence of lemmas.
Let $\nfrob\colon \ncons\to\ncons$ be the $q$-Frobenius.
The set of all $\Fq$-algebra morphisms $\ncons\to K$
is
\begin{equation*}
\{ \ngamma, \:\ngamma\nfrob,\:\dotsc,\:\ngamma\nfrob^{\nresdeg-1}\}.
\end{equation*}
The choice of $\ngamma$ gives us a bijection between this set
and the abelian group $\bZ/\nresdeg$.
For each $i\in\bZ/\nresdeg$
we denote by $\nKi{i}$ the field $K$ with a $\ncons$-algebra structure $\nkmap{i}$
and
$\nkproj{i}\colon \nBK\to \nBi{i}$
the completion of the natural morphism
$K\otimes_{\Fq}\ncoint \to \nKi{i}\otimes_{\ncons} \ncoint$.
By definition $\nkproj{0} = \nkprojo$.

\begin{lem}\label{trad-iso}%
The morphisms $\nkproj{i}$ induce an isomorphism $\nBK \xrightarrow{\isosign} \prod_{i\in\bZ/\nresdeg} \nBi{i}$.%
\end{lem}%
\begin{proof}%
We have an isomorphism of $\ncons$-algebras
\begin{equation*}
K\otimes_{\Fq}\ncons \xrightarrow{\isosign} \prod_i \nKi{i}, \quad \alpha \otimes_{\Fq} x \mapsto
\big(\alpha\cdot\nkmap{i}(x)\big)_i
\end{equation*}
Taking the tensor product over $\ncons$ with $\ncoint$ on the right we obtain an isomorphism
\begin{equation*}
K\otimes_{\Fq}\ncoint \xrightarrow{\isosign}
\prod_i \nKi{i}\otimes_{\ncons} \ncoint, \quad
\alpha \otimes_{\Fq} x \mapsto
\big(\alpha \otimes_{\ncons} x\big)_i
\end{equation*}
Its $\nadic$-adic completion is the desired isomorphism.%
\end{proof}

Let $\nShift\colon \nBi{i} \to \nBi{i+1}$
be the completion of the morphism which acts as the $q$-Frobenius on $K$
and as the identity on $\ncoint$.

\begin{lem}\label{trad-com}%
For each $i\in\bZ/\nresdeg$ we have
$\nkproj{i} \circ \nsigma = \nShift \circ \nkproj{i-1}$.\qed%
\end{lem}

Let $\nbackm\colon \nBi{0} \to \nBK$ be the unique morphism
such that $\nkproj{i} \circ \nbackm = \nShift^i$ for every
$i \in \bZ/\nresdeg\bZ$.

\begin{lem}\label{trad-comp}%
We have
\begin{equation*}
\nkproj{i} \circ (\nsigma \circ \nbackm) = \left\{\!\!\begin{array}{ll}%
\nShift^i, & i \in \{1,\dotsc,\nresdeg-1\}, \\
\nsigmao, & i = 0. \qed%
\end{array}\right.%
\end{equation*}%
\end{lem}

Inverting $z$ we obtain morphisms
$\nkproj{i}\colon \nBKz\to\nBzi{i}$,
$\nShift\colon\nBzi{i} \to \nBzi{i+1}$ and
$\nbackm\colon \nBzi{0} \to \nBKz$.
Lemma~\ref{trad-iso} implies that the morphisms $\nkproj{i}$
induce an isomorphism $\nBKz\xrightarrow{\isosign}\prod_i\nBzi{i}$.

Let $(M,a)$ be a left $\nBzi{0}\{\tauo\}$-module.
Observe that $\nkproj{i}^*(\nbackm^* M) = (\nShift^i)^* M$
by definition of $\nbackm$.
We denote by $\nbackm^*(M,a)$ the left $\nBKz\{\tau\}$-module
whose underlying $\nBKz$-module is $\nbackm^* M$
and the structure morphism 
is the unique morphism
which restricts to the identity on $\nBzi{i}$, $i \in \{1,\dotsc,\nresdeg-1\}$,
and to the morphism $a$ on $\nBzi{0}$.
This construction makes sense thanks to Lemma~\ref{trad-comp}.

Let $(M,a)$ be a left $\nBKz\{\tau\}$-module.
For each $i\in\bZ/\nresdeg$ we set
$a_i = \nkproj{i}^*(a)$. Lemma~\ref{trad-com}
implies that this is a morphism
from $\nShift^*(\nkproj{i-1}^* M)$ to $\nkproj{i}^* M$.
For every $n \geqslant 0$
we denote by $a_i^{[n]}\colon (\nShift^n)^* \nkproj{i-n}^* M \to \nkproj{i}^* M$
the composition which is defined inductively
as follows:
\begin{equation*}
a_i^{[0]} = \textup{id}_{\nkproj{i}^*M}, \quad
a_i^{[n+1]} = a_i \circ \nShift^*(a_{i-1}^{[n]}).
\end{equation*}

\begin{lem}\label{trad-fund}%
$\nkProj(M,a) = (\nkproj{0}^* M, \,a_0^{[\nresdeg]})$.\qed%
\end{lem}

We denote by $\mu\colon \nbackm^* \nkProj(M,a) \to (M,a)$ the unique morphism
which restricts to $a_i^{[i]}$ on every $\nBzi{i}$.

\begin{lem}\label{trad-nat}%
The map $\mu$ is a natural morphism of left $\nBKz\{\tau\}$-modules.\qed%
\end{lem}

\begin{lem}\label{trad-natiso}%
The map $\mu$ is an isomorphism for every $\nBKz$-isocrystal $(M,a)$.\qed%
\end{lem}

\begin{lem}\label{trad-equiv}%
The functor $(M,a) \mapsto \nkProj(M,a)$ is an equivalence of
categories of $\nBKz$-isocrystals and $\nBzo$-isocrystals.%
\end{lem}%
\begin{proof}%
The functor $\nbackm^*$ transforms $\nBzo$-isocrystals to $\nBKz$-isocrystals.
Lemma~\ref{trad-fund} implies that
the composition of functors $\nkProj \circ \nbackm^*$ is isomorphic to the identity.
The natural transformation $\mu\colon \nbackm^* \circ \nkProj \to \textrm{id}$
is an isomorphism on the category of $\nBKz$-isocrystals.
Whence the result.%
\end{proof}

\begin{lem}\label{trad-latred}%
Let $(M,a)$ be a $\nBKz$-isocrystal and $\nBlat \subset M$ a lattice.
Then for every $n \geqslant 0$ we have
$\nkProj\nlati{n\nresdeg}{\nBlat} = \nlatio{n}{\nkProj\nBlat}$.\qed%
\end{lem}


\begin{lem}\label{trad-lat}%
For every $\nBKz$-isocrystal $(M,a)$
there exists a family $\nlatfam$ of lattices in $M$ 
with the following properties:
\begin{enumerate}
\item\label{trad-lat-lift}%
For each lattice $\nBlat \subset \nkProj M$
there is a unique $\nlatlift \in \nlatfam$ such that
$\nkProj\nlatlift = \nBlat$.

\item\label{trad-lat-tau}%
For each $\nlatlift \in \nlatfam$ the lattice $\nlati{\nresdeg}{\nlatlift}$
belongs to $\nlatfam$.

\item\label{trad-lat-aut}%
For each automorphism $f\colon M \to M$
and each $\nlatlift \in \nlatfam$ the lattice $f(\nlatlift)$
belongs to $\nlatfam$.
\end{enumerate}
\end{lem}%
\begin{proof}%
The natural transformation $\mu$ is an isomorphism on the category
of $\nBKz$-isocrystals. So for every lattice $\nBlat$ in $\nkProj M$
the $\nBK$-module $\mu(\nbackm^* \nBlat)$ is a lattice in $M$.
Consider the family $\nlatfam = \{ \mu(\nbackm^* \nBlat) \}$ where $\nBlat$ runs
over the lattices in $\nkProj M$.
This family has property \eqref{trad-lat-lift} by construction.
Property~\eqref{trad-lat-tau} follows from Lemma~\ref{trad-latred}
and property~\eqref{trad-lat-aut} is a consequence of the fact
that $\mu$ is a natural transformation.%
\end{proof}

\begin{proof}[Proof of Proposition~\ref{traditio}]
In view of Lemma~\ref{trad-equiv} it remains to show that an isocrystal $(M,a)$ is pure of slope $\nslope$
if and only if $\nkProj(M,a)$ is pure of slope $\nslope$.

If $(M,a)$ is pure of slope $\nslope$ then Lemma~\ref{trad-latred}
implies that $\nFgamma{\ngamma}(M,a)$ is pure of slope $\nslope$.
To prove the converse
write $\nslope = \frac{s}{r}$ with $r > 0$.
Suppose that there is a lattice $\nBlat$ in $\nkProj M$
such that $z^s\nBlat = \nlatio{r}{\nBlat}$.
Let $\nlatfam$ be the family of lattices introduced in Lemma~\ref{trad-lat}
and let $\nlatlift\in\nlatfam$ be the unique lattice such that
$\nkProj\nlatlift = \nBlat$.
The lattice $\nlati{r\nresdeg}{\nlatlift}$
is an element of $\nlatfam$ by \eqref{trad-lat-tau}.
The lattice $z^s\nlatlift$ belongs to $\nlatfam$ by \eqref{trad-lat-aut}.
We have $\nFgamma{\ngamma}\nlati{r\nresdeg}{\nlatlift} = \nFgamma{\ngamma}z^s\nlatlift$ by construction.
The unicity part of \eqref{trad-lat-lift} implies that $\nlati{r\nresdeg}{\nlatlift} = z^s \nlatlift$.
Hence the isocrystal $(M,a)$ is pure of slope $\frac{s}{r}$.%
\end{proof}

\subsection{Pure isocrystals}\label{sec:dmfield}


\begin{prp}\label{subquot}%
Every subquotient of a pure isocrystal is pure of the same slope.%
\end{prp}%
\begin{proof}%
Let $M$ be a pure isocrystal and $f\colon M \to N$ a surjection.
The ring $\nBK$ is a finite product of discrete valuation rings and $\nBKz$ is the product of the corresponding
fraction fields. Hence the image of any lattice in $M$ under $f$ is a lattice in $N$,
and the claim follows.

If $f\colon N \to M$ is injective
then $f^*\colon M^* \to N^*$ is an epimorphism
as the functor $M \mapsto M^*$ is an antiequivalence by Lemma~\ref{doubledual}.
Proposition~\ref{ab} implies that $f^*$ is surjective on
the level of $\nBKz$-modules.
So $N^*$ is pure of the same slope as $M^*$ by the argument above.%
\end{proof}

\begin{prp}\label{fdmpure0}%
An isocrystal is pure of slope~$0$ if and only if
it arises from a $\nsigma$-module over $\nBK$.%
\end{prp}%
\begin{proof}%
Let $M$ be a pure isocrystal of slope~$0$.
By definition there is a lattice $\nBlat$ in $M$ and an integer $r > 0$ such that
$\nBlat = \nlati{r\nresdeg}{\nBlat}$.
Consider the $\nBK$-module $\navg{\nBlat} = \ngen{\nBlat \cup \tau\nBlat \cup \dotsc \cup \tau^{r\nresdeg-1}\nBlat}$.
By construction $\navg{\nBlat} = \nlati{}{\navg{\nBlat}}$ and $M = \navg{\nBlat}[z^{-1}]$.
The ring $\nBK$ is a finite product of discrete valuation rings.
Hence $\navg{\nBlat}$ is locally free of finite type and the claim follows.%
\end{proof}

\begin{prp}\label{purehom}%
Assume that $K$ is separably closed.
For all pure isocrystals $M$ and $N$ of the same slope we have
\begin{equation*}
\dim_F \Hom(M,\,N) = \rank M \cdot \rank N.%
\end{equation*}%
\end{prp}%
\begin{proof}%
The isocrystal $H = \iHom(M,\,N)$ is pure of slope~$0$. 
By Proposition~\ref{fdmpure0} it arises from a $\nsigma$-module over $\nBK$.
As $K$ is separably closed Theorem~\ref{wellknown}
implies that every such $\nsigma$-module is a direct sum of unit objects.
Hence the isocrystal $H$ is itself a direct sum of unit objects.
Now $H^\tau = \Hom(M,\,N)$ by Lemma~\ref{dmhominv} and the result follows.%
\end{proof}

\begin{prp}[Laumon~\cite{laumon}]\label{simple}%
Assume that $K$ is separably closed.
\begin{enumerate}
\item\label{simple-dfn}%
For every rational number $\nslope$ there is a simple
isocrystal $M_\nslope$ which is pure of slope $\nslope$.
This isocrystal is unique up to isomorphism.

\item\label{simple-rank}%
The rank of $M_\nslope$ is equal to the denominator
of $\nslope$ written in lowest terms.


\item\label{simple-end}%
The ring $\End M_\nslope$
is a central division algebra over $\ncoef$.
Its Hasse invariant is the residue class of $-\nslope$ in $\bQ/\bZ$.%
\end{enumerate}%
\end{prp}%
\begin{proof}
First we prove the unicity part of \eqref{simple-dfn}.
Let $M$ be a simple isocrystal which is pure of slope $\nslope$.
Proposition~\ref{purehom} shows that there is a nonzero morphism $M_\nslope \to M$.
Since $M_\nslope$ and $M$ are simple
Proposition~\ref{ab} implies that this is an isomorphism.

Let us establish the existence of $M_\nslope$.
The field $K$ contains a copy of $\ncons$, 
so in view of Proposition~\ref{traditio}
it is enough to deal with the case $\ncoef = \rF{\Fq}$.
This case was treated by Laumon \cite[Appendix~B]{laumon}.
Although Laumon restricts to the case when $K$ is algebraic over $\Fq$
we shall rely on the arguments which work without change for arbitrary $K$.

Write $\nslope = \frac{s}{r}$ in lowest terms and let $M_\nslope$
be the left $\nBKz\{\tau\}$-module
generated by an element $m$ with a relation $\tau^r m = z^s m$.
This is clearly an isocrystal.
Lemma~B.1.10 of \cite{laumon} shows that $M_\nslope$ is simple.
Let $\nBlat \subset M_\nslope$ be the free $\nBK$-submodule on the basis
$m, \tau m, \dotsc, \tau^{r-1} m$. 
Then $\tau^r \nBlat$ generates $z^s\nBlat$ and so
$M_\nslope$ is pure of slope $\nslope$. We thus get the existence
part of \eqref{simple-dfn}. Property \eqref{simple-rank} holds by construction.
Property~\eqref{simple-end} is the second claim of
Theorem~2.4.5 in \cite[Section~2.4]{laumon}.
\end{proof}

\begin{rmk}%
It is sometimes claimed in the literature that
the classification theorem of Laumon \cite[Theorem~2.4.5]{laumon} works for every separably closed field $K$,
see \cite[Proposition~5.1.3]{taelman-art} or \cite[Proposition~4.5]{zywina-satotate}.
The following counterexample shows that this is false: If $K$ is not perfect
then there are mixed isocrystals which do not split into a direct sum of pure ones.

Assume that $\ncoef = \rF{\Fq}$. Then $\nBKz = \rF{K}$.
Let $\alpha \in K$ be an element which has no $q$-th root.
Consider the isocrystal $M$ with a $\rF{K}$-basis $e_0$, $e_1$ on which $\tau$
acts as follows:
\begin{equation*}
\tau(e_0) = e_0, \quad \tau(e_1) = z e_1 + \alpha e_0.
\end{equation*}
The subisocrystal $M_0 = \rF{K}\cdot e_0$ is pure of slope~$0$
and $M/M_0$ is pure of slope~$1$. Proposition~\ref{subquot} implies
that $M$ is not pure. We claim that $M_0$ is not a direct summand of $M$.

Let $s\colon M \to M_0$ be a splitting of the inclusion $M_0 \hookrightarrow M_1$.
Then $s(e_0) = e_0$ and $s(e_1) = x e_0$ where $x \in \rF{K}$
is a solution of the equation $\nsigma(x) = z x + \alpha$.
The coefficients of $x$ at negative powers of $z$ are zero
and the coefficient at $z^0$ is a $q$-th root of $\alpha$
which does not exist by assumption.
\end{rmk}

\begin{rmk}%
Isocrystals over any field $K$ carry a unique increasing filtration with pure subquotients of strictly increasing slopes,
the \emph{Harder-Narasimhan filtration},
see \cite[Proposition~1.5.10]{hartl-annals}.
If the field $K$ is perfect then this filtration splits
giving rise to a canonical decomposition into pure sub-isocrystals.
\end{rmk}

\begin{prp}\label{puredecomp}
If $K$ is separably closed then every pure isocrystal
of slope $\nslope$ decomposes into a direct sum of $M_\nslope$.%
\end{prp}%
\begin{proof}%
Suppose that $M\ne 0$.
By Proposition~\ref{purehom} there is a nonzero morphism $f\colon M_\nslope\to M$.
As $M_\nslope$ is simple Proposition~\ref{ab} shows that $f$ is injective.
Let $M'$ be the cokernel of $f$ and consider the sequence
\begin{equation*}
0 \to \Hom(M',\,M_\nslope)
\to \Hom(M,\,M_\nslope)
\xrightarrow{-\circ f} \Hom(M_\nslope,\,M_\nslope)
\to 0.
\end{equation*}
The isocrystal $M'$ is pure of slope $\nslope$ by Proposition~\ref{subquot}.
Comparing the dimensions we deduce from Proposition~\ref{purehom}
that this sequence is exact. Hence $f$ splits.
We get the result by induction on the rank of $M$.%
\end{proof}

%
%
%

\begin{prp}\label{simpleiso}
Let $\lambda\in \bQ$ and let $r > 0$ be the denominator of $\lambda$ written in lowest terms.
If a pure isocrystal of slope $\lambda$ has rank $r$
then it is simple.%
\end{prp}%
\begin{proof}%
We are free to assume that $K$ is separably closed.
The claim follows by applying Proposition~\ref{puredecomp} and comparing
the ranks.%
\end{proof}

\section{The \texorpdfstring{$\nadic$-}{}adic topology}\label{sec:top}


Let $K$ be a field over $\Fq$.
Some of our arguments employ the notion of $\nadic$-adic topology
on $\nBK$-modules and $\nBKz$-modules which we now introduce.

\subsection{\texorpdfstring{$\nBK$}{A\_K}-modules}

\begin{dfn}%
The \emph{$\nadic$-adic topology} on a finitely generated $\nBK$-module $\nBlat$ is given
by the fundamental system of open submodules
$\ncomax^n\nBlat$.%
\end{dfn}

Let $K^s$ be a separable closure of $K$ and let $G$ be the corresponding Galois group.
The ring $\nBR{K^s}$ carries an action of $G$ by functoriality.

\begin{dfn}\label{semiact}%
Let $\nBlat$ be a finitely generated $\nBR{K^s}$-module
equipped with an action of $G$ on the underlying abelian group.
The action is called \emph{semi-\hspace{0pt}linear} if the identity
$\gamma(x t) = \gamma(x)\cdot\gamma(t)$
holds for all $(\gamma,x,t)\in G\times \nBR{K^s}\times \nBlat$.%
\end{dfn}

Our main application of the $\nadic$-adic topology is the following
descent result:

\begin{prp}\label{descent}%
Let $\nBlat$ be a locally free $\nBR{K^s}$-module of finite type
equipped with a semi-\hspace{0pt}linear action of ${G}$
which is continuous in the $\nadic$-adic topology.
Then the natural map 
$\nBR{K^s}\otimes_{\nBK} \nBlat^{{G}} \xrightarrow{\isosign} \nBlat$
is an isomorphism.%
\end{prp}%

We split the proof into a sequence of lemmas.
\begin{lem}\label{descent-lim}%
The natural morphism $\nBlat^{{G}} \xrightarrow{\isosign} \lim (\nBlat/\ncomax^n\nBlat)^{{G}}$
is an isomorphism.%
\end{lem}%
\begin{proof}%
Given a topological abelian group $A$ on which ${G}$ acts continuously
we denote by $\ngcoc{A}$ the abelian group of continuous maps ${G} \to A$
and $d\colon A \to \ngcoc{A}$ the morphism which sends an element $a\in A$
to the map $\gamma \mapsto \gamma a - a$.
Consider a commutative diagram
\begin{equation*}
\xymatrix{
0 \ar[r] & \nBlat^{{G}} \ar[r] \ar[d] & \nBlat \ar[r]^{d\quad} \ar[d] & \ngcoc{\nBlat} \ar[d] \\
0 \ar[r] & \lim(\nBlat/\ncomax^n\nBlat)^{{G}} \ar[r] &
\lim \nBlat/\ncomax^n\nBlat \ar[r]^{\lim d\quad\quad} &
\lim \ngcoc{\nBlat/\ncomax^n\nBlat}.
}
\end{equation*}
The bottom row of this diagram is exact since limits are left exact.
The middle and right vertical arrows are isomorphisms since $\nBlat$ is
$\nadic$-adically complete. Hence the first vertical arrow is an isomorphism.%
\end{proof}

\begin{lem}\label{descent-termwise}%
The natural morphism 
$K^s\otimes_K (\nBlat/\ncomax^n\nBlat)^{G} \xrightarrow{\isosign} \nBlat/\ncomax^n\nBlat$
is an isomorphism for all $n>0$.%
\end{lem}%
\begin{proof}%
The quotient $\nBlat/\ncomax^n\nBlat$ is a finite-\hspace{0pt}dimensional
$K^s$-vector space equipped with a semi-\hspace{0pt}linear action of ${G}$ which is
continuous in the discrete topology. So the result follows from Hilbert's Theorem~90.%
\end{proof}

\begin{lem}\label{descent-locfree}%
The $\nBK$-module $\nBlat^{G}$ is locally free of finite type.
Furthermore the natural morphism
$\nBlat^{G}/\ncomax^n\nBlat^{G} \xrightarrow{\isosign} (\nBlat/\ncomax^n\nBlat)^{G}$
is an isomorphism for each $n > 0$.%
\end{lem}%
\begin{proof}%
Set $U_n = (\nBlat/\ncomax^n\nBlat)^{G}$
and consider the inverse system 
$\{ U_n \}_{n>0}$. We claim that for every $n > 0$%
\begin{enumerate}%
\renewcommand{\theenumi}{\roman{enumi}}%
\item%
\label{descent-locfree-termwise}%
the $K\otimes_{\Fq}\ncoint/\ncomax^n$-module $U_n$ is locally free of finite type,

\item%
\label{descent-locfree-reg}%
the transition maps induce isomorphisms
$U_{n+1}/\ncomax^n U_{n+1} \xrightarrow{\isosign} U_n$.%
\end{enumerate}%
Now $\nBlat^{G} = \lim U_n$ by Lemma~\ref{descent-lim}.
In view of \eqref{descent-locfree-termwise} and \eqref{descent-locfree-reg}
the result follows from \stacks{0D4B}.

The base change of $U_n$ to the ring $K^s\otimes_{\Fq}\ncoint/\ncomax^n$
is locally free of finite type by Lemma~\ref{descent-termwise}.
So faithfully flat descent implies \eqref{descent-locfree-termwise}.
To prove \eqref{descent-locfree-reg}
we consider a complex of ${G}$-modules
\begin{equation*}
C = \big[\:
\nBlat/\ncomax^{n+1}\nBlat \xrightarrow{\:z^n\:} \nBlat/\ncomax^{n+1}\nBlat\:\big]
\end{equation*}
placed in degrees $0$ and $1$.
Hilbert's Theorem~90 shows that the ${G}$-module
$\nBlat/\ncomax^{n+1}\nBlat$ is a direct sum of copies of $K^s$.
Additive Theorem~90 implies that the hypercohomology
spectral sequence $\uH^j({G},\,C^i) \Rightarrow \uH^{i+j}({G},\,C)$ collapses
at the first page.
The same argument shows that $\uH^0(C)$ and $\uH^1(C)$ are direct sums
of copies of $K^s$ so the hypercohomology spectral sequence
$\uH^i({G},\,\uH^j(C)) \Rightarrow \uH^{i+j}({G},\,C)$ collapses at the second page.
Comparing the pages we obtain~\eqref{descent-locfree-reg}.%
\end{proof}

\begin{proof}[Proof of Proposition~\ref{descent}]%
The $\nBK$-module $\nBlat^{G}$ is locally free of finite type by Lemma~\ref{descent-locfree}.
So it is enough to show that the natural map
$\nBR{K^s}\otimes_{\nBK}\nBlat^{G} \to \nBlat$
is an isomorphism modulo $\ncomax$.
The natural morphism
$\nBlat^{G}/\ncomax\nBlat^{G} \to (\nBlat/\ncomax\nBlat)^{G}$
is an isomorphism by Lemma~\ref{descent-locfree}.
The claim then follows from Lemma~\ref{descent-termwise}.%
\end{proof}

\subsection{\texorpdfstring{$\nBKz$}{B\_K}-modules}

\begin{dfn}%
Let $M$ be a locally free $\nBKz$-module of finite type.
Pick a lattice $\nBlat \subset M$.
The \emph{$\nadic$-adic topology} on $M$ is
given by the fundamental system of open submodules
$\ncomax^n\nBlat$, $n\geqslant 0$.%
\end{dfn}

\begin{lem}%
The $\nadic$-adic topology is well-defined
and does not depend on the choice of a lattice.
Every $\nBKz$-linear morphism of locally free $\nBKz$-modules
of finite type is continuous in the $\nadic$-adic topology.%
\end{lem}%
\begin{proof}%
Follows since
the ring $\nBK$ is a finite product of discrete valuation rings
and $\nBKz = \nBK[z^{-1}]$ is the product of the corresponding fields of fractions.%
\end{proof}


The notion of a semi-linear action of ${G}$
on a locally free $\nBRz{K^s}$-module of finite type
is defined in the same way as for $\nBR{K^s}$-modules.
%
The following proposition will be useful for averaging arguments.

\begin{prp}\label{actavg}%
Let $M$ be a locally free $\nBRz{K^s}$-module of finite type
equipped with a semi-\hspace{0pt}linear action of ${G}$
which is continuous in the $\nadic$-adic topology.
Then the stabilizer of every lattice $\nBlat \subset M$ is open.%
\end{prp}%
\begin{proof}%
Let $x_1, \dotsc, x_n$ be the generators of $\nBlat$.
For every $i$ the map ${G} \to M$, $\gamma \mapsto \gamma x_i$ is continuous
and so the subset $U_i = \{ \gamma \mid \gamma x_i \in \nBlat \}$
is open.
The open subset $\bigcap_{i=1}^n U_i$ contains the identity element and
stabilizes $\nBlat$. Whence the result.%
\end{proof}
%

%

\section{Tate modules for pure isocrystals}\label{sec:tate}

Let $K$ be a field over~$\Fq$.
We shall define Tate modules for pure $\nBKz$-isocrystals.
Under a suitable condition on the field $K$
we shall prove that
the Tate module functor is fully faithful
and describe its essential image. 

\subsection{Tate modules}\label{ssec:tate}
Let $\nslope$ be a rational number.
We need to fix two objects:
\begin{enumerate}%
\renewcommand{\theenumi}{\roman{enumi}}%
\item\label{chooseksep}%
A separable closure $K^s$ of $K$.
Let $\Fqbar$ be the algebraic closure of~$\Fq$ in~$K^s$.

\item\label{choosechar}%
A pure and simple isocrystal $N$ of slope $\nslope$ over $\nBRz{\Fqbar}$.%
\end{enumerate}%
Such an isocrystal exists by Proposition~\ref{simple}.
In analogy with $K^s$ the isocrystal $N$ is unique up to a non-unique isomorphism.

\begin{dfn}\label{dfntate}%
For each pure isocrystal $M$ of slope $\nslope$ we set
\begin{equation*}
\nTate{M} = \Hom_{\nBRz{\Fqbar}\{\tau\}}(N, \:\nBRz{K^s} \otimes_{\nBKz} M).%
\end{equation*}%
We call this the \emph{Tate module} of $M$.%
\end{dfn}%

Equivalently one can define the Tate module by the formula
\begin{equation*}
\nTate{M} = \Hom(\nBRz{K^s}\otimes_{\nBRz{\Fqbar}} N, \:\nBRz{K^s} \otimes_{\nBKz} M)
\end{equation*}
where the $\Hom$~set is computed in the category of $\nBRz{K^s}$-isocrystals.
For our purposes it will be more convenient to work with the formula of Definition~\ref{dfntate}.

By Proposition~\ref{simple} the ring $\nDN = \End N$ is a central division algebra
over $\ncoef$ and its Hasse invariant is the residue class
of $-\nslope$ in $\bQ/\bZ$.
The Tate module $\nTate{M}$
carries a right action of $D$ by composition.

Our next task is to equip $\nTate{M}$ with a Galois action.
We identify $\Gal(\Fqbar/\Fq)$ with $\Zhat$ using the \emph{geometric} Frobenius as a generator.
Let $G_K = \Gal(K^s/K)$.
Restriction to $\Fqbar$ defines a homomorphism $\ord\colon G_K \to \Zhat$.

\begin{dfn}\label{dfnweil}%
The \emph{Weil group} $W_K$ is $\ord^{-1} \bZ$
equipped with the topology in which 
$\ord^{-1}\{0\}$ is an open subgroup
with its natural profinite topology.%
\end{dfn}

The $\nBKz$-algebra $\nBRz{K^s}$ carries an action of $G_K$ by functoriality.
Extending it by the identity on $M$ we obtain an action of $G_K$ on $\nBRz{K^s} \otimes_{\nBKz} M$.

\begin{dfn}\label{dfnweilact}%
Let $\gamma$ be an element of $W_K$.
For every $f \in \nTate{M}$
we define a map $\gamma(f)\colon N \to \nBRz{K^s}\otimes_{\nBKz} M$ by the formula
$x \mapsto \gamma\big(f(\tau^{\ord\gamma} x)\big).$%
\end{dfn}%

\begin{lem}\label{actwd}%
The map $\gamma(f)$ is a well-defined morphism
of left $\nBRz{\Fqbar}\{\tau\}$-modules.\end{lem}

We thus get a left action of $W_K$ on $\nTate{M}$.
It commutes with the right action of $D$ by construction.

\begin{proof}[Proof of Lemma~\ref{actwd}]%
To make sense of the negative powers of $\tau$ note that the map
$\tau\colon N \to N$ is bijective since $\Fqbar$ is perfect.
%
The endomorphism $\sigma$ of $\nBRz{K^s}$ commutes with $\gamma$
by functoriality of $\nBRz{K^s}$. 
As a consequence $\gamma(f)$ commutes with the multiplication by $\tau$.

The key point of Definition~\ref{dfnweilact} is that
it produces a $\nBRz{\Fqbar}$-linear map. Let us explain this.
Let $a \in \nBRz{\Fqbar}$ and $x \in N$. Applying $\gamma(f)$ to $ax$
we get
\begin{equation*}
\gamma(f)(ax) = \gamma(f(\tau^{\ord\gamma} \cdot ax)) =
\gamma(f(\sigma^{\ord\gamma}(a) \cdot \tau^{\ord\gamma} x)) =
\gamma(\sigma^{\ord\gamma}(a)) \cdot \gamma(f)(x).
\end{equation*}
Consulting the definition of the homomorphism $\ord\colon W_K \to \bZ$
we see that $\gamma(\alpha) = \alpha^{q^{-\ord\gamma}}$
for all $\alpha\in\Fqbar$.
Hence
$\gamma(\sigma^{\ord\gamma}(a)) = a$
and the claim follows.%
\end{proof}




\begin{prp}\label{taterank}%
Let $r > 0$ be the denominator of $\nslope$ written in lowest terms.
Then for every pure isocrystal $M$ of slope $\nslope$ we have
\begin{equation*}
\rank_{\nDN} \nTate{M} = \tfrac{1}{r}\rank M.%
\end{equation*}%
\end{prp}%
\begin{proof}%
The isocrystal $N$ has rank $r$ by Proposition~\ref{simple}
so Proposition~\ref{purehom} shows that $\dim_\ncoef \nDN = r^2$.
Now $\dim_\ncoef \nTate{M} = r \cdot \rank M$ by the same Proposition~\ref{purehom}
and the claim follows.%
\end{proof}

Next we shall 
describe the target category of the functor $M \mapsto \nTate{M}$.

\begin{dfn}%
A \emph{$(W_K,\nDN)$-representation} is a right $\nDN$-module $V$ of finite
rank equipped with a left action of $W_K$ which is 
continuous in the 
$\ncoef$-vector space topology.
A \emph{morphism} of $(W_K,\nDN)$-representations is a morphism of underlying $\nDN$-modules
which commutes with the action of $W_K$.%
\end{dfn}
 

\begin{dfn}%
A \emph{lattice} in a finite-dimensional $\ncoef$-vector space $V$
is a finitely generated $\ncoint$-submodule which spans $V$ over $\ncoef$.%
\end{dfn}


\begin{dfn}\label{dfnwdslope}%
A $(W_K,\nDN)$-representation $V$ is called \emph{$\nslope$-admissible}
if there are integers $s$ and $r$, $r > 0$, and a lattice $\Lambda\subset V$
such that $\nslope = \frac{s}{r}$ and
$\gamma^{r\nresdeg} \Lambda = \Lambda z^{s\ord\gamma}$ for all $\gamma\in W_K$.%
\end{dfn}

In contrast to pure isocrystals
the rational number $\frac{s}{r}$
may depend on the choices of $s$ and $r$.
The reason is that $\ord(W_K)$ may be the zero subgroup
in which case every $s$ works.


\begin{thm}\label{tate}%
The functor $M \mapsto \nTate{M}$ transforms pure isocrystals of slope $\nslope$
to $\nslope$-admissible $(W_K,\nDN)$-representations.
If $W_K$ is dense in $G_K$ then this functor
is an equivalence of categories.%
\end{thm}%

\begin{rmk}%
The Weil group $W_K$ is dense in $G_K$ if and only if $K \cap \Fqbar$ is either finite or coincides with $\Fqbar$.
This condition holds for every field which is finitely generated over $\Fq$,
for every local field, and is preserved under perfect closure.%
\end{rmk}%

\begin{exa}%
In the case $\nslope = 0$ the algebra $\nDN$ coincides with $\ncoef$.
A~$(W_K,\ncoef)$-representation is $0$-admissible
if and only if it contains a $W_K$-stable lattice.
This happens if and only if the action of the subgroup $W_K \subset G_K$
extends to a continuous action of its closure.
So Theorem~\ref{tate} reproduces the classical equivalence of categories of pure $\nBKz$-isocrystals of slope~$0$ and
continuous $G_K$-representations in finite-dimensional $\ncoef$-vector spaces
for every field $K$ such that $W_K$ is dense in $G_K$.%
\end{exa}

%

\begin{exa}
Suppose that $\ncoef = \rF{\Fq}$ and set $K = \Fq(\!(\zeta)\!)$.
Upon the identification of the uniformizers $z$ and $\zeta$
local class field theory produces an isomorphism of topological groups
$W_K^{\textup{ab}} \xrightarrow{\isosign} \ncoef^\times$.
Composing it with the quotient morphism $W_K \twoheadrightarrow W_K^{\textup{ab}}$
we get a character $\rho_{\nArt}\colon W_K \twoheadrightarrow \ncoef^\times$.

This character defines a $(W_K,\ncoef)$-representation $V_{\nArt}$ of rank $1$.
Every arithmetic Frobenius element $\gamma \in W_K$
is mapped by $\rho_{\nArt}$ to a uniformizer of $F$,
and so $V_{\nArt}$ is $(-1)$-admissible.
By Theorem~\ref{tate} the representation $V_{\nArt}$
corresponds to a pure isocrystal $M_{\nArt}$ of slope~$-1$ and rank~$1$.
One can prove that
the isocrystal $M_{\nArt}$ is generated over $\nBKz$ by an element $e$
such that $\tau(e) = (1-\frac{\zeta}{z}) e$.
This gives a neat description of $\rho_{\nArt}$ as the
homomorphism associated to the Tate module $T(\nBKz e)$
with $\tau(e) = (1-\frac{\zeta}{z}) e$.
\end{exa}

\begin{rmk}\label{tatefisoc}%
Set $p = \textrm{char}\,\Fq$. Let us describe a $p$-adic analog
of Tate modules for isocrystals.
We need to assume that the field $K$ is perfect.
For each $\Fq$-algebra $R$ we denote by $W(R)$ the ring of $p$-typical
Witt vectors. 
Consider the rings
\begin{equation*}
\nQq = W(\Fq)[\tfrac{1}{p}], \quad
\nqBK = W(K), \quad
\nqBKz = W(K)[\tfrac{1}{p}].
\end{equation*}
The $q$-Frobenius of $K$ induces an automorphism $\nsigma\colon \nqBKz \to \nqBKz$. 

An \emph{$F$-isocrystal over $K$} is a $\nsigma$-module over $\nqBKz$,
compare \cite[Remark~2.10]{kedlaya-notes}.
The coefficient field of $F$-isocrystals is~$\nQq$
and the parameter $\nresdeg$ is equal to~$1$.
The notions of purity and slope 
are defined in the same way as for $\nFBKz{\ncoef}$-isocrystals. 
The classification theorem of Dieudonn\'e-Manin \cite[Theorem~3.2]{kedlaya-notes}%
\footnote{NB: Kedlaya's convention for the sign of Hasse invariant is opposite to the one of this paper.}
subsumes our Propositions~\ref{simple} and \ref{puredecomp}.
For each pure $F$-isocrystal $M$ of slope $\nslope$ we have
the Tate module
\begin{equation*}
\nTate{M} = \Hom_{\nqBRz{\Fqbar}\!\{\tau\}}(N,\:\nFBRz{\,\nQq}{K^s\!}\otimes_{\nqBKz} M).
\end{equation*}
Definition~\ref{dfnweilact} provides this with an action of $W_K$.
Theorem~\ref{tate} remains true. 
Its proof translates to the $p$-adic setting in a formal way:
It is enough to replace $\nBK$ with $\nqBK$, 
the uniformizer $z$ with $p$ and Theorem~\ref{wellknown}
with the parallel result for lisse $p$-adic sheaves on $\etsite{K}$.%
\end{rmk}

\subsection{Proof of Theorem~\ref{tate}}
\begin{lem}\label{tate-wdrep}%
For every pure isocrystal $M$ of slope $\nslope$
the Tate module $\nTate{M}$ is a $(W_K,\nDN)$-representation.%
\end{lem}%
\begin{proof}%
We need to show that the open subgroup ${\nGalZero} \subset W_K$ of elements of order $0$ 
acts continuously on $\nTate{M}$.
The fixed field $L = (K^s)^{\nGalZero}$
contains $\Fqbar$ by construction.
We have
$\nTate{M} = (\nBRz{K^s}\otimes_{\nBRz{L}} H)^\tau$
where
$H = \iHom(\nBRz{L}\otimes_{\nBRz{\Fqbar}} N,\:\nBRz{L}\otimes_{\nBKz} M)$.
The subgroup ${\nGalZero}$ acts on $\nTate{M}$ via $K^s$.

The isocrystal $H$ is pure of slope~$0$
and so arises from a $\nsigma$-module $\nBlat$ over $\nBR{L}$
by Proposition~\ref{fdmpure0}.
Theorem~\ref{wellknown} associates to $\nBlat$ a lisse $\nadic$-adic sheaf on the small \'etale site of $\Spec L$.
The equivalence of Theorem~\ref{wellknown} is compatible with the base change $L \to K^s$.
So the Galois representation corresponding to $\nBlat$ is
\begin{equation*}
\Lambda = (\nBR{K^s}\otimes_{\nBR{L}} \nBlat)^{\tau}
\end{equation*}
on which ${\nGalZero} = \Gal(K^s/L)$ acts via $K^s$.

Theorem~\ref{wellknown} implies that $\Lambda$ is a finitely generated free $\ncoint$-module
and that the action of ${\nGalZero}$ is continuous in the $\nadic$-adic topology.
Since $\nBlat$ is a lattice in $H$ it follows that $\Lambda$ is open in $\nTate{M}$,
and we get the result.%
\end{proof}

\pagebreak
\begin{lem}\label{fincrit}%
An $\ncoint$-submodule $\Lambda$ of a finite-\hspace{0pt}dimensional $\ncoef$-vector space $V$
is finitely generated if and only if it is $\nadic$-adically separated.%
\end{lem}%
\begin{proof}%
Every finite-dimensional $\ncoint/\ncomax$-vector subspace
$\xbar{M} \subset \Lambda/\ncomax\Lambda$ lifts to a finitely generated $\ncoint$-submodule $M\subset\Lambda$.
The rank of $M$ is bounded by the dimension of~$V$.
Since $\rank_\ncoint M \geqslant \dim_{\ncoint/\ncomax} \xbar{M}$
it follows that $\Lambda/\ncomax\Lambda$ is finite.
Invoking \stacks{031D} we conclude that $\Lambda$ is finitely generated.%
\end{proof}

\begin{lem}\label{tate-wdslope}%
For every pure isocrystal $M$ of slope $\nslope$ 
the Tate module $\nTate{M}$ is $\nslope$-admissible.%
\end{lem}%
\begin{proof}%
Since $N$ is pure of slope $\nslope$
there are integers $s$ and $r$, $r > 0$,
and a lattice $\nlatN \subset N$ such that $\nslope = \frac{s}{r}$
and $\tau^{r\nresdeg} \nlatN = z^s\nlatN$.
Pick a lattice $\nBlat \subset M$ and consider the $\ncoint$-submodule
\begin{equation*}
\Lambda = \{f \in \nTate{M} \mid f(\nlatN) \subset \nBR{K^s}\otimes_{\nBK}\nBlat \}.
\end{equation*}
%
Let us show that $\Lambda$ is finitely generated.
Note that
\begin{equation*}
\Lambda \ncomax^n = \{ f \in \nTate{M} \mid f(\nlatN) \subset \nBR{K^s}\otimes_{\nBK}  \ncomax^n \nBlat \}.
\end{equation*}
Since $\nBR{K^s}\otimes_{\nBK} \nBlat$ is $\nadic$-adically separated we conclude that
\begin{equation*}
\bigcap\nolimits_{n>0}\Lambda \ncomax^n = \{ f \in \nTate{M} \mid f(\nlatN) = 0 \}.
\end{equation*}
Now $\nlatN$ is an $\nBR{\Fqbar}$-lattice in $N$ and so the right hand side of this equation
is zero. In other words $\Lambda$ is $\nadic$-adically separated.
As $\nTate{M}$ is a finite-\hspace{0pt}dimensional $\ncoef$-vector space
Lemma~\ref{fincrit} shows that $\Lambda$ is finitely generated over $\ncoint$.

The $\nBR{\Fqbar}$-module $\nlatN$ is finitely generated.
The fact that $\nBlat$ is a lattice in $M$ then implies that $\Lambda$ is a lattice in $\nTate{M}$.
It remains to prove that $\Lambda$ satisfies the condition of Definition~\ref{dfnwdslope}.
Let $f \in \nTate{M}$ and let $\gamma \in W_K$. Set $n =\ord\gamma$.
The identity $\tau^{r\nresdeg}\nlatN = z^s \nlatN$ implies that
\begin{equation*}
(\gamma^{r\nresdeg} f z^{-s n})(\nlatN) =
\gamma(f(\tau^{r\nresdeg n} z^{-s n} \nlatN)) = 
\gamma(f(\nlatN))
\end{equation*}
Now $\gamma$ is an automorphism of $\nBR{K^s}\otimes_{\nBK}\nBlat$
and so $\gamma^{r\nresdeg} f z^{-sn}$ belongs to $\Lambda$ if and only if $f$ does.
In other words $\gamma^{r\nresdeg} \Lambda = \Lambda z^{s\ord\gamma}$.
The result follows since $\nslope = \frac{s}{r}$ by construction.%
\end{proof}

Recall that an \emph{order} in $\nDN$ is a subring which is at the same time a lattice.
According to Theorems 12.8 and 13.3 of \cite[Chapter~3]{reiner}
the $\ncoef$-algebra $\nDN$ contains a unique maximal order. 

\begin{dfn}%
We denote this maximal order by $\nOD$.%
\end{dfn}


\begin{lem}\label{admequiv}
In Definition~\ref{dfnwdslope} one is free to assume
that the lattice $\Lambda$ is an $\nOD$-submodule
and is stable under $\nGalZero = \ord^{-1}\{0\}$.%
\end{lem}%
\begin{proof}%
Let $s$ and $r$ be the integers appearing in Definition~\ref{dfnwdslope}.
We have
\begin{equation}\label{admcond}\tag{$\star$}
\gamma^{r\nresdeg} \Lambda = \Lambda z^{s\ord\gamma}
\end{equation}
for every $\gamma\in W_K$.

As the action of $W_K$ on $V$ is continuous the lattice $\Lambda$
is stabilized by an open subgroup of $G_0$.
Such a subgroup has finite index since $G_0$ is compact.
Hence the $\ncoint$-module
$\Lambda' = \sum_{g\in\nGalZero} g\Lambda$
is finitely generated, and so is a lattice in $V$.
The fact that $G_0$ is normal in $W_K$ implies that
for each $\gamma \in W_K$ we have
$\gamma \Lambda' = \sum_{g\in\nGalZero} g \gamma\Lambda$.
As a consequence $\Lambda'$ satisfies \eqref{admcond}.
Moreover, $\Lambda'$ is $\nGalZero$-stable by construction.


Let $\Lambda''$ be the $\nOD$-submodule of $V$ generated by $\Lambda'$.
This is a lattice since $\nOD$ is finitely generated as an $\ncoint$-module.
As the action of $W_K$ on $V$ is $\nDN$-linear it follows that
$\Lambda''$ is $\nGalZero$-stable and satisfies \eqref{admcond}.
Whence the result.%
\end{proof}%

\begin{dfn}\label{tate-revfunct}%
To every $(W_K,\nDN)$-representation $V$ we attach a $\nBRz{\Fqbar}$-isocrystal
\begin{equation*}
N(V) = V\otimes_{\nDN} N
\end{equation*}
equipped with the following action of $W_K$:
\begin{equation*}
\gamma(v \otimes x) = \gamma(v) \otimes \tau^{-\ord\gamma} x, \quad
(\gamma,v,x) \in W_K \times V \times N.
\end{equation*}
\end{dfn}

The group $G_K$ 
preserves the subring $\nBRz{\Fqbar} \subset \nBRz{K^s}$.
So it makes sense to speak of a semi-linear action of $G_K$
and its subgroups on $\nBRz{\Fqbar}$-modules.
The action of $W_K$ on $N(V)$ is semi-linear in this sense.

\begin{lem}\label{tate-weillat}%
For every $\nslope$-admissible $(W_K,\nDN)$-representation $V$ the isocrystal $N(V)$ admits a $W_K$-stable lattice.
Moreover the action of $W_K$ on $N(V)$ is continuous in the $\nadic$-adic topology.%
\end{lem}%
\begin{proof}%
Since the isocrystal $N$ is pure of slope~$\nslope$
there are integers $s$ and $r$, $r > 0$,
and a lattice $\nlatN \subset N$ such that
$\nslope = \frac{s}{r}$ and
$\tau^{r\nresdeg}\nlatN = z^s \nlatN$.

The representation $V$ is $\nslope$-admissible,
so multiplying $s$ and $r$ by a common factor
we are free to assume that there is a lattice $\Lambda \subset V$
such that $\gamma^{r\nresdeg} \Lambda = \Lambda z^{s\ord\gamma}$ for all $\gamma\in W_K$.
In view of Lemma~\ref{admequiv} we are also free to assume
that $\Lambda$ is an $\nOD$-submodule and that it is stable under
$\nGalZero = \ord^{-1}\{0\}$.

The $\nBR{\Fqbar}$-lattice $\nBlat = \Lambda\otimes_{\nOD}\nlatN$ in $N(V)$ is stable under $\nGalZero$.
For each $n > 0$ the group $\nGalZero$ acts on the quotient
$\nBlat/\ncomax^n\nBlat = (\Lambda/\ncomax^n\Lambda) \otimes_{\nOD} \nlatN$
through the factor $\Lambda/\ncomax^n\Lambda$.
Since the action of $\nGalZero$ on $\Lambda$ is continuous
it follows that the $\nGalZero$-stabilizer of every element of $\nBlat/\ncomax^n\nBlat$
is open.
Hence the action of $W_K$ is continuous in the $\nadic$-adic topology on $N(V)$.

Next we claim that $\gamma^{r\nresdeg}\nBlat = \nBlat$ for all $\gamma \in W_K$.
By assumption for every pair $(v,x) \in \Lambda\times\nlatN$
there is a pair $(w,y) \in \Lambda\times\nlatN$
such that
\begin{equation*}
\gamma^{r\nresdeg} v = w z^{s\ord\gamma}, \quad
\gamma^{-r\nresdeg\ord\gamma} x =  z^{-s\ord\gamma} y.
\end{equation*}
Therefore
\begin{equation*}
\gamma^{r\nresdeg}(v \otimes x) = 
w z^{s\ord\gamma} \otimes z^{-s\ord\gamma} y = w\otimes y.
\end{equation*}
We conclude that $\gamma^{r\nresdeg}\nBlat \subset \nBlat$.
This inclusion is an equality since $\gamma^{-r\nresdeg}\nBlat \subset \nBlat$
by the same argument.

Pick an element $\gamma_1 \in W_K$ such that $\ord(\gamma_1)$ generates $\ord(W_K)$
and set
\begin{equation*}
\navg{\nBlat} = \sum_{i\in\bZ} \gamma_1^i \nBlat.
\end{equation*}
This is an $\nBR{\Fqbar}$-lattice since $\gamma_1^{r\nresdeg n} \nBlat = \nBlat$ for every $n \in \bZ$.
By construction $\navg{\nBlat}$ is stable
under $\gamma_1^{\bZ}$ and $\nGalZero$.
Since $W_K = \gamma_1^{\bZ} \nGalZero$
we conclude that $\navg{\nBlat}$ is $W_K$-stable.%
\end{proof}

\begin{dfn}
We denote by $\nWInv$ the 
fixed field of the Weil group $W_K$
and set $\nWCl = \Gal(K^s/\nWInv)$.
\end{dfn}

By construction $\nWCl$ is the closure of $W_K$ in $G_K$.
When $\ord(W_K)$ is not zero the inclusion $W_K\hookrightarrow \nWCl$
induces a topology on $W_K$
which is coarser than the topology of Definition~\ref{dfnweil}.

Our next step is to show that the action of $W_K$ on $N(V)$
extends in a canonical way to an action of $\nWCl$. To this end
we shall use cocycles with values in non-commutative groups.

Let $H$ be a not necessarily commutative group equipped with
an action of $W_K$. Given an element $h \in H$ and an element $\gamma\in W_K$
we denote the result of application of $\gamma$ to $h$ by $\app{\gamma}{h}$.
%
Recall that a \emph{cocycle} $a\colon W_K \to H$
is a map that satisfies the identity
\begin{equation*}
a(\gamma_1 \gamma_2) = a(\gamma_1) \cdot \app{\gamma_1}{a(\gamma_2)}, \quad \gamma_1, \gamma_2 \in W_K.
\end{equation*}
We shall use the same terminology with respect to the group $\nWCl$.

\begin{lem}\label{cocycle}
Let $H$ be a discrete group equipped with a continuous action of the compact group $\nWCl$.
Suppose that
\begin{enumerate}
\item
$H$ is a filtered colimit of finite $\nWCl$-stable subgroups,

\item
the subgroup $\nGalZero\subset\nWCl$ acts trivially on $H$.
\end{enumerate}
Then for every continuous cocylce $a\colon W_K \to H$
there is a unique continuous cocycle $\abar\colon \nWCl \to H$
such that $\abar|_{W_K} = a$.
\end{lem}

Here the action of $W_K$ on $H$ is induced by the inclusion $W_K \hookrightarrow\nWCl$.
We avoid speaking of extension of $a$ to $\nWCl$ since the inclusion
$W_K \hookrightarrow \nWCl$ is not a topological embedding in general.

\begin{proof}[Proof of Lemma~\ref{cocycle}]%
Without loss of generality we assume that $\ord(W_K) \ne 0$.
Pick an element $\gamma_1 \in W_K$ such that $\ord(\gamma_1)$ generates the subgroup $\ord(W_K) \subset\bZ$.
The choice of $\gamma_1$ identifies the group $W_K$ with a topological semidirect product
$\nGalZero \rtimes(\nZ)$
and the group $\nWCl$ with a topological semidirect product
$\nGalZero \rtimes \nZCl$ where
$\nZCl$ is the procyclic subgroup generated by $\gamma_1$.

Let $a\colon W_K \to H$ be a continuous cocycle.
First we shall prove that the restriction of $a$ to $\nZ$ extends to $\nZCl$.
Since $H$ is a filtered colimit of finite $\nWCl$-stable subgroups it follows
that the image $a(\nZ)$ is contained in a finite $\nWCl$-stable subgroup $H' \subset H$.
As $H'$ is finite a sufficiently large subgroup
$\gamma_1^{N\bZ} \subset \nZ$ acts on $H'$ trivially.
Enlarging $N$ if necessary we are free to assume that
$a(\gamma_1^{N\bZ}) = 1$. The cocycle identity then implies
that $a(\gamma_1^n)$ depends only on the residue class of $n$
modulo $N$. In other words the cocycle $a$ arises from
a cocycle $a_0\colon \gamma_1^\bZ/\gamma_1^{N\bZ} \to H'$.
This $a_0$ induces a cocycle $a_\nZCl\colon \nZCl \to H$
by composition with the quotient morphism
$\nZCl \twoheadrightarrow \gamma_1^{\bZ}/\gamma_1^{N\bZ}$
and the inclusion $H' \hookrightarrow H$.

Next we define a map $\abar\colon \nGalZero \rtimes \nZCl \to H$
by the formula $g \gamma \mapsto a(g) a_\nZCl(\gamma)$,
$g \in \nGalZero$, $\gamma\in\nZCl$.
This map is continuous since the restriction of $a$ to $\nGalZero$
is continuous and $a_{\nZCl}$ is continious.
By assumption $\nGalZero$ acts trivially on $H$
so the cocycle identity implies that $\abar$ extends the map~$a$.
As $W_K = \nGalZero \rtimes \nZ$ is dense in $\nWCl = \nGalZero\rtimes \nZCl$
we conclude that the map $\abar$ is a cocycle.
The density property also implies that $\abar$ is the unique
extension of~$a$.
\end{proof}

\begin{lem}\label{tate-aext}%
For every $\nslope$-admissible $(W_K,\nDN)$-representation $V$
the action of $W_K$ on $N(V)$ extends uniquely to a continuous semi-\hspace{0pt}linear
action of $\nWCl$.%
\end{lem}%
\begin{proof}%
Let $\nBlat \subset N(V)$ be the $W_K$-stable lattice constructed in Lemma~\ref{tate-weillat}.
We shall study the action of $W_K$ on the quotients $\nBlat/\naideal\nBlat$
where $\naideal$ runs over nonzero proper ideals of $\ncoint$.

Fix an isomorphism $\nBlat/\naideal\nBlat \xrightarrow{\isosign} (\naring{\Fqbar})^{m}$.
The resulting action of $W_K$ on the free module $(\naring{\Fqbar})^{m}$ is encoded
by a map $a\colon W_K \to \GL_m(\naring{\Fqbar})$ such that
\begin{equation*}
\gamma\cdot\begin{pmatrix}
x_1 \\ \vdots \\ x_m
\end{pmatrix} = 
a(\gamma) \cdot \begin{pmatrix} \gamma(x_1) \\ \vdots \\ \gamma(x_m) \end{pmatrix}, \quad
\gamma \in W_K, \:
x_1, \dotsc, x_m \in \naring{\Fqbar}.
\end{equation*}
This is a cocycle with respect to the coordinatewise action of $W_K$ on $\GL_m(\naring{\Fqbar})$.
Lemma~\ref{tate-weillat} implies that the cocycle $a$ is continuous.
Now
\begin{enumerate}
\item
the group $\nGalZero$ acts trivially on $\GL_m(\naring{\Fqbar})$,

\item
the group $\GL_m(\naring{\Fqbar})$ is a filtered colimit of finite
$\nWCl$-stable subgroups $\GL_m(\naring{k})$ where $k$ runs over finite extensions of $\Fq$.
\end{enumerate}
Hence Lemma~\ref{cocycle} shows that the cocycle $a$
extends uniquely to the group $\nWCl$.
In other words the action of $W_K$ on $\nBlat/\naideal\nBlat$
extends uniquely to a continuous action of $\nWCl$.
Applying this argument for every $I = \ncomax^n$ we conclude that
the action of $W_K$ on $\nBlat$ extends uniquely
to a continuous action of $\nWCl$.%
\end{proof}

Given a $\nslope$-admissible $(W_K,\nDN)$-representation $V$
we equip $\nBRz{K^s}\otimes_{\nBRz{\Fqbar}} N(V)$
with a diagonal action of $W_K$ and set
\begin{equation*}
M(V) = (\nBRz{K^s}\otimes_{\nBRz{\Fqbar}} N(V) )^{W_K}.
\end{equation*}
The structure of a left $\nBRz{K^s}\{\tau\}$-module on
$\nBRz{K^s}\otimes_{\nBRz{\Fqbar}} N(V)$
descends to a structure of
a left $\nBRz{\nWInv}\{\tau\}$-module on $M(V)$.

\begin{lem}\label{tate-fontiso}%
For every $\nslope$-admissible $(W_K,\nDN)$-representation $V$
the natural morphism
$\nBRz{K^s}\otimes_{\nBRz{\nWInv}} M(V) \to \nBRz{K^s}\otimes_{\nBRz{\Fqbar}} N(V)$
is a $W_K$-equivariant isomorphism of left $\nBRz{K^s}\{\tau\}$-modules.%
\end{lem}%
\begin{proof}%
Lemma~\ref{tate-aext} implies that the action of $W_K$
on $\nBRz{K^s}\otimes_{\nBRz{\Fqbar}} N(V)$ extends to a continuous semi-linear
action of $\nWCl$. As $W_K$ is dense in $\nWCl$ we conclude that
\begin{equation*}
M(V) = (\nBRz{K^s}\otimes_{\nBRz{\Fqbar}} N(V))^{\nWCl}.
\end{equation*}
Pick a lattice $\nBlat\subset \nBRz{K^s}\otimes_{\nBRz{\Fqbar}} N(V)$.
Its stabilizer is open by Proposition~\ref{actavg} so taking an average over
the action of $\nWCl$ we are free to assume that $\nBlat$ is $\nWCl$-stable.
Then $M(V) = \ncoef\otimes_{\ncoint}(\nBlat^{\nWCl})$ and 
the result follows from Proposition~\ref{descent}.%
\end{proof}

\begin{lem}\label{tate-descent}%
Let $M$ be a left $\nBKz\{\tau\}$-module.
If the base change of $M$ to $\nBRz{K^s}$ is a pure isocrystal of slope $\nslope$
then $M$ is a pure $\nBKz$-isocrystal of slope $\nslope$.%
\end{lem}%
\begin{proof}%
The ring $\nBRz{K^s}$ is faithfully flat over $\nBKz$.
Invoking faithfully flat descent we conclude that $M$ is a $\nBKz$-isocrystal.

By assumption there is a lattice $\nBlat \subset \nBRz{K^s}\otimes_{\nBKz} M$
and integers $s$ and $r$, $r > 0$, such that
$\nslope = \frac{s}{r}$ and $z^s\nBlat = \nlati{r\nresdeg}{\nBlat}$.
We equip $\nBRz{K^s}\otimes_{\nBKz} M$ with an action of $G_K$ via $\nBRz{K^s}$.
The stabilizer of $\nBlat$ is open by Proposition~\ref{actavg}.
Averaging over the action of $G_K$ we are free to assume that
$\nBlat$ is $G_K$-stable.

The natural map $\nBR{K^s}\otimes_{\nBK} \nBlat^G \to \nBlat$
is an isomorphism by Proposition~\ref{descent}.
Since the ring
$\nBR{K^s}$ is faithfully flat over $\nBK$
it follows that $\nBlat^G$ is a lattice in~$M$.
Taking the base change to $\nBR{K^s}$ we conclude
that $\nlati{r\nresdeg}{\nBlat^G} = 
\nlati{r\nresdeg}{\nBlat^G} + z^s \nBlat^G = z^s\nBlat^G$.
Therefore $M$ is pure of slope $\frac{s}{r} = \nslope$.%
\end{proof}

\begin{lem}\label{tate-fontcrys}
For every $\nslope$-admissible $(W_K,\nDN)$-representation $V$
the module $M(V)$ is a pure $\nBRz{\nWInv}$-isocrystal of slope $\nslope$.%
\end{lem}%
\begin{proof}%
The isocrystal $\nBRz{K^s}\otimes_{\nBRz{\Fqbar}} N(V)$ is pure of slope $\nslope$
by construction. So it follows from Lemma~\ref{tate-fontiso} that the base
change of $M(V)$ to $\nBRz{K^s}$ is a pure isocrystal of slope $\nslope$.
Applying Lemma~\ref{tate-descent} with $K$ replaced by $\nWInv$ we get the result.%
\end{proof}

Let $M$ be a pure isocrystal of slope $\nslope$.
Consider the evaluation morphism
\begin{equation*}
\nev\colon \nBRz{K^s}\otimes_{\nBRz{\Fqbar}} N(\nTate{M}) \to \nBRz{K^s}\otimes_{\nBKz} M, \quad
\alpha \otimes (f \otimes x) \mapsto \alpha f(x).
\end{equation*}
\begin{lem}%
The evaluation morphism is $W_K$-equivariant with respect to the diagonal action
on the source and the action via $\nBRz{K^s}$ on the target.%
\end{lem}%
\begin{proof}%
Recall that $N(\nTate{M}) = \nTate{M}\otimes_{\nDN} N$.
Let $f \otimes x$ be an element of this tensor product
and let $\alpha \in \nBRz{K^s}$. Pick an element $\gamma\in W_K$.
By definition
\begin{equation*}
\gamma(\alpha\otimes(f\otimes x)) = \gamma(\alpha) \otimes (\gamma(f) \otimes \tau^{-\ord\gamma} x).
\end{equation*}
Now $\gamma(f)(\tau^{-\ord\gamma} x) = \gamma(f(x))$ by Definition~\ref{dfnweilact}
and so
\begin{equation*}
\nev(\gamma(\alpha \otimes (f \otimes x))) =
\gamma(\alpha) \cdot \gamma(f(x)) = \gamma(\nev(\alpha\otimes(f\otimes x)))
\end{equation*}
as desired.
\end{proof}

Passing to the invariants of the evaluation morphism we obtain a morphism
\begin{equation*}
e\colon M(\nTate{M}) \to \nBRz{\nWInv}\otimes_{\nBKz} M.
\end{equation*}

\begin{lem}\label{tate-eiso}%
For every pure $\nBKz$-isocrystal $M$ of slope $\nslope$
the natural morphism
$e\colon M(\nTate{M}) \xrightarrow{\isosign} \nBRz{\nWInv}\otimes_{\nBKz} M$
is an isomorphism.%
\end{lem}%
\begin{proof}%
We shall prove that the evaluation morphism is an isomorphism.
Set $N^s = \nBRz{K^s}\otimes_{\nBRz{\Fqbar}} N(\nTate{M})$
and
$M^s = \nBRz{K^s}\otimes_{\nBKz} M$.
The evaluation morphism
is natural not only in $M$ but also in $M^s$.
Proposition~\ref{puredecomp} shows that $M^s$
is a direct sum of copies of $N^s$. So it is enough
to treat the case $M^s = N^s$.
Recall that $\nDN = \Hom_{\nBRz{\Fqbar}\{\tau\}}(N,N)$.
The natural map
$\nDN \to \Hom_{\nBRz{\Fqbar}\{\tau\}}(N,\,N^s)$
is injective.
Proposition~\ref{purehom} implies that its source and target
are of the same dimension over $\ncoef$.
Hence this map is an isomorphism.
As a consequence the evaluation map is an isomorphism
in the case $M^s = N^s$, and the result follows.%
\end{proof}

Let $V$ be a $(W_K,\nDN)$-representation.
Recall that
\begin{equation*}
N(V) = V\otimes_{\nDN} N, \quad
M(V) = \big(\nBRz{K^s}\otimes_{\nBRz{\Fqbar}} N(V)\big)^{W_K}.
\end{equation*}
\begin{lem}\label{tate-mubij}%
For every $(W_K,\nDN)$-representation $V$
the morphism
\begin{equation*}
\ncoev\colon V \xrightarrow{\isosign} \Hom_{\nBRz{\Fqbar}\{\tau\}}(N,\:\nBRz{K^s}\otimes_{\nBRz{\Fqbar}} N(V)), \quad
v \mapsto \big(x \mapsto 1 \otimes (v \otimes x)\big)
\end{equation*}
is a bijection.%
\end{lem}%
\begin{proof}%
Note that $\ncoev$ is natural in $V$ as a right $\nDN$-module.
The claim then follows from the case $V = \nDN$ which is clear.%
\end{proof}

The Tate module of the isocrystal $M(V)$ is
\begin{equation*}
\nTate{M(V)} = \Hom_{\nBRz{\Fqbar}\{\tau\}}(N,\:\nBRz{K^s}\otimes_{\nBRz{\nWInv}}\!M(V))
\end{equation*}
So the canonical morphism
$\nBRz{K^s}\otimes_{\nBRz{\nWInv}} M(V) \to \nBRz{K^s}\otimes_{\nBRz{\Fqbar}} N(V)$
induces a morphism
\begin{equation*}
\nTate{M(V)} \to \Hom_{\nBRz{\Fqbar\{\tau\}}}(N, \:\nBRz{K^s}\otimes_{\nBRz{\Fqbar}} N(V)).
\end{equation*}
Taking its composition with $\ncoev^{-1}$ we get a natural morphism
\begin{equation*}
g\colon \nTate{M(V)} \to V.%
\end{equation*}

\begin{lem}\label{tate-fiso}%
For every $\nslope$-admissible $(W_K,\nDN)$-representation $V$
the natural morphism $g\colon \nTate{M(V)} \xrightarrow{\isosign} V$ is an isomorphism
of $(W_K,\nDN)$-representations.%
\end{lem}%
\begin{proof}%
The canonical map
$\nBRz{K^s}\otimes_{\nBRz{\nWInv}} M(V) \xrightarrow{\isosign} \nBRz{K^s}\otimes_{\nBRz{\Fqbar}} N(V)$
is an isomorphism by Lemma~\ref{tate-fontiso}.
Hence $g$ is a bijection.
Let us show that $g$ is $W_K$-equivariant.
Consulting Definition~\ref{dfnweilact}
we conclude that
induced action of $W_K$ on the Hom set
\begin{equation*}
\Hom_{\nBRz{\Fqbar}\{\tau\}}(N, \: \nBRz{K^s}\otimes_{\nBRz{\Fqbar}} N(V)).
\end{equation*}
has the following description.
Let $f\colon N \to \nBRz{K^s}\otimes_{\nBRz{\Fqbar}} N(V)$ be a morphism
and let $\gamma \in W_K$.
Then $\gamma(f)$ is defined by the formula
\begin{equation*}
x \mapsto \gamma\big(f(\tau^{\ord\gamma} x)\big).
\end{equation*}
By Lemma~\ref{tate-mubij} there is a unique element $v \in V$
such that $f(x) = 1 \otimes (v \otimes x)$ for all $x \in N$.
Hence
\begin{equation*}
\gamma(f)(x) = 1 \otimes \gamma(v \otimes \tau^{\ord\gamma} x).
\end{equation*}
By Definition~\ref{tate-revfunct} we have
$\gamma(v \otimes \tau^{\ord\gamma} x) = \gamma(v) \otimes x$
and so $g$ is $W_K$-equivariant.
In the same way one shows that $g$ is $\nDN$-linear, and the claim follows.%
\end{proof}

\begin{proof}[Proof of Theorem~\ref{tate}]
For every pure isocrystal $M$ of slope $\nslope$
the Tate module $\nTate{M}$ is a $\nslope$-admissible
$(W_K,\nDN)$-representation by Lemma~\ref{tate-wdslope}.

Assume that $W_K$ is dense in $G_K$.
Then $\nWCl = G_K$ and $\nWInv = K$.
The natural morphism $e\colon M(\nTate{M}) \to M$ is an isomorphism by Lemma~\ref{tate-eiso}.
Conversely, for every $\nslope$-admissible $(W_K,\nDN)$-representation $V$
the left $\nBKz\{\tau\}$-module $M(V)$ is a pure isocrystal of slope $\nslope$
by Lemma~\ref{tate-fontcrys}.
Lemma~\ref{tate-fiso} shows that the natural morphism $g\colon \nTate{M(V)} \to V$
is an isomorphism, and the result follows.%
\end{proof}

\subsection{Admissibility for rank~1}

Consider a $(W_K,\nDN)$-representation $V$ of rank $1$ over $\nDN$.
Let us explain how to test the $\nslope$-admissibilty condition for $V$ 
without a recourse to lattices.

Let $\nDval\colon \nDN^\times \to \bQ$ be the unique
extension of the normalized $\nadic$-adic valuation of $\ncoef$.
Pick a nonzero element $v \in V$ and
let $\rho\colon W_K \to \nDN^\times$ be the homomorphism
defined by the rule $v \cdot \rho(\gamma) = \gamma\cdot v$.
By construction the homomorphism
\begin{equation*}
\ord_V\colon W_K \to \bQ, \quad
\gamma \mapsto \nresdeg \cdot \nDval(\rho(\gamma))
\end{equation*}
is independent of the choice of~$v$.

\begin{prp}\label{rankoneadm}%
The representation $V$ is $\nslope$-admissible if and only if
\begin{equation*}
\ord_V(\gamma) = \nslope \cdot \ord\gamma
\end{equation*}
for all $\gamma\in W_K$.%
\end{prp}%
\begin{proof}%
Write $\nslope = \frac{s}{r}$, $r > 0$,
and consider the lattice $\Lambda = v \cdot \nOD$.
We call the triple $(s,r,\Lambda)$ admissible if it
satisfies the condition of Definition~\ref{dfnwdslope}.

For each $\gamma \in W_K$ we have
$\gamma^{r\nresdeg} \Lambda = v \cdot \rho(\gamma)^{r\nresdeg} \nOD$.
So the triple $(s,r,\Lambda)$ is admissible
if and only if $\rho(\gamma)^{r\nresdeg}\nOD = \nOD z^{s\ord\gamma}$
for all $\gamma$.
Now
\begin{equation*}
\nOD = \{0\} \cup \{\,x \in\nDN^\times\mid\nDval(x)\geqslant 0\,\}
\end{equation*}
by Theorem 12.8 of \cite[Chapter~3]{reiner}.
Hence the admissibility condition holds
if and only if $r\nresdeg\,\nDval(\rho(\gamma)) = s \ord\gamma$ for all $\gamma$,
or in other words $\ord_V = \nslope \cdot \ord$.

At the same time Lemma~\ref{admequiv} shows that $V$ is $\nslope$-admissible
if and only if the condition of Definition~\ref{dfnwdslope} holds for
a lattice of the form $v \cdot \nOD$, $v \in V\backslash\{0\}$.
Comparing with the discussion above we get the result.%
\end{proof}


\section{Formal Drinfeld motives}

Fix a \emph{perfect} field $\nKyu$ over $\Fq$.
We shall prove that Tate modules of $\nBRz{\nKyu}$-isocrystals generalize
a construction of J.-K.~Yu \cite[Section~2]{yu}.

\subsection{Definition and properties}\label{sec:fdm}


\begin{dfn}%
We denote by
$\rtauF{\nKyu}$ the skew field of Laurent series in $\tau^{-1}$
with finite principal parts.
The multiplication in $\rtauF{\nKyu}$ is twisted by the $q$-Frobenius:
for all $\alpha \in \nKyu$ one has $\alpha \tau^{-1} = \tau^{-1} \alpha^q$.
The subring of power series in $\tau^{-1}$ is denoted by $\rtauO{\nKyu}$.%
\end{dfn}

The $\tau^{-1}$-adic topology on $\rtauF{\nKyu}$ is
given by the ring of definition $\rtauO{\nKyu}$ and the ideal of definition $\rtauO{\nKyu}\tau^{-1}$.
This induces a canonical topology on every left $\rtauF{\nKyu}$-vector space of dimension $1$.

\begin{dfn}%
A \emph{formal Drinfeld motive} over $\nKyu$ is
a left $\rtauF{\nKyu}\otimes_{\Fq}\ncoef$-module $\nV$ such that
\begin{itemize}
\item $\dim_{\rtauF{\nKyu}} \nV = 1$,

\item the action of $\ncoef$ on $\nV$ is continuous in the $\tau^{-1}$-adic topology.%
\end{itemize}%
A \emph{morphism} of formal Drinfeld motives is a morphism of left $\rtauF{\nKyu}\otimes_{\Fq} \ncoef$-modules.%
\end{dfn}

\begin{dfn}%
For each continuous morphism $\nphi\colon\ncoef\to\rtauF{\nKyu}$ 
we denote by $\nV(\nphi)$ the formal Drinfeld motive $\rtauF{\nKyu}$
with the action
\begin{equation*}
(\alpha\otimes x) \cdot \nv = \alpha \nv \nphi(x)
\end{equation*}
for all
$(\alpha,x,\nv) \in  \rtauF{\nKyu}\times\ncoef\times\nV(\nphi)$.
We say that $\nV(\nphi)$ is \emph{defined by $\nphi$}.
This construction gives all the formal Drinfeld motives up to isomorphism.
\end{dfn}%

Formal Drinfeld motives are related to formal $\ncoint$-modules of Drinfeld \cite[\S1]{drinfeld-ell},
see \cite[Section 2.3]{zywina-satotate}.
We shall use them to bridge J.-K.~Yu's construction of $\infty$-adic
Galois representations \cite{yu} and the theory of isocrystals.

As the name suggests, one can attach a formal Drinfeld motive to a Drinfeld module,
see the proof of Theorem~\ref{tateyu}.
The word ``motive'' is used to emphasise the fact that this construction is contravariant.%


Let
$\nvaltau\colon\rtauF{\nKyu}^\times \to \bZ$ be the normalized $\tau^{-1}$-adic valuation. 

\begin{lem}\label{laurentord}%
Let $\nV$ be a formal Drinfeld motive.
Fix a nonzero element $\nv\in\nV$ 
and let
$\nphi\colon\ncoef\to\rtauF{\nKyu}$
be the map such that
$(\nphi(x)\otimes 1) \cdot \nv = (1\otimes x) \cdot \nv$ for all $x\in\ncoef$.
Then
the map $\ncoef^\times \to \bZ$, $x \mapsto \nvaltau(\nphi(x))$
is a positive integer multiple of the normalized $\nadic$-adic valuation
and is independent of the choice of $\nv$.%
%
\end{lem}

\begin{dfn}%
The number $\nht = \frac{1}{\nresdeg}\cdot \nvaltau(\nphi(z))$ is called the \emph{height} of $\nV$.
In a moment we shall see that it is always an integer.%
\end{dfn}


\begin{proof}[Proof of Lemma~\ref{laurentord}]%
Since the action of $\ncoef$ on $\nV$ is continuous in the $\tau^{-1}$-adic topology
it follows that $\nphi$ is continuous.
In particular
$\nphi(\ncomax^n)$ is contained in $\rtauO{\nKyu}\tau^{-1}$
for some $n \gg 0$.
This forces the $\tau^{-1}$-adic valuation of every nonzero element
of $\nphi(\ncomax)$ to be positive.

Let $\ncons$ be the algebraic closure of $\Fq$ in $\ncoef$.
The subgroup $\nvaltau(\phi(\ncons^\times)) \subset \bZ$
is finite and so is equal to $\{0\}$.
Hence $\nvaltau(\nphi(\ncoint^\times)) = \{0\}$.
We conclude that the map $x \mapsto \nvaltau(\nphi(x))$ is a positive
integer multiple of the normalized $\nadic$-adic valuation.
A different choice of $\nv$ changes $\nphi$ by conjugation.
This does not affect the value of $\nvaltau$, so the result follows.%
\end{proof}

\begin{prp}\label{laurentdm}%
Let $\nV$ be a formal Drinfeld motive of height $\nht$.%
\begin{enumerate}%
\item\label{laurentdm-struct}%
The $\nKyu\otimes_{\Fq} \ncoef$-module structure on $\nV$ extends uniquely
to a continuous $\nBRz{\nKyu}$-module structure.

\item\label{laurentdm-dm}%
Multiplication by $\tau$ makes $\nV$ into
a pure isocrystal of rank~$\nht$ and slope $-\frac{1}{\nht}$.
In particular the height $\nht$ is a positivie integer.%

\item\label{laurentdm-top}%
The $\nadic$-adic topology on $\nV$ coincides with the $\tau^{-1}$-adic topology.%
\end{enumerate}%
\end{prp}

\begin{proof}%
The unicity part of \eqref{laurentdm-struct}
follows since $\nKyu\otimes_{\Fq} \ncoef$ is dense in $\nBRz{\nKyu}$.
To prove the existence
pick a left $\rtauO{\nKyu}$-submodule $\naV\subset\nV$ of rank $1$.
Lemma~\ref{laurentord} implies that
for all $n \in \bZ$
we have $z^n\naV = \tau^{-n\nht\nresdeg}\naV$.
Hence $\naV$ is a $\nKyu\otimes_{\Fq}\ncoint$-module
which is complete with respect to the ideal $\nKyu\otimes_{\Fq} \ncomax$.
Now $\nBR{\nKyu}$ is the completion of $\nKyu\otimes_{\Fq}\ncoint$ at $\nKyu\otimes_{\Fq}\ncomax$
and so the $\nKyu\otimes_{\Fq}\ncoint$-module structure on $\naV$
extends uniquely to a structure of a continuous $\nBR{\nKyu}$-module.
The natural map $\naV[z^{-1}] \to \nV$ is a bijection since
$\nht > 0$.
Thus $\nV$ acquires a structure of a continuous $\nBRz{\nKyu}$-module
compatible with $\naV$.
We get the existence part of \eqref{laurentdm-struct}.

The quotient $\naV/z\naV$ is of finite dimension over $\nKyu$
and so is finitely generated over $\nBR{\nKyu}/z$.
The fact that $\naV = \mlim_{n>0} \naV/z^n\naV$
then implies that $\naV$ is finitely generated over $\nBR{\nKyu}$
\stacks{031D}.
As $\nBRz{\nKyu}$ is a finite product of fields it follows
that $\nV$ is a locally free $\nBRz{\nKyu}$-module of finite type
and $\naV$ is a lattice in $\nV$.

The identity $\tau \cdot (x \nv) = \nsigma(x) \cdot \tau \nv$ holds
for all $x \in \nKyu\otimes_{\Fq}\ncoef$ and $\nv \in \nV$.
Since the multiplication by $\tau$ is continuous and $\nKyu\otimes_{\Fq}\ncoef$
is dense in $\nBRz{\nKyu}$ we conclude that this identity holds
for every $x\in\nBRz{\nKyu}$.
The multiplication by $\tau$ is bijective and so $\nV$ is an isocrystal.
 
The $\nBRz{\nKyu}$-module
$\nV$ is free of finite rank by Corollary~\ref{bRrank}.
Now $\nBR{\nKyu}$ is a finite product of discrete valuation rings
and $\nBRz{\nKyu}$ is the product of the corresponding fraction fields.
It follows that the lattice $\naV$ is a free $\nBR{\nKyu}$-module.
As a consequence $\naV/z\naV$ is a free $\nBR{\nKyu}/z$-module.
Since
$\dim_K\nBR{\nKyu}/z = \nresdeg$
and $\dim_K\naV/z\naV = \nht\nresdeg$
we conclude that $\naV$ is a free $\nBR{\nKyu}$-module of rank $\nht$.
The equality
$\tau^{\nht\nresdeg}\naV = z^{-1}\naV$
shows that the isocrystal $\nV$ is pure of slope $-\frac{1}{\nht}$
and we get \eqref{laurentdm-dm}.
Property \eqref{laurentdm-top} follows since
$z^n\naV = \tau^{-n\nht\nresdeg}\naV$ as observed above.%
\end{proof}

\begin{prp}\label{laurentmod}%
Let $M$ be a pure isocrystal of rank $\nht > 0$ and slope $-\frac{1}{\nht}$.%
\begin{enumerate}
\item\label{laurentmod-alg}%
The left $\nKyu\{\tau\}$-module structure on $M$ extends uniquely
to a structure of a continous left $\rtauF{\nKyu}$-module.%

\item\label{laurentmod-mod}%
Together with the action of $\ncoef\subset \nBRz{\nKyu}$ this makes
$M$ into a formal Drinfeld motive of height $\nht$.%

\item\label{laurentmod-top}%
The $\tau^{-1}$-adic topology on $M$ coincides with the $\nadic$-adic
topology.%
\end{enumerate}%
\end{prp}%
\begin{proof}%
The field $\nKyu$ is perfect and so the multiplication by $\tau$ is bijective on $M$.
Hence the left $K\{\tau\}$-module structure on $M$ extends uniquely
to a left module structure over the subalgebra
of Laurent polynomials $\nKyu\{\tau,\tau^{-1}\} \subset \rtauF{\nKyu}$.
The unicity part of \eqref{laurentmod-alg} follows since
$\nKyu\{\tau,\tau^{-1}\}$ is dense in $\rtauF{\nKyu}$.
Let us prove the existence.
The isocrystal $M^*$ is pure of rank $\nht$ and slope~$\frac{1}{\nht}$.
Applying Proposition~\ref{embed} to $M^*$ and dualizing
we conclude that there is
an increasing family of lattices $\{\nBlat_n\}_{n\in\bZ}$ in $M$
such that 
\begin{enumerate}
\renewcommand{\theenumi}{\roman{enumi}}%
\item\label{laurentmor-shift}%
$\nBlat_{n+1} = \tau\nBlat_n$,

\item\label{laurentmor-slope}%
$\nBlat_{n+\nht\nresdeg} = z^{-1}\nBlat_n$.

\item\label{laurentmor-dim}%
$\dim_\nKyu \nBlat_n/\nBlat_{n-1} = 1$.%
\end{enumerate}
Set $\nBlat = \nBlat_0$.
Property~\eqref{laurentmor-shift} shows that $\nBlat$
is closed under the action of $\tau^{-1}$.
It follows from property~\eqref{laurentmor-slope} that
the lattices $\nBlat_{-n} = \tau^{-n}\nBlat$, $n\geqslant 0$,
form an $\nadic$-adic fundamental system.
Now $\nBlat$ is $\nadic$-adically complete
and so the structure of a left $\nKyu\{\tau^{-1}\}$-module
on it extends uniquely to a structure of 
a topological left $\rtauO{\nKyu}$-module.

It follows from Property \eqref{laurentmor-dim} that
$\dim_{\nKyu}\nBlat/\tau^{-n}\nBlat = n$ for each $n > 0$.
In particular there is a $\rtauO{\nKyu}$-linear morphism 
$f\colon \rtauO{\nKyu} \to \nBlat$ which reduces to an isomorphism modulo $\tau^{-1}$.
The quotient $\nBlat/\tau^{-n}\nBlat$ is finitely generated over $\rtauO{\nKyu}$
so noncommutative Nakayama's lemma implies that $f$ is surjective modulo $\tau^{-n}$.
Comparing the dimensions of the source and the target we conclude
that $f$ is an isomorphism modulo $\tau^{-n}$.
As $\nBlat$ is $\tau^{-1}$-adically complete it follows that $f$ is an isomorphism.
Hence $\nBlat$ is a free $\rtauO{\nKyu}$-module of rank $1$.

Property~\eqref{laurentmor-slope} implies that 
the multiplication by $z$ on $\nBlat$ coincides with
the multiplication by an element $\alpha\in\rtauO{\nKyu}$
of $\tau^{-1}$-adic valuation $\nht\nresdeg$.
As $\nBlat[z^{-1}] = M$ we conclude that the structure
of a left $\rtauO{\nKyu}$-module on $\nBlat$
extends uniquely to
a structure of a left $\rtauF{\nKyu}$-module on $M$.
We get the existence part of \eqref{laurentmod-alg}.


The action of the subfield $\ncoef\subset\nBRz{\nKyu}$ on $M$
commutes with the action of a dense subalgebra $\nKyu\{\tau,\tau^{-1}\}\subset\rtauF{\nKyu}$
and so with the whole $\rtauF{\nKyu}$.
Therefore the isocrystal $M$ acquires a structure of a formal Drinfeld motive.
This motive is of height $\nht$ since $\tau^{-\nht\nresdeg}\nBlat = z\nBlat$.
We thus obtain property~\eqref{laurentmod-mod}.
Property~\eqref{laurentmod-top} follows since
the $\tau^{-1}$-adic topology on $\nBlat$ coincides
with the $\nadic$-adic toplogy by construction.%
\end{proof}

\begin{dfn}%
Let $\nV$ be a formal Drinfeld motive.
We denote by $M(\nV)$ the isocrystal associated to $\nV$ by Proposition~\ref{laurentdm}.%
\end{dfn}

\begin{thm}\label{laurentmor}%
The construction $\nV \mapsto M(\nV)$ is a fully faithful
functor from the category of formal Drinfeld motives to the category of isocrystals.
An isocrystal of rank $r>0$ is in the essential image of this
functor if and only if it is pure of slope~$-\frac{1}{r}$.%
\end{thm}%

So a formal Drinfeld motive of height $\nht$
is just a different way to represent a pure $\nBRz{\nKyu}$-isocrystal
of rank $\nht$ and slope $-\frac{1}{\nht}$.
It~follows that the isomorphism classes of such isocrystals (of arbitrary rank $\nht$)
are in bijection with conjugacy classes of continuous homomorphisms $\ncoef\to\rtauF{\nKyu}$.

\begin{proof}[Proof of Theorem~\ref{laurentmor}]%
A morphism of formal Drinfeld motives
$f\colon \nV \to \nW$ 
is $\nKyu\otimes_{\Fq} \ncoef$-linear and continuous in the $\tau^{-1}$-adic topology.
As $\nKyu\otimes_{\Fq} \ncoef$ is
dense in $\nBRz{\nKyu}$ it follows that $f$ is $\nBRz{\nKyu}$-linear.
Moreover $f$ commutes with the left multiplication by $\tau$ and so
gives
a morphism of left $\nBRz{\nKyu}\{\tau\}$-modules $M(\nV) \to M(\nW)$.

Conversely let $\nV$ and $\nW$ be formal Drinfeld motives and
let $f\colon M(\nV) \to M(\nW)$ be a morphism of left $\nBRz{\nKyu}\{\tau\}$-modules.
The morphism $f$ is $\nKyu\{\tau,\tau^{-1}\}$-linear
and continuous in the $\nadic$-adic topology.
By Proposition~\ref{laurentdm} the $\nadic$-adic topology coincides with the
$\tau^{-1}$-adic topology.
Since $\nKyu\{\tau,\tau^{-1}\}\subset\rtauF{\nKyu}$ is dense
it follows that $f\colon \nV \to \nW$ is $\rtauF{\nKyu}$-linear.

Finally let $M$ be a pure isocrystal of rank $r> 0$ and slope $-\frac{1}{r}$.
Proposition~\ref{laurentmod} equips $M$ with a structure of a formal Drinfeld motive~$V$.
The unicity part of Proposition~\ref{laurentmod} implies that $M = M(V)$ and we get the result.%
\end{proof}

\begin{prp}\label{laurentiso}%
Let $\nV$ be a formal Drinfeld motive
and let $\nKyualt$ be a perfect field over $\nKyu$.
Set $\nV_\nKyualt = \rtauF{\nKyualt}\otimes_{\rtauF{\nKyu}} \nV$.
\begin{enumerate}
\item\label{laurentiso-welldef}
The inclusion $M(\nV) \hookrightarrow M(\nV_\nKyualt)$ is $\nBRz{\nKyu}$-linear.

\item\label{laurentiso-iso}
The induced map $\nBRz{\nKyualt}\otimes_{\nBRz{\nKyu}} M(\nV) \xrightarrow{\isosign} M(\nV_\nKyualt)$
is an isomorphism of $\nBRz{\nKyualt}$-isocrystals.%
\end{enumerate}%
\end{prp}%
\begin{proof}%
\eqref{laurentiso-welldef}
The inclusion is $\nKyu\otimes_{\Fq} F$-linear and $\tau^{-1}$-adically continuous.
The claim follows since the subring $\nKyu\otimes_{\Fq} F$ is dense in $\nBRz{\nKyu}$.
\eqref{laurentiso-iso}
Let $\nht$ be the height of $V$.
The height of $\nV_\nKyualt$ is $\nht$ by construction.
Proposition~\ref{laurentdm} shows that
the isocrystals $\nBRz{\nKyualt}\otimes_{\nBRz{\nKyu}} M(\nV)$
and $M(\nV_\nKyualt)$ are pure of slope $-\frac{1}{\nht}$
and rank $\nht$.
Hence they are simple by Proposition~\ref{simpleiso}.
The morphism in question is nonzero and so is an isomorphism.%
\end{proof}

\begin{prp}\label{laurentgal}%
Let $\nKyualt$ be a Galois extension of $\nKyu$ 
and let $G$ be the Galois group.
Equip $\nBRz{\nKyualt}$ with an action of $G$ by functoriality
and $\rtauF{\nKyualt}$ with the coefficientwise action.
Then 
for every formal Drinfeld motive $\nV$
the natural isomorphism
\begin{equation*}
\nBRz{\nKyualt}\otimes_{\nBRz{\nKyu}} M(\nV) \xrightarrow{\isosign} M(\rtauF{\nKyualt}\otimes_{\rtauF{\nKyu}} \nV)
\end{equation*}
is $G$-equivariant.%
\end{prp}%
\begin{proof}%
This is a formal consequence of Proposition~\ref{laurentiso}.%
\end{proof}

\subsection{Tate modules}
Fix an integer $\nht > 0$ and
an algebraic closure $\nKyual$ of $\nKyu$.
Let $\Fqbar$ be the algebraic closure of $\Fq$ in $\nKyual$.

\begin{dfn}\label{dfntauind}%
For each automorphism $\gamma$ of $\nKyual$ we denote by $\ntauind{\gamma}$
the induced automorphism of $\rtauF{\nKyual}$.%
\end{dfn}

Pick a formal Drinfeld motive $\nYuW$ of height $\nht$ over $\Fqbar$ and set $\nDN = \End \nYuW$.

\begin{dfn}\label{dfnformaltate}%
Let $\nYuV$ be a formal Drinfeld motive of height $\nht$ over $\nKyu$
and let
$\nYuVal = \rtauF{\nKyual}\otimes_{\rtauF{\nKyu}} \nYuV$.
The \emph{Tate module} of $\nYuV$ is the right $\nDN$-module
\begin{equation*}
\nTate{\nV} = \Hom_{\rtauF{\Fqbar}\otimes_{\Fq}\ncoef}(\nYuW, \;\nYuVal)
\end{equation*}
equipped with the following left action of the Weil group $W_K$:
%
For each $\gamma\in W_K$ and $f \in \nTate{\nYuV}$
we define $\gamma(f)$ by the formula $\nfnel \mapsto (\ntauind{\gamma}\otimes\textrm{id}_{\nYuV})(f(\tau^{\ord\gamma}\nfnel))$.%
\end{dfn}

\begin{lem}\label{formalN}%
The $\nBRz{\Fqbar}$-isocrystal $N = M(\nYuW)$ is pure and simple of slope $-\frac{1}{\nht}$.%
\end{lem}%
\begin{proof}%
Theorem~\ref{laurentmor} shows that $N$ is pure of rank $\nht$ and slope $-\frac{1}{\nht}$.
The fact that $N$ is simple then follows by Proposition~\ref{simpleiso}.%
\end{proof}

\begin{lem}\label{formalD}%
The ring $\nDN$ coincides with $\End N$ and is a central division algebra
of invariant $\frac{1}{\nht}$ over $\ncoef$.%
\end{lem}%
\begin{proof}%
Theorem~\ref{laurentmor} shows that $\nDN = \End N$. Since $N$ is pure and simple
the claim on the structure of $\nDN$ follows from Proposition~\ref{simple}.%
\end{proof}

In view of Lemma~\ref{formalN} the pair $(\nKyual, N)$ defines a Tate module functor
for pure $\nBRz{\nKyu}$-isocrystals of slope $-\frac{1}{\nht}$. 
The isocrystal $M = M(\nV)$ is pure of slope $-\frac{1}{\nht}$ by Theorem~\ref{laurentmor}.
So the Tate module of $M$ is
\begin{equation*}
\nTate{M} = \Hom_{\nBRz{\Fqbar}\{\tau\}}(N,\;\nBRz{\nKyual}\otimes_{\nBRz{\nKyu}} M).
\end{equation*}
Set $\nYuWal = \rtauF{\nKyual}\otimes_{\rtauF{\Fqbar}} \nYuW$.
Consider the composition of isomorphisms
\begin{equation*}
\begin{array}{rll}
\nTate{\nYuV}
&\xrightarrow{\isosign}\;\;\Hom(\nYuWal, \,\nYuVal)
&\textrm{(adjunction),} \\
&\xrightarrow{\isosign}\;\;\Hom(M(\nYuWal), \,M(\nYuVal))
&\textrm{(Theorem~\ref{laurentmor}),} \\
&\xrightarrow{\isosign}\;\;\Hom(\nBRz{\nKyual}\otimes_{\nBRz{\Fqbar}} N, \,\nBRz{\nKyual}\otimes_{\nBRz{\nKyu}} M)
&\textrm{(Proposition~\ref{laurentiso}),} \\
&\xrightarrow{\isosign}\;\;\nTate{M}
&\textrm{(adjunction).}
\end{array}
\end{equation*}

\begin{prp}\label{formalcomp}%
The resulting isomorphism $\nTate{\nYuV} \xrightarrow{\isosign} \nTate{M}$ is $W_K$-equivariant and $\nDN$-linear.%
\end{prp}%
\begin{proof}%
By Proposition~\ref{laurentgal}
the isomorphism
$\nBRz{\nKyual}\otimes_{\nBKz} M \xrightarrow{\isosign} M(\nYuVal)$
is $\Gal(\nKyual/\nKyu)$-equivariant.
Comparing the actions of $W_K$ of $\nTate{\nYuV}$ and $\nTate{M}$
we conclude that the isomorphism $\nTate{\nYuV}\xrightarrow{\isosign}\nTate{M}$ is $W_K$-equivariant.
The $\nDN$-linearity is clear.
\end{proof}


\begin{thm}\label{formaltate}%
The functor $\nYuV \mapsto \nTate{\nYuV}$ 
transforms formal Drinfeld motives of height $\nht$ 
to $(-\frac{1}{\nht})$-admissible $(W_K,\nDN)$-representations of rank~$1$ over $\nDN$.
This functor is an equivalence of categories if $W_K$ is dense in $G_K$.%
\end{thm}%
The meaning of the admissibility condition 
is clarified by Proposition~\ref{rankoneadm}.%
\begin{proof}[Proof of Theorem~\ref{formaltate}]%
Let $M$ be a pure $\nBRz{\nKyu}$-isocrystal of slope $\nslope = -\frac{1}{\nht}$.
By Proposition~\ref{taterank} the rank of $M$ is equal to $\nht$ if and only if $\dim_{\nDN} \nTate{M} = 1$.
In view of Proposition~\ref{formalcomp} the claim follows by
Theorems~\ref{laurentmor} and \ref{tate}.%
\end{proof}

Suppose that $W_K$ is dense in $G_K$ and
let $\nDval\colon D^\times \to \bQ$ be the unique
extension of the normalized valuation of $\ncoef$.
As an offshot of Theorem~\ref{formaltate} we get a natural bijection 
between the conjugacy classes of the following continuous homomorphisms:
\begin{itemize}
\item
$\rho\colon W_K \to D^\times$ such that
$\nDval(\rho(\gamma)) = -\frac{1}{\nht\nresdeg} \cdot \ord(\gamma)$ for all $\gamma \in W_K$,

\item
$\varphi\colon \ncoef \to \rtauF{K}$ such that
$\nvaltau(\varphi(x)) = \nht\nresdeg \cdot \nval(x)$
for all $x \in\ncoef$.
\end{itemize}
This describes a class of representations of $W_K$
purely in terms of $K$.

\subsection{Comparison}\label{ss:yucomp}
As before we fix $\nht > 0$,
an algebraic closure $\nKyual$ of $\nKyu$
and denote by $\Fqbar$ the algebraic closure of $\Fq$ in $\nKyual$.
Pick a continuous morphism $\iota\colon \ncoef \to \rtauF{\Fqbar}$
such that $\nvaltau(\iota(z)) = \nht\nresdeg$.
Let $\nCentN$ be the centralizer of $\iota(\ncoef)$ in $\rtauF{\Fqbar}$.
We shall see in a moment that $\nCentN$ is a central division algebra
of invariant $-\frac{1}{r}$ over $\ncoef$.

Let $\nphi\colon\ncoef\to\rtauF{\nKyu}$ be a continuous morphism
such that $\nvaltau(\nphi(z)) = \nht\nresdeg$.
Following J.-K.~Yu \cite[Construction~2.5]{yu} we consider the set
\begin{equation*}
Y_\iota = \{ \nyel \in \rtauF{\nKyual}^\times \mid \textrm{ for all } x \in \ncoef \textrm{ one has } \iota(x) = \nyel \nphi(x) \nyel^{-1} \}
\end{equation*}
The Weil group $W_K$ acts on $Y_\iota$ by the rule
$\gamma \cdot \nyel = \tau^{\ord\gamma} \ntauind{\gamma}(\nyel)$
where $\ntauind{\gamma}$ is as in Definition~\ref{dfntauind}.
The centralizer $\nCentN^\times$ acts on $Y_\iota$ by the left multiplication.%

Let $\nYuW$ be the formal Drinfeld motive over $\Fqbar$ defined by $\iota$.
As before we denote by $\nDN$ the ring $\End\nYuW$. 

\begin{lem}%
The map
$\nCentN^\nop \shortisosign \nDN$,
$\neCentN \mapsto (\nfnel \mapsto \nfnel\neCentN)$
is an isomorphism.\qed%
\end{lem}
In the following we identify $\nCentN^\nop$ with $\nDN$ via this map.
Lemma~\ref{formalD} implies that $\nCentN$ is a central division algebra of invariant $-\frac{1}{\nht}$ over $\ncoef$.
Consider the formal Drinfeld motive $\nYuV$ defined by $\nphi$.
The Tate module of $\nYuV$ is
\begin{equation*}
\nTate{\nYuV} = \Hom_{\rtauF{\Fqbar}\otimes_{\Fq} \ncoef}(\nYuW, \:\nYuVal)
\end{equation*}
where $\nYuVal = \rtauF{\nKyual}\otimes_{\rtauF{\nKyu}} \nV$.
Theorem~\ref{formaltate} shows that
$\dim_{\nDN} \nTate{\nYuV} = 1$.

\begin{thm}\label{tateyu-formal}%
The map $Y_\iota \xrightarrow{\isosign} \nTate{\nYuV}\backslash\{0\}$,
$\nyel \mapsto (\nfnel \mapsto \nfnel\nyel)$,
is an isomorphism of $W_K$-sets
and transforms the left action of $\nCentN^\times$ 
to the right action of $\nDN^\times = (\nCentN^{\textup{op}})^\times$.%
\end{thm}%
\begin{proof}%
This map is bijective since $\dim_{\rtauF{\Fqbar}} \nYuW = 1$
and is compatibile with the actions of $\nCentN^\times$ and $\nDN^\times$
by construction.
It remains to check the $W_K$-equivariance.
Pick an automorphism $\gamma \in W_K$ and set $n = \ord\gamma$.
Let $f$ be an element of $\nTate{\nYuV}$ corresponding to $u\in Y_\iota$.
By Definition~\ref{dfntate} the morphism $\gamma(f)\colon \nYuW \to \nYuVal$
is given by the formula
$\nfnel \mapsto \ntauind{\gamma}(\tau^n\nfnel\nyel)$.
Now $\gamma$ acts as the $q^{-n}$-Frobenius on $\Fqbar$ and so
\begin{equation*}
\ntauind{\gamma}(\tau^n \nfnel \nyel) = \tau^n \ntauind{\gamma}(\nfnel) \ntauind{\gamma}(\nyel) = \nfnel \tau^n \ntauind{\gamma}(\nyel).
\end{equation*}
Hence $\gamma(f)$ corresponds to the element $\tau^n \ntauind{\gamma}(\nyel) = \gamma \cdot \nyel$ of $Y_\iota$.%
\end{proof}

\section{Isocrystals over a discrete valuation ring}\label{sec:isocdvr}
Let $K$ be a field over $\Fq$ equipped with a nontrivial 
valuation $\nval\colon K^\times \to \bZ$.
Let $\nOK$ be the ring of integers of $\nval$ and let $\nresf$ be the residue field.
We set
\begin{equation*}
\nBl = K\otimes_{\nOK} \nB, \quad
\nBzl = K\otimes_{\nOK} \nBz.
\end{equation*}
The endomorphisms $\nsigma$ of $\nB$, $\nBz$ extend uniquely to $\nBl$, $\nBzl$.

\begin{dfn}%
A \emph{$\nBzl$-isocrystal} is a $\nsigma$-module over $\nBzl$.%
\end{dfn}

\subsection{Base change results}

In \cite{dejong} de~Jong demonstrated
that the base change functor
from $F$-isocrystals over $\nOK$ to $F$-isocrystals over $K$
is fully faithful.
He observed that this functor factorizes through an intermediate category, that of overconvergent $F$-isocrystals,
and proved the full faithfulness of the first resulting functor.
It was later shown by Kedlaya that the second functor is also fully faithful~\cite[Theorem~5.1]{kedlaya-ff}.

Our aim is to prove a related result for $\nBz$-isocrystals.
We consider the following factorization of
the base change functor:
\begin{equation*}
\nBz\textrm{-isocrystals} \Rightarrow
\nBzl\textrm{-isocrystals} \Rightarrow
\nBKz\textrm{-isocrystals}.
\end{equation*}
We show that the first functor is fully faithful
and that the second one is fully faithful
on $\nBzl$-isocrystals $M$ such that $\nBKz\otimes_{\nBzl} M$ is pure.
The latter result will be useful on its own,
see Propositions~\ref{modcmp}, \ref{metagood} and Lemma~\ref{main}.
Example~\ref{bzlmixed} shows that for this result the assumption of purity is essential.
%

%

In his PhD thesis Watson proved that
the base change functor from $\nBz$-isocrystals
to $\nBKz$-isocrystals is fully
faithful even in the mixed case 
\cite[Theorem~4.1, p.~15]{watson}.
For our purposes we shall rely on a weak form
of Watson's theorem provided by Proposition~\ref{dmloc}.

Let us begin with the base change functor $\nBzl \Rightarrow \nBKz$.
As the first step we introduce the notion of purity for $\nBzl$-isocrystals.

\begin{dfn}\label{defabunlat}%
Let $M$ be a locally free $\nBzl$-module of finite type.
A \emph{lattice} $\nBlat$ in $M$ is a locally free $\nBl$-submodule
of finite type which spans $M$ over $\nBzl$.%
\end{dfn}

\begin{dfn}%
A $\nBzl$-isocrystal $M$ is \emph{pure}
if there are integers $s$ and $r$, $r > 0$, and a lattice $\nBlat$ in $M$ such that 
$z^s\nBlat$ is generated by $\tau^{r\nresdeg}\nBlat$ over $\nBl$.
The rational number $\frac{s}{r}$ is called the \emph{slope} of $M$.%
\end{dfn}

%

\begin{lem}\label{abunlatshift}%
Let $M$ be a $\nBzl$-isocrystal and $\nBlat$ a lattice in $M$.
Then for every $n\geqslant 0$ the $\nBl$-submodule of $M$ generated
by $\tau^n \nBlat$ is a lattice.%
\end{lem}%
\begin{proof}%
Same as the proof of Lemma~\ref{latshift}.%
\end{proof}

\begin{lem}\label{abunlat}%
Let $M$ be a locally free $\nBzl$-module of finite type and
$\nBKlat$ a lattice in $\nBKz\otimes_{\nBzl} M$.
Then $\nBlat = \nBKlat \cap M$
is the unique lattice in $M$
which generates $\nBKlat$ over $\nBK$.%
\end{lem}

\begin{proof}
The completion of the ring $\nBl$ at the ideal $z\nBl$ is $\nBK$
and $\nBl[z^{-1}] = \nBzl$.
Applying Beauville-Laszlo glueing theorem \stacks{0BP2} to the ring $\nBl$ and the element $z$
we conclude that $\nBlat$ is
the unique $\nBl$-submodule which generates $M$ over $\nBzl$ and $\nBKlat$ over $\nBK$.
Since the schemes $\Spec\nBzl$ and $\Spec\nBK$ form an fpqc covering of $\Spec\nBl$
it follows that $\nBlat$ is a locally free $\nBl$-module of finite type.%
\end{proof}

%
\begin{prp}\label{abunpure}%
Let $\nslope$ be a rational number.
For every $\nBzl$-isocrystal $M$ the following are equivalent:
\begin{enumerate}
\item
The $\nBzl$-isocrystal $M$ is pure of slope $\nslope$.

\item
The $\nBKz$-isocrystal $M_K = \nBKz\otimes_{\nBzl} M$ is pure of slope $\nslope$.%
\end{enumerate}%
\end{prp}%
\begin{proof}%
The implication (1) $\Rightarrow$ (2) follows since the base change
of a lattice $\nBlat \subset M$ to $\nBK$ is a lattice in $M_K$.
Let us prove the implication (2) $\Rightarrow$ (1).
By assumption there are integers $s$ and $r$, $r>0$, and a lattice $\nBKlat$ in $M_K$
such that $\nslope = \frac{s}{r}$ and $z^s \nBKlat = \nlati{r\nresdeg}{\nBKlat}$.
The intersection $\nBlat = \nBKlat \cap M$ is a lattice in $M$ by Lemma~\ref{abunlat}.
Let $\nBlat'$ be the $\nBl$-submodule of $M$ generated by $\tau^{r\nresdeg}{\nBlat}$.
This is a lattice in $M$ which generates $z^s\nBKlat$ over $\nBK$.
The unicity part of Lemma~\ref{abunlat} implies that $z^s \nBlat = \nBlat'$.
Hence $M$ is pure of slope~$\nslope$.%
\end{proof}

\begin{prp}\label{abunpure0}
For every $\nBzl$-isocrystal $M$ the following are equivalent:
\begin{enumerate}
\item
The isocrystal $M$ arises from a $\nsigma$-module over $\nBl$.

\item
The isocrystal $M$ is pure of slope~$0$.%
\end{enumerate}%
\end{prp}%
\begin{proof}%
Assume that $M$ is pure of slope~$0$.
By Proposition~\ref{fdmpure0} 
the $\nBKz$-isocrystal $M_K = \nBKz\otimes_{\nBzl} M$
arises from a $\nsigma$-module $\nBKlat$ over $\nBK$.
The intersection $\nBlat = \nBKlat \cap M$ is a lattice in $M$ by Lemma~\ref{abunlat}.
Let $\nBlat'$ be the $\nBl$-submodule of $M$ generated by $\tau \nBlat$.
This is a lattice in $M$ which generates $\nBKlat$ over $\nBK$.
The unicity part of Lemma~\ref{abunlat} implies that $\nBlat = \nBlat'$.
In other words $\nBlat$ is generated by $\tau\nBlat$ and so is a $\nsigma$-module
over $\nBl$.%
\end{proof}

We shall deduce the base change result for $\nBzl$-isocrystals from a sequence of lemmas.
Let $\nsigma$ be the $q$-Frobenius endomorphism of $\nOK$ and $K$.

\begin{lem}\label{fundestim}%
For each left $\nOK\{\tau\}$-module $M$
which is finitely generated free over~$\nOK$
we have
$\Hom(M, \,\unitob{\nOK}) = \Hom(K\otimes_{\nOK}M, \:\unitob{K})$.%
\end{lem}

\begin{proof}%
Pick an $\nOK$-basis $m_1, \dotsc, m_n$ of $M$.
Let $I$ be the ideal of the ring $S = \nOK[x_1,\dotsc,x_n]$ generated by the elements
\begin{equation*}
x_i^q - \sum_{j=1}^n u_{ij} x_j, \quad
i \in \{1,\dotsc,n\}
\end{equation*}
where the coefficients $u_{ij} \in \nOK$ are defined by the formulas
$\tau(m_i) = \sum_{j=1}^n u_{ij} m_j$.
By construction
\begin{align*}
\Hom_{\nOK\textup{-alg}}(S/I, \,K) &= \Hom(K\otimes_{\nOK}M, \,\unitob{K}),\\
\Hom_{\nOK\textup{-alg}}(S/I, \,\nOK) &= \Hom(M, \,\unitob{\nOK}).
\end{align*}
The result follows since the ring homomorphism $\nOK \to S/I$ is finite.%
\end{proof}

\pagebreak
\begin{lem}\label{fundestimz}%
For every left $\nB\{\tau\}$-module $\nBlat$
which is locally free of finite type over $\nB$
we have 
$\Hom(\nBlat, \,\unitob{\nB}) = \Hom(\nBlat, \:\unitob{\nBK})$.%
\end{lem}%
\begin{proof}%
Lemma~\ref{fundestim} implies that for all $n$ the natural map
\begin{equation*}
\Hom(\nBlat, \:\unitob{\nB}/\ncomax^n\unitob{\nB}) \xrightarrow{\isosign}
\Hom(\nBlat, \:\unitob{\nBK}/\ncomax^n\unitob{\nBK})
\end{equation*}
is a bijection.
The result follows since $\unitob{\nB}$ and $\unitob{\nBK}$
are $\nadic$-adically complete.%
\end{proof}

\begin{lem}\label{pid}%
The ring $\nBl$ is a finite product of principal ideal domains.%
\end{lem}%
\begin{proof}%
The ring $\nB$ is a product of finitely many $\nOK$-algebras $A_i$
which are regular local rings of dimension~$2$.
Every localization $\nboun{A}_i = K\otimes_{\nOK} A_i$
is a Dedekind domain since it is regular of dimension $1$.
Thus every ideal of $\nboun{A}_i$ decomposes into a product of prime ideals.
The nonzero prime ideals of $\nboun{A}_i$ come from the prime ideals of $A_i$ of height~$1$.
We claim that all such ideals are principal.
Indeed, the regular local ring $A_i$ is a unique factorization domain \stacks{0AG0},
and the prime ideals of height~$1$ in such a ring are principal \stacks{0AFT}.%
\end{proof}

\begin{lem}\label{fundestim3}%
For every $\nsigma$-module $\nBlat$ over $\nBl$ we have 
$\nBlat^{\tau} = (\nBK\otimes_{\nBl} \nBlat)^{\tau}$.%
\end{lem}%
\begin{proof}%
It follows from Lemma~\ref{pid} that there is a locally free $\nB$-submodule of finite type
$P \subset \nBlat^*$ which generates $\nBlat^*$ over $\nBl$.
Multiplying $P$ by a high enough power of a uniformizer $\zeta\in\nOK$
we reduce to the situation where $P$ is a left $\nB\{\tau\}$-submodule. 
Consider the diagram of natural inclusions
\begin{equation*}
\xymatrix{
\Hom(\nBlat^*, \:\unitob{\nBl}) \ar[r]
& \Hom(\nBlat^*, \:\unitob{\nBK}) \\
\Hom(P, \:\unitob{\nB}) \ar[u] \ar[r]
& \Hom(P, \:\unitob{\nBK}) \ar[u].
}
\end{equation*}
The right inclusion is an equality since $P$ generates $\nBlat^*$ over $\nBl$.
The bottom inclusion is an equality by
Lemma~\ref{fundestimz}.
We therefore have a natural bijection
\begin{equation*}
\Hom(\nBlat^*, \,\unitob{\nBl}) \xrightarrow{\:\isosign\:}
\Hom(\nBlat^*, \,\unitob{\nBK}).
\end{equation*}
Lemma~\ref{dmhominv} identifies
the left hand side with
$(\nBlat^{**})^{\tau}$ and the right hand side with
$(\nBK \otimes_{\nBl} \nBlat^{**})^{\tau}$.
We get the result since
$\nBlat \cong \nBlat^{**}$ by Lemma~\ref{doubledual}.%
\end{proof}


\begin{thm}\label{bzlff}%
The base change functor $M \mapsto \nBKz\otimes_{\nBzl} M$
from the category of pure $\nBzl$-isocrystals
to the category of $\nBKz$-isocrystals is fully faithful.%
\end{thm}%
\begin{proof}
Let $M$, $N$ be pure $\nBzl$-isocrystals and let $H = \iHom(M,N)$.
We denote by $H_K$ the isocrystal $\nBKz\otimes_{\nBzl} H$.
In view of Lemma~\ref{dmhominv}
it is enough to prove that the inclusion $H^\tau \subset (H_K)^\tau$ is an equality.
If the slopes of $M$ and $N$ are different then $(H_K)^\tau = 0$ by Proposition~\ref{slopehom}.
Assume that the slopes of $M$ and $N$ coincide.
Then $H$ is pure of slope~$0$ and so arises from a $\nsigma$-module
$\nBlat$ over $\nBl$ by Proposition~\ref{abunpure0}.
Lemma~\ref{fundestim3} shows that
$\nBlat^\tau = (\nBK\otimes_{\nBl}\nBlat)^\tau$
and the result follows.%
\end{proof}

\begin{exa}\label{bzlmixed}%
This result does not extend to mixed $\nBzl$-isocrystals.
Suppose that $\ncoef = \rF{\Fq}$.
Pick an element $\alpha \in K^\times$ such that $\nval(\alpha) < 0$.
Let $M$ be an isocrystal with a $\nBzl$-basis $e_0$, $e_1$
on which $\tau$ acts as follows:
\begin{equation*}
\tau(e_0) = z e_0, \quad \tau(e_1) = \alpha e_0 + e_1.
\end{equation*}
This isocrystal is not pure since the $\nBKz$-isocrystal
$\bm\hat{M} = \nBKz\otimes_{\nBzl} M$
has a sub-isocrystal $\nBKz e_0$ of slope $1$
and the quotient is of slope~$0$.

We claim that the inclusion
$M^\tau \subset \bm\hat{M}^\tau$
is proper.
The equation $x = \alpha + z \sigma(x)$ has a solution
\begin{equation*}
x = \sum_{n\geqslant 0} \alpha^{q^n} z^n
\end{equation*}
in $\nBKz$.
Hence the vector $e = x e_0 + e_1$ belongs to $\bm\hat{M}^\tau$.
However the valuation of the coefficients of $x$ is not bounded from below and so $e \not\in M$.
\end{exa}

\begin{rmk}%
Let $\nBz^\dagger \subset \rF{K}$ be the subring of series
which converge on a nonempty punctured open disk around~$0$.
This is the $z$-adic counterpart of the ring $\Gamma_{\textup{con}}[p^{-1}]$ from the theory
of overconvergent $F$-isocrystals \cite[~\S2]{kedlaya-ff}.
By analogy with Kedlaya's theorem \cite{kedlaya-ff} one may
ask whether the base change functor from $\nBz^\dagger$-isocrystals
to $\nBKz$-isocrystals is fully faithful.
The example above shows that the answer to this question is negative:
The series $x = \sum_{n\geqslant 0} \alpha^{q^n} z^n$
has radius of convergence~$0$ whenever $\nval(\alpha) < 0$.
\end{rmk}

Finally we shall prove that the base change functor
from $\nBz$-isocrystals to $\nBzl$-isocrystals is fully faithful
and combine this result with the theorem above.

\begin{prp}\label{bzff}%
The functor $M\mapsto \nBzl\otimes_{\nBz} M$
from the category of $\nBz$-isocrystals
to the category of $\nBzl$-isocrystals
is fully faithful.%
\end{prp}%
\begin{proof}%
Let $M$ and $N$ be $\nBz$-isocrystals and set $H = \iHom(M,\,N)$.
According to Lemma~\ref{dmhominv} we have
\begin{equation*}
H^\tau = \Hom(M,\,N), \quad
(\nMzl{H})^\tau =
\Hom(\nMzl{M},\:\nMzl{N}).
\end{equation*}
We shall prove that $H^\tau = (\nMzl{H})^\tau$.
Pick an element $x \in (\nMzl{H})^\tau$.
Let $\zeta\in K$ be a uniformizer
and $n\geqslant 0$ an integer such that $\zeta^n x \in H$.
We have
\begin{equation*}
\tau(\zeta^n x) = \zeta^{q n} x = \zeta^{(q-1)n}\cdot \zeta^n x.
\end{equation*}
So if $n > 0$ then
$\tau$ acts by zero on the image of $\zeta^n x$
in $\nBRz{\nresf}\otimes_{\nBz} H$.
As $\nsigma$ is injective on $\nBRz{\nresf}$
it follows that
$\zeta^n x$ belongs to the kernel of the reduction map
$H \to \nBRz{\nresf}\otimes_{\nBz} H$.
This kernel is $\zeta H$ since
$\nBRz{\nresf} = \nBz/\zeta\nBz$.
Dividing by $\zeta$
we conclude that $\zeta^{n-1} x \in H$.
An induction over $n$ implies that $x \in H$.%
\end{proof}
\begin{prp}\label{dmloc}%
The functor $M \mapsto \nBKz\otimes_{\nBz} M$
from the category of pure $\nBz$\nobreakdash-isocrystals 
to the category of $\nBKz$-isocrystals is fully faithful.%
\end{prp}%
\begin{proof}%
Combine Proposition~\ref{bzff} and Theorem~\ref{bzlff}.%
\end{proof}

\subsection{Pure isocrystals}

\begin{lem}\label{bunlat}%
Let $M$ be a locally free $\nBz$-module of finite type and
$\nBllat$ a lattice in $\nBzl\otimes_{\nBz} M$.
Then $\nBlat = \nBllat \cap M$ is the unique lattice in $M$
which generates $\nBllat$ over~$\nBl$.%
\end{lem}%
\begin{proof}%
The ring $\nB$ has the following properties:
\begin{itemize}
\item $\nB$ is a finite product regular local rings of dimension~$2$.

\item
The set of maximal ideals of $\nB$ is
the complement of $\Spec\nBl \cup \Spec\nBz$ in $\Spec\nB$.%
\end{itemize}
Applying \cite[Proposition~2.1~(b)]{pink-isogeny} we conclude that
$\nBlat = \nBllat \cap M$ is a locally free $\nB$-module of finite type
and that it spans $M$ over $\nBz$ and $\nBllat$ over $\nBl$.
The unicity of $\nBlat$ follows since
the restriction functor from the category of quasi-coherent sheaves
on $\Spec\nB$ to the category of quasi-coherent sheaves on $\Spec\nBl \cup \Spec\nBz$
is fully faithful on the subcategory of locally free sheaves of finite type.%
\end{proof}


\pagebreak
\begin{prp}\label{dmredpure}%
Let $\nslope$ be a rational number.
For every $\nBz$-isocrystal $M$ the following are equivalent:
\begin{enumerate}
\item
The isocrystal $M$ is pure of slope $\nslope$.

\item
The isocrystal $M_K = \nBKz\otimes_{\nBz} M$ is pure of slope~$\nslope$.
\end{enumerate}
\end{prp}%
The corresponding result for $F$-isocrystals 
follows from Grothendieck's specialization theorem for Newton polygons
\cite[Theorem~2.3.1]{katz-slopes}.%
\begin{proof}[Proof of Proposition~\ref{dmredpure}]%
Assume that $M_K$ is pure of slope $\nslope$.
Proposition~\ref{abunpure} shows that the $\nBzl$-isocrystal $\nMzl{M}$
is pure of slope $\nslope$.
So there are integers $s$ and $r$, $r > 0$,
and a lattice $\nBllat$ in $\nMzl{M}$
such that $\frac{s}{r} = \nslope$ and $z^s \nBllat$ is generated by $\tau^{r\nresdeg}\nBllat$.
The intersection $\nBlat = \nBllat \cap M$ is a lattice by Lemma~\ref{bunlat}.
The unicity part of this lemma implies that 
$z^s\nBlat = \nlati{r\nresdeg}{\nBlat}$
and the claim follows.%
\end{proof}

\begin{prp}\label{dmpure0}%
For every $\nBz$-isocrystal $M$
the following are equivalent:
\begin{enumerate}
\item The isocrystal $M$ arises from a $\nsigma$-module over $\nB$.

\item The isocrystal $M$ is pure of slope~$0$.%
\end{enumerate}
\end{prp}%
\begin{proof}%
Assume that the isocrystal $M$ is pure of slope~$0$.
Proposition~\ref{abunpure0} shows that the $\nBzl$-isocrystal $\nMzl{M}$
arises from a $\nsigma$-module $\nBllat$ over $\nBl$.
The intersection $\nBlat = \nBllat \cap M$ is a lattice by Lemma~\ref{bunlat}.
The unicity part of this lemma shows that $\nBlat = \nlati{}{\nBlat}$.
Hence $\nBlat$ is a $\nsigma$-module over $\nB$.%
\end{proof}

Let $\nresf$ be the residue field of $\nOK$.
\begin{lem}\label{dmpurepre}
If $\nOK$ is henselian then the functor
$M \mapsto \nBRz{\nresf}\otimes_{\nBz} M$
is an equivalence of categories of pure $\nBz$-isocrystals of slope~$0$
and pure $\nBRz{\nresf}$-isocrystals of slope~$0$.%
\end{lem}%
\begin{proof}%
As $\nOK$ is henselian
the pullback functor from lisse $\nadic$-adic sheaves on $\etsite{\nOK}$
to lisse $\nadic$-adic sheaves on $\etsite{\nresf}$ is an equivalence of categories.
So Theorem~\ref{wellknown} shows that the functor
$\nBlat\mapsto \nBR{\nresf}\otimes_{\nB} \nBlat$
is an equivalence of categories of $\nsigma$-modules over $\nB$ and
$\nsigma$-modules over $\nBR{\nresf}$.

By Proposition~\ref{fdmpure0} every pure $\nBRz{\nresf}$-isocrystal of slope~$0$
arises from a $\nsigma$-module over $\nBR{\nresf}$.
We thus get essential surjectivity.
It remains to prove full faithfulness.
Let $M$, $N$ be pure $\nBz$-isocrystals of slope~$0$.
Then $H = \iHom(M,N)$ is also pure of slope~$0$
and so arises from a $\nsigma$-module $\nBlat$ over $\nB$ by Proposition~\ref{dmpure0}.
The claim follows since the reduction map $\nBlat^\tau \to (\nBR{\nresf}\otimes_{\nB} \nBlat)^\tau$ is an isomorphism.%
\end{proof}

The following result is related to a constancy theorem of Katz for $F$-isocrystals \cite[Theorem~2.7.1]{katz-slopes}.%
\begin{thm}\label{dmpure}%
If $\nOK$ is henselian
then the functor $M \mapsto \nBRz{\nresf}\otimes_{\nBz} M$
from pure $\nBz$-isocrystals to pure $\nBRz{\nresf}$-isocrystals
is fully faithful.%
\end{thm}%
\begin{proof}
%
Let $M$, $N$ be pure $\nBz$-isocrystals.
If the slopes of $M$, $N$ are different then
Proposition~\ref{slopehom} shows that
\begin{equation*}
\Hom(M, \:N) = 0, \quad
\Hom(\nBRz{\nresf}\otimes_{\nBz} M, \:\nBRz{\nresf}\otimes_{\nBz} N) = 0
\end{equation*}
If $M$ and $N$ have the same slope
then the isocrystal $H = \iHom(N,M)$ is pure of slope~$0$.
The reduction map $H^\tau \to (\nBRz{\nresf}\otimes_{\nBz} H)^\tau$
is an isomorphism by Lemma~\ref{dmpurepre}. So we get the result.%
%
%
%
\end{proof}
\begin{exa}\label{exa-dmpure}%
The assumption of purity in Theorem~\ref{dmpure} is essential. 
Suppose that $\ncoef = \rF{\Fq}$ and $K = \Fq(\!(\zeta)\!)$.
Let $\nval\colon K^\times \to \bZ$ be the $\zeta$-adic valuation.
Then $\nBz = \rF{\Fq[[\zeta]]}$.
Let $M$ be an isocrystal with a $\nBz$-basis $e_0, e_1$ on which $\tau$
acts as follows:
\begin{equation*}
\tau(e_0) = e_0, \quad
\tau(e_1) = z e_1 -\zeta e_0.
\end{equation*}
The isocrystal $M$ is mixed and has slopes $0$ and $1$.
Let $i\colon \unit \to M$ be the embedding which sends $1$ to $e_0$.
This clearly splits modulo $\zeta$. However $i$ is not split over $\nBz$.

Indeed let $s\colon M \to \unit$ be a splitting and set $x = s(e_1)$.
The identity $s(\tau e_1) = \tau s(e_1)$ implies that $x$ satisfies
the equation $\nsigma(x) = z x -\zeta$. Hence the coefficients of $x$
at negative powers of $z$ are zero
and the coefficient at $z^0$ is a $q$-th root of $-\zeta$ which
does not exist in $K$.%
\end{exa}

\subsection{Models}\label{ss:red}


\begin{dfn}%
Let $M$ be a $\nBKz$-isocrystal
and let $\nVarB$ be one of the rings $\nBz$, $\nBzl$.
A \emph{$\nVarB$-model of $M$} is
is a left $\nVarB\{\tau\}$-submodule $\roM\subset M$ such that
\begin{itemize}
\item $\roM$ is a $\nVarB$-isocrystal,

\item the natural map $\nBKz\otimes_{\nVarB}\roM \to M$ is an isomorphism.%
\end{itemize}%
We say that $M$ \emph{has good reduction over $\nOK$} if it admits a $\nBz$-model.
\end{dfn}
\begin{conv}%
We do not mention $\nOK$ when its choice is clear from the context.%
\end{conv}

%
%

\begin{exa}%
The $\infty$-adic isocrystal coming from a Drinfeld module $E$
has a $\nBzl$-model
whenever the residual characteristic of $E$ is finite,
see Proposition~\ref{mloctaubun}.
\end{exa}

%

Models of pure isocrystals have a N\'eron property:

\begin{prp}\label{dmmoduniq}%
Every pure isocrystal admits at most one $\nBz$-model.
If pure isocrystals $M$, $N$ have good reduction then every morphism
$M \to N$ arises from a unique morphism of their $\nBz$-models.%
\end{prp}%
\begin{proof}
A $\nBz$-model is a pure $\nBz$-isocrystal by Proposition~\ref{dmredpure}
so the claim follows from Proposition~\ref{dmloc}.%
\end{proof}


\begin{prp}\label{motmoduniq}%
Every pure isocrystal admits at most one $\nBzl$-model.
If pure isocrystals $M$, $N$ 
have $\nBzl$-models
then every morphism $M \to N$
arises from a unique morphism of their $\nBzl$-models.%
\end{prp}%
\begin{proof}%
A $\nBzl$-model is a pure $\nBzl$-isocrystal by Proposition~\ref{abunpure}
so the claim follows from Theorem~\ref{bzlff}.%
\end{proof}

The next result shows that the two types of models are compatible.
This will be used in the proof of the good reduction criterion, see Proposition~\ref{metagood}.

\begin{prp}\label{modcmp}%
Let $M$ be a pure isocrystal 
which admits a $\nBzl$-model $\nMloc{M}$.
Suppose that $M$ also admits a $\nBz$-model $\roM$.
Then $\nMloc{M} = K\otimes_{\nOK} \roM$.%
\end{prp}%
\begin{proof}%
Indeed $K\otimes_{\nOK}\roM$ is a $\nBzl$-model of $M$
so the claim follows by unicity of $\nBzl$-models
(Proposition~\ref{motmoduniq}).%
\end{proof}

The main result of this section is that
the reduction behaviour of pure isocrystals is controlled
by their Tate modules:

\begin{thm}\label{tatered}%
For every pure $\nBKz$-isocrystal $M$ the following are equivalent:
\begin{enumerate}%
\item\label{tatered-goodred}%
The isocrystal $M$ has good reduction over $\nOK$.

\item\label{tatered-unram}%
The Tate module $\nTate{M}$ is unramified at the valuation $\nval\colon K^\times \to \bZ$.
\end{enumerate}%
\end{thm}

A convenient feature of this theorem is that it works for every discrete valuation:
The ring of integers $\nOK$ need not be $v$-adically
complete and its residue field may be infinite.
%
We shall deduce Theorem~\ref{tatered} from a few intermediate results.%
\begin{lem}\label{fundsurj}%
Let $\xbar{x}$ be a geometric point of $\Spec K$.
Then the morphism of \'etale fundamental groups
$\pi_1(\Spec K, \xbar{x}) \to \pi_1(\Spec\nOK, \xbar{x})$
is surjective and its kernel is the normal closure of any inertia
subgroup at the valuation $\nval\colon K^\times \to \bZ$.%
\end{lem}%
\begin{proof}%
See \cite[Corollary~6.17]{lenstra} or \stacks{0BQM}.%
\end{proof}

\begin{lem}\label{tatered0}
For every pure $\nBKz$-isocrystal $M$ of slope~$0$ the following are equivalent:%
\begin{enumerate}%
\item\label{tatered0-mod}%
The isocrystal $M$ has good reduction over $\nOK$.

\item\label{tatered0-unram}%
The Tate module $\nTate{M}$ is unramified at the valuation $\nval\colon K^\times \to \bZ$.
\end{enumerate}
\end{lem}%
\begin{proof}%
Let $K^s$ be a separable closure of $K$ and set $G_K = \Gal(K^s/K)$.
The Tate module of $M$ is
$\nTate{M} = (\nBRz{K^s} \otimes_{\nBKz} M)^\tau$
on which the Weil group $W_K$ acts via $\nBRz{K^s}$.
By Proposition~\ref{fdmpure0} the isocrystal $M$ arises
from a $\nsigma$-module $\nBlat$ over $\nBK$.
The $\ncoint$-module
$(\nBR{K^s}\otimes_{\nBK} \nBlat)^\tau$ is a lattice in $\nTate{M}$.
Since each inertia subgroup of $G_K$ at $\nval$ is contained in $W_K$
we conclude that $\nTate{M}$ is unramified at $\nval$ if and only if
$(\nBR{K^s}\otimes_{\nBK} \nBlat)^\tau$ is unramified at $\nval$.

Set $\xbar{x} = \Spec K^s$.
Lemma~\ref{fundsurj} implies that a lisse $\nadic$-adic sheaf $\cF$ over $\etsite{K}$
arises from $\etsite{\nOK}$ if and only if the $G_K$-representation
$\cF_{\xbar{x}}$ is unramified at $\nval$.
In view of Theorem~\ref{wellknown} it follows that 
$(\nBR{K^s}\otimes_{\nBK} \nBlat)^\tau$ is unramified at $\nval$
if and only if $\nBlat$ arises from a $\nsigma$-module over $\nBR{\nOK}$.
Therefore \eqref{tatered0-unram} implies \eqref{tatered0-mod}.

Conversely, assume that $M$ admits a $\nBz$-model $\roM$.
Proposition~\ref{dmpure0} shows that $\roM$ arises from a $\nsigma$-module
$\rolat$ over $\nBR{\nOK}$.
Theorem~\ref{wellknown} and Lemma~\ref{fundsurj} imply that the $G_K$-representation
$(\nBR{K^s}\otimes_{\nBR{\nOK}} \rolat)^\tau$ is unramified at $\nval$.
Therefore $\nTate{M} = (\nBRz{K^s}\otimes_{\nBR{\nOK}} \rolat)^\tau$
is unramified at $\nval$.%
\end{proof}

\begin{prp}\label{moritagood}%
Let $M$, $N$ be pure $\nBKz$-isocrystals.
If $N$ and $N\otimes M$ have good reduction then so does $M$.%
\end{prp}
\begin{proof}%
We denote by $H$ the isocrystal $N \otimes M$.
Let $\ro{N}$ and $\ro{H}$ be the $\nBz$-models of $N$ and $H$.
Set $S = \iHom(N,N)$ and $\ro{S} = \iHom(\ro{N},\ro{N})$.

The isocrystal $H$ is a left $S$-isocrystal in the sense of Definition~\ref{dfnisomod}.
The isocrystals $S$ and $H$ are pure so
Proposition~\ref{dmmoduniq} shows that the structure of a left $S$-isocrystal on $H$
restricts to a structure of a left $\ro{S}$-isocrystal on $\ro{H}$.
Invoking Proposition~\ref{generalmorita} we conclude that
there is a $\nBz$-isocrystal $\ro{M}$ such that
$\ro{H}$ is isomorphic to $\ro{N}\otimes\ro{M}$ as a left $\ro{S}$-isocrystal.
So the left $S$-isocrystal $H$
is isomorphic to $N \otimes (\nBKz\otimes_{\nBz} \ro{M})$.
Applying Proposition~\ref{generalmorita} one more time
we deduce that $M$ is isomorphic to $\nBKz\otimes_{\nBz}\ro{M}$.%
\end{proof}

\begin{proof}[Proof of Theorem~\ref{tatered}]%
We denote by $\nslope$ the slope of $M$.
Fix a separable closure $K^s$ of $K$.
Let $\Fqbar \subset K^s$ be the algebraic closure of $\Fq$.
Pick a pure and simple $\nBRz{\Fqbar}$-isocrystal $N$ of slope $\nslope$.
The pair $(K^s, N)$ defines a Tate module functor. 

In addition to $N$ we shall use a nonzero pure $\nBRz{\Fq}$-isocrystal $P$ of slope~$\nslope$.
This exists by Proposition~\ref{pureexist}.
The $\nBKz$-isocrystal
$H = \iHom(\nBKz\otimes_{\nBRz{\Fq}}P, \,M)$
is pure of slope~$0$.
We claim that the following properties are equivalent:
\begin{enumerate}%
\renewcommand{\theenumi}{\roman{enumi}}%
\item\label{tatered-isounram}
the representation $T(M)$ is unramified at $\nval$,

\item\label{tatered-hunram}
the representation $T(H)$ is unramified at $\nval$,

\item\label{tatered-hgood}
the isocrystal $H$ has good reduction over $\nOK$,

\item\label{tatered-isogood}
the isocrystal $M$ has good reduction over $\nOK$.%
\end{enumerate}
Let $L$ be the subfield of $K^s$ generated by $K$ and $\Fqbar$.
We have
\begin{equation*}
\nTate{M} = \big(\nBRz{K^s} \otimes_{\nBRz{L}} \iHom(\nBRz{L}\otimes_{\nBRz{\Fqbar}}N, \,M)\big)^\tau.
\end{equation*}
The group $\nGalZero = \ord^{-1}\{0\}$ acts on $\nTate{M}$ via $K^s$.
Similarly
$\nTate{H} = (\nBRz{K^s}\otimes_{\nBKz} H)^\tau$
on which $G_0$ acts via $K^s$.

\eqref{tatered-isounram} $\Leftrightarrow$ \eqref{tatered-hunram}
Proposition~\ref{puredecomp} shows that for a suitable $m > 0$
there is an isomorphism
$\nBRz{\Fqbar} \otimes_{\nBRz{\Fq}} P \cong N^{\oplus m}$.
This induces an isomorphism of $G_0$-representations
$\nTate{M}^{\oplus m} \cong \nTate{H}$.
The claim follows since every inertia subgroup of $\Gal(K^s/K)$ is contained in $G_0$.

\eqref{tatered-hunram} $\Leftrightarrow$ \eqref{tatered-hgood}
Follows by Lemma~\ref{tatered0} since $H$ is pure of slope~$0$.

\eqref{tatered-hgood} $\Leftrightarrow$ \eqref{tatered-isogood}
The isocrystal $P_K = \nBKz\otimes_{\nBRz{\Fq}} P$ is pure and has a $\nBz$-model
$\nBz\otimes_{\nBRz{\Fq}} P$.
Hence the isocrystal $Q = (P_K)^*$ is pure and has good reduction.
By construction $H = Q \otimes M$.
Applying Proposition~\ref{moritagood} to $Q$ and $M$
we get the result.%
\end{proof}

\subsection{Hartl-Pink theory}\label{ss:hp}

We assume that $\ncoef = \rF{\Fq}$.
For expository purposes we shall consider an arbitrary valuation $\nval\colon K^\times \to \bR$
even though all our results will be stated for a discrete $\nval$.
We denote by $|x| = q^{-\nval(x)}$ the induced norm on $K$.
The field $K$ is assumed to be complete with respect to~$|\cdot|$.

\begin{dfn}\label{defdk}%
We denote by $\nDK$ the ring of Laurent series
\begin{equation*}
\sum_{n \in \bZ} \alpha_n z^n, \quad \alpha_n\in K
\end{equation*}
which converge on the punctured open unit disk: 
$\lim_{n \to \pm\infty} |\alpha_n|\varepsilon^n = 0$ 
for all real numbers $\varepsilon \in (0,1)$. 
We equip $\nDK$ with an endomorphism $\nsigma$ given by the formula
$\nsigma\left(\sum_n \alpha_n z^n \right) =
\sum_n \alpha_n^q z^n$.%
\end{dfn}

\begin{dfn}%
A $\nsigma$-module over $\nDK$ is called a \emph{$\sigma$-bundle}.%
\end{dfn}

The theory of Hartl and Pink \cite{hartl-pink}
classifies $\sigma$-bundles for every algebraically closed field $K$.

\begin{dfn}\label{dfnsimplebun}%
Let $\nslope = \frac{s}{r}$ be a rational number written in lowest terms.
We denote by $\cM_\nslope$
the $\sigma$-bundle 
with a $\nDK$-basis $e_1, \dotsc, e_r$ on which $\tau$
acts as follows:
\begin{equation*}
e_1 \xrightarrow{\,\tau\,}
e_2 \xrightarrow{\,\tau\,}
\dotsc \xrightarrow{\,\tau\,}
e_r \xrightarrow{\,\tau\,}
z^s e_1.%
\end{equation*}%
\end{dfn}

The $\sigma$-bundle
$\cM_\nslope$ is isomorphic to the $\sigma$-bundle $\cF_{-s,r}$ of \cite[Section 8]{hartl-pink}.
Note that Hartl and Pink use the geometric convention for the signs of slopes:
The $\sigma$-bundle $\cM_\nslope$ is stable of weight $-\nslope$ in their terminology.

The ring $\nDK$ can be defined in a way which does not involve a choice of a
uniformizer $z$. However the isomorphism type of
$\cM_\nslope$ depends on this choice when the field $K$ does not contain an agebraic closure of $\Fq$.

Our criterion of good reduction hinges on the following property of $\cM_\nslope$:

\begin{prp}\label{buninv}%
If the valuation $\nval$ is discrete then
\begin{equation*}
\dim_F \cM_\nslope^{\tau} = \left\{\hspace{-0.4em}%
\begin{array}{ll}
1, &\nslope = 0, \\
0, &\nslope\ne 0.
\end{array}\right.
\end{equation*}
\end{prp}

In contrast if $K$ is algebraically closed 
and $\nslope < 0$
then the $\ncoef$-vector space $\cM_\nslope^{\tau}$ is infinite-\hspace{0pt}dimensional
(see \cite[Proposition 8.4]{hartl-pink}; note the change of sign for $\nslope$).
The case $\nslope < 0$ is exactly the one appearing in 
the good reduction criterion.

\begin{rmk}\label{flop}%
Suppose that the valuation $\nval$ is discrete.
Proposition~\ref{buninv} implies that 
the graded algebra $P_{K,F,z}$ of Fargues-Fontaine~\cite[Section~7.3]{lacourbe}
consists of $F$ in degree~$0$.
It follows that the Fargues-Fontaine curve $X_{K,F} = \Proj(P_{K,F,z})$ is empty.%
\end{rmk}

\begin{rmk}%
Proposition~\ref{buninv} is also true for valuations $\nval$ satisfying the
condition of Hartl's theorem \cite[Theorem~2.5.3]{hartl-annals}
on the admissibility of Hodge-Pink isocrystals.
\end{rmk}

\begin{proof}[Proof of Proposition \ref{buninv}]%
Write $\nslope  = \frac{s}{r}$ in lowest terms.
Let $[r]\colon \nDK \to \nDK$ be the homomorphism defined by the formula
\begin{equation*}
[r]\Big(\sum_n \alpha_n z^n\Big) = \sum_n \alpha_n z^{rn}.
\end{equation*}
This homomorphism commutes with $\nsigma$ and so defines 
a functor of restriction of scalars $[r]_\ast$
on left $\nDK\{\tau\}$-modules.
Observe that $[r]_\ast \cM_s \cong \cM_\nslope$
\cite[the first paragraph of Section 8]{hartl-pink}.
Now
$([r]_\ast \cM_s)^{\tau} = \cM_s^{\tau}$
so we can assume that $\nslope = s$.

The elements of $\cM_s^{\tau}$ are solutions of the equation
\begin{equation*}
\sum_n \alpha_n z^n = \sum_n \alpha_n^q z^{n+s}.
\end{equation*}
Hence for all $n, m \in \bZ$ we have $\alpha_{n + m s} = \alpha_n^{q^m}$.
If $s = 0$ then clearly $\alpha_n \in \Fq$ and the convergence condition
implies that $\alpha_n = 0$ for all $n \ll 0$.
Assume that $s \ne 0$.
Then for all $n \in \bZ$ and all $m > 0$ the $q^m$-th root of $\alpha_n$
belongs to $K$. As $K$ is discretely valued we conclude that
$|\alpha_n|$ is either $0$ or $1$ and moreover $|\alpha_{n + s}| = |\alpha_n|$.
The convergence condition then implies that
$\alpha_n = 0$.%
\end{proof}

Next we relate $\sigma$-bundles to isocrystals.
The coefficient field $\ncoef$ is $\rF{\Fq}$ by assumption
and so $\nBKz = \rF{K}$ is the ring of Laurent series with finite principal parts.
The subring $\nBz \subset \rF{K}$
consists of the series with integral coefficients.
All such series converge on the punctured open unit disk. 
We thus have a natural inclusions $\nBz \subset \nDK$ which is compatible with
the action of $\nsigma$.

%

\begin{thm}\label{hpzero}%
If the valuation $\nval$ is discrete then
the functor $M \mapsto \nDK\otimes_{\nBz} M$
is fully faithful on the subcategory of pure isocrystals.%
\end{thm}%
\begin{exa}%
This theorem does not generalize to mixed isocrystals. 
Suppose that we are in the setting
of Example~\ref{exa-dmpure} and let $i\colon \unitob{\nBz} \to M$
be the non-split embedding of isocrystals constructed there.
The map $s\colon \nDK\otimes_{\nBz} M \to \unitob{\nDK}$
given by the formula
$s(\alpha e_0 + \beta e_1) = \alpha + \beta \sum_{n\geqslant 1} \zeta^{q^{n-1}} z^{-n}$
splits 
$i$ over $\nDK$.%
\end{exa}%
\begin{proof}[Proof of Theorem~\ref{hpzero}]%
Let $M$, $N$ be pure $\nBz$-isocrystals. Set $H = \iHom(M,N)$.
We shall prove that the natural map
$H^{\tau} \to (\nDK\otimes_{\nBz} H)^{\tau}$ is an isomorphism.
 
Let $\nKur$ be the maximal unramified extension of $K$.
Denote its completion by $\nKcur$ and let $\nRcur$ be the ring of integers of $\nKcur$.
Let $\nDKcur$ be the ring of Definition~\ref{defdk}
constructed for the field $\nKcur$.
The Galois group $G = \Gal(\nKur/K)$ acts on $\nKcur$ by continuity.
We get an induced action on $\nBRz{\nRcur}$ and $\nDKcur$.
The natural map
\begin{equation*}
(\nBRz{\nRcur}\otimes_{\nBz} H)^{\tau} \to
(\nDKcur\otimes_{\nBz} H)^{\tau}
\end{equation*}
is $G$-equivariant.
The action of $G$ commutes with $\tau$.
Hence the equality
$(\nKcur)^G = K$ implies that
\begin{equation*}
\big[(\nBRz{\nRcur}\otimes_{\nBz} H)^{\tau}\big]^G = H^{\tau}, \quad
\big[(\nDKcur\otimes_{\nBz} H)^{\tau}\big]^G = (\nDK\otimes_{\nBz} H)^{\tau}
\end{equation*}
Thus it is enough to prove the theorem for $\nKcur$ in place of $K$.
We are then free to assume that the residue field $\nresf$ of $K$ is separably closed.

The isocrystal $H$ is pure.
Let $\nslope = \frac{s}{r}$ be its slope written in lowest terms.
Denote by $M_\nslope$ the isocrystal with a $\nBz$-basis $e_1,\dotsc,e_r$
on which $\tau$ acts as in the definition of $\cM_\nslope$.
Proposition~\ref{puredecomp} implies that
for a suitable $n$ there exists an isomorphism
$\nBRz{\nresf}\otimes_{\nBz} H \xrightarrow{\isosign} \nBRz{\nresf} \otimes_{\nBz} M_\nslope^{\oplus n}$.
Theorem~\ref{dmpure} lifts it to an isomorphism
$H \xrightarrow{\isosign} M_\nslope^{\oplus n}$.
By naturality we reduce to the case $H = M_\nslope$ where
the claim follows from Proposition~\ref{buninv}.%
\end{proof}

\section{Drinfeld modules}\label{sec:goodred}

Let $A$ be a Dedekind domain of finite type over $\Fq$ with finite group of units $A^{\times}$.
We call $A$ the \emph{coefficient ring}.
We use the following notation:
\begin{itemize}
\item $\ncurve$ is the projective compactification of $\Spec A$ over $\Fq$,

\item $\infty$ is the unique closed point in the complement of $\Spec A$ in $\ncurve$,

\item $F_\infty$ is the local field of $\ncurve$ at $\infty$,

\item $\ninfdeg$ is the degree of the residue field of $F_\infty$ over $\Fq$,
%
\end{itemize}
%
Let $K$ be a field over $\Fq$. Fix a Drinfeld $A$-module $E$ of rank $r$ over $\Spec K$.

\subsection{The motive of a Drinfeld module}\label{ss:mot}%
Let us review the notion of a motive of $E$
introduced by Anderson~\cite{anderson}
and discuss several results 
of
Anderson~\cite{anderson}
and
Drinfeld~\cite{drinfeld-commrgs}
concerning its structure.

\begin{dfn}%
Given an $\Fq$-algebra $R$ we set $\nAR{R} = R\otimes_{\Fq} A$.
We equip $\nAR{R}$ with an endomorphism $\nsigma$
acting as the identity on $A$ and as the $q$-Frobenius on $R$.%
\end{dfn}


\begin{dfn}\label{dfnmot}%
The \emph{motive} $M$ of $E$ is
the left $\nAR{K}\{\tau\}$-module
$\Hom_{\Fq}(E,\,\bG_{a,K})$
of $\Fq$-linear group scheme
morphisms from $E$ to $\bG_{a,K}$.
The ring $A$ acts on $M$ via the group scheme~$E$.
The field $K$ and the element $\tau$ act by composition with the scalar endomorphisms
and the $q$-Frobenius of $\bG_{a,K}$ respectively.
\end{dfn}


\begin{prp}[Drinfeld \cite{drinfeld-commrgs}]\label{motrank}%
The $\nAR{K}$-module $M$ is locally free of rank $r$.%
\end{prp}%
\begin{proof}%
Let $X = \Spec K\times_{\Fq} \ncurve$. 
The $\nAR{K}$-module $M$ defines a quasi-coherent sheaf on 
the open subscheme $\Spec(K\otimes_{\Fq} A) \subset X$.
Proposition~3 of \cite{drinfeld-commrgs} implies that this sheaf
extends to a vector bundle of rank $r$ on $X$.%
\end{proof}

Let $\Lie_E(K)$ be the tangent space of $E$ at~$0$. 
The action of $A$ on $\Lie_E(K)$ provides the cotangent space
$\Lie_E(K)^* = \Hom_K(\Lie_E(K),\,K)$ with a structure of an $\nAK$-module.

\begin{prp}[Anderson~\cite{anderson}, Drinfeld~\cite{drinfeld-commrgs}]\label{motlie}%
The morphism $\taulin\colon\nsigma^* M \hookrightarrow M$ is injective
and the $\nAK$-module $\coker(\taulin)$ is canonically isomorphic
to $\Lie_E(K)^*$.%
\end{prp}%
\begin{proof}%
\cite[Lemma 1.3.4]{anderson}.
This also follows from
\cite[Proposition 3 (2)]{drinfeld-commrgs}.%
\end{proof}

For each $\Fq$-algebra $R$ we set $\nBR{R} = \nFBR{F_\infty}{R}$
and $\nBRz{R} = \nFBRz{F_\infty}{R}$.

\begin{dfn}%
$M_\infty = \nBKz\otimes_{\nAR{K}} M$.
\end{dfn}

\begin{prp}[Drinfeld~\cite{drinfeld-commrgs}]\label{infslope}%
The module $M_\infty$ is a pure isocrystal of slope~$-\frac{1}{r}$.
%
\end{prp}%
\begin{proof}
Let $X = \Spec K\times_{\Fq} \ncurve$ and let $\nsigma\colon X \to X$
be the endomorphism which acts as the $q$-Frobenius on $K$ and as
the identity on $\ncurve$. Denote by $U$ the open subscheme
$\Spec (K\otimes_{\Fq} A) \subset X$.
The complement of $U$ is a divisor. Let $\cO(\infty)$ be the corresponding line bundle.

According to Proposition~3 of \cite{drinfeld-commrgs} there is an
increasing family $\{\cF_n\}_{n\in\bZ}$
of rank $r$ vector bundles on $X$ extending $M$ on $U$
and a compatible system of morphisms $\nsigma^\ast\cF_n \to \cF_{n+1}$
extending the structure morphism $\sigma^\ast M \to M$ such that
for all $n$
\begin{enumerate}%
\renewcommand{\theenumi}{\roman{enumi}}%
\item\label{infslope-bunquot}%
the induced map $\nsigma^\ast(\cF_n/\cF_{n-1}) \to \cF_{n+1}/\cF_n$
is an isomorphism,

\item\label{infslope-bunshift}%
$\cF_{n+r\ninfdeg} = \cF_n(\infty)$.%
\end{enumerate}%
Set $\nBlat_n = \uH^0(\Spec\nBK,\,\cF_n)$.
The fibre product of $\Spec\nBK$ and $U$ over $X$ is $\Spec\nBKz$.
Therefore
$\{\nBlat_n\}_{n\in\bZ}$ is an increasing family of lattices in $M_\infty$.
We claim that
\begin{enumerate}%
\renewcommand{\theenumi}{\alph{enumi}}%
\item\label{infslope-step}%
$\nBlat_{n+1} = \nlati{}{\nBlat_n}$,

\item\label{infslope-shift}%
$\nBlat_{n+r\ninfdeg} = z^{-1}\nBlat_n$
where $z \in F_\infty$ is a uniformizer.%
\end{enumerate}%
Property~\eqref{infslope-shift} follows immediately from \eqref{infslope-bunshift}.
Let us prove~\eqref{infslope-step}.

Since the morphism $\sigma^\ast\cF_n \to \cF_{n+1}$ restricts to the
structure morphism of $M$ on $U$ it follows that
$\nlati{}{\nBlat_n} \subset \nBlat_{n+1}$.
Property~\eqref{infslope-bunquot} implies that
the induced map $\nsigma^\ast(\cF_n/\cF_{n-r\ninfdeg}) \to \cF_{n+1}/\cF_{n-r\ninfdeg+1}$
is an isomorphism. Now $\cF_{n-r\ninfdeg} = \cF_n(-\infty)$ by \eqref{infslope-bunshift}
so the inclusion
$\nlati{}{\nBlat_n} \hookrightarrow \nBlat_{n+1}$
induces an isomorphism modulo $z\nBK$.
As $\nBK$ is $\nadic$-adically complete Nakayama's lemma
implies that $\nlati{}{\nBlat_n} = \nBlat_{n+1}$.

It follows from \eqref{infslope-step} and \eqref{infslope-shift}
that $z^{-1}\nBlat_0 = \nlati{r\ninfdeg}{\nBlat_0}$.
This fact has two consequences.
First, the structure morphism of $M_\infty$ is surjective.
As it is also injective we conclude that $M_\infty$
is an isocrystal.
Second, the isocrystal $M_\infty$ is pure of slope $-\frac{1}{r}$. 
Whence the result.%
\end{proof}

%
%
%
%
%

\subsection{The Tate module at infinity}\label{ss:inftate}%
We shall define the Tate module of $M$ at the place $\infty$
and relate it to the representation of the Weil group introduced
by J.-K.~Yu~\cite{yu}. 
This construction was envisaged by Taelman. 

Fix a separable closure $K^s$ of $K$ and
let $\Fqbar \subset K^s$ be the separable closure of $\Fq$.
Fix a simple $\nBRz{\Fqbar}$-isocrystal $N$ of slope $-\frac{1}{r}$.

\begin{dfn}%
The \emph{Tate module of $M$ at infinity} is
\begin{equation*}
\nTinf{M} = \Hom_{\nBRz{\Fqbar}\{\tau\}}(N,\:\nBRz{K^s}\otimes_{\nAR{K}} M).
\end{equation*}
It carries an action of the ring $\nDN =\End N$ on the right
and the Weil group $W_K$ on the left, see Definition~\ref{dfnweilact}.
\end{dfn}

According to Proposition~\ref{simple}
the ring $\nDN$ is a central division algebra of invariant $\frac{1}{r}$ over $F_\infty$.
Proposition~\ref{taterank} shows that $\dim_\nDN \nTinf{M} = 1$.%

\begin{rmk*}%
We speak of the Tate module of the motive $M$ rather than $E$
to respect the directions of the underlying constructions:
the functor $E \mapsto M$ is contravariant but $M \mapsto \nTinf{M}$ is covariant.%
\end{rmk*}

Let $K\{\tau\}$ be the twisted polynomial ring associated to the $q$-Frobenius endomorphism $\nsigma\colon K \to K$.
Recall that $\End_{\Fq}(\bG_{a,K}) = K\{\tau\}$.
Fix an isomorphism $\mu\colon \bG_{a,K} \xrightarrow{\isosign} E$.
This determines a morphism
$\nphi\colon A \to K\{\tau\}$, $a \mapsto \mu^{-1} \circ a \circ \mu$. 

A construction of J.-K.~Yu~\cite[Construction~2.5 and Section~3.1]{yu}
associates to $\nphi$ 
a left $(\nDN^{\textup{op}})^\times$-torsor $Y_\iota$ equipped with an action of $W_K$
(see also our Section~\ref{ss:yucomp}).
Making a choice of an element $u \in Y_\iota$ one 
obtains a representation
\begin{equation*}
\rho_\infty\colon W_K \to (\nDN^{\textup{op}})^\times, \quad
\rho_\infty(\gamma) \cdot (\gamma u) = u\textup{ for all }\gamma\in W_K.
\end{equation*}
A different choice of $u$ changes $\rho_\infty$ by conjugation.

\begin{thm}\label{tateyu}%
There is a natural isomorphism of $W_K$-sets
$Y_\iota\xrightarrow{\isosign} \nTinf{M}\backslash\{0\}$.
This isomorphism transforms the left action of $(D^{\textup{op}})^\times$ 
to the right action of $D^\times$.
\end{thm}%
\begin{rmk}\label{yudeps}%
The construction of the torsor $Y_\iota$ involves a choice of an isomorphism $\mu$.
In contrast the Tate module $\nTinf{M}$ is canonical in $E$ itself.
\end{rmk}%
\begin{proof}[Proof of Theorem~\ref{tateyu}]%
Let $K^{\npf}$ be the perfect closure of $K$.
The construction of the motive $M$ is compatible with arbitrary base change.
Hence the motive of $E$ over ${\Spec K^\npf}$ is $K^\npf\otimes_K M$
and the corresponding $\infty$-adic isocrystal is
$\nBRz{K^\npf}\otimes_{\nBKz} M_\infty$.
It is easy to see that $\nTinf{M} = \nTate{\nBRz{K^\npf}\otimes_{\nBKz} M_\infty}$.
At the same time the torsor $Y_\iota$ is defined in terms of the base change of $E$ to $\Spec K^\npf$.
Hence we are free to assume that $K$ is perfect.

Let $\nvalinf$ be the normalized valuation on $F_\infty$
and $\nvaltau$ the normalized $\tau^{-1}$-adic valuation on $\rtauF{K}$.
The fact that $E$ is a Drinfeld module of rank $r$ implies that
\begin{equation*}
\nvaltau(\nphi(a)) = r \ninfdeg \cdot \nvalinf(a)
\end{equation*}
for all $a\in A$.
So $\nphi$ extends uniquely to a continuous morphism $\nphi\colon F_\infty \to \rtauF{K}$.
By construction $\nvaltau(\nphi(z)) = r \ninfdeg$ for a uniformizer $z\in F_\infty$.

Set $\nV = \rtauF{K}\otimes_{K\{\tau\}} M$. The discussion above shows
that $\nV$ is of dimension~$1$ over $\rtauF{K}$
and that the action of $A$ on $M$ extends uniquely to a continuous
action of $F_\infty$ on $\nV$. This makes $\nV$ a formal Drinfeld motive.
The height of $\nV$ is $r$ since $\nvaltau(\nphi(z)) = r\ninfdeg$.

Proposition~\ref{laurentdm} equips $\nV$ with a structure of a $\nBKz$-isocrystal $M(\nV)$.
The embedding $M \hookrightarrow M(\nV)$ is $\nAK\{\tau\}$-linear by construction
and so extends to a nonzero morphism of isocrystals $M_\infty \to M(\nV)$.
Its source and target are of slope $-\frac{1}{r}$ and rank $r$
by Propositions~\ref{infslope} and \ref{laurentdm} respectively.
Proposition~\ref{simpleiso} shows that $M_\infty$ and $M(\nV)$ are simple.
Hence $M_\infty \xrightarrow{\isosign} M(\nV)$ is an isomorphism.

The choice of the element $1 \otimes \mu^{-1} \in \nV$ identifies $\nV$
with the formal Drinfeld motive $\nV(\nphi)$ defined by $\nphi$. 
Applying Theorem~\ref{tateyu-formal} to $\nV$ we get 
an isomorphism $Y_\iota \xrightarrow{\isosign} \nTate{\nV} \backslash\{0\}$.
The Tate module $\nTate{\nV}$ is isomorphic to $\nTate{M(\nV)}$ by Proposition~\ref{formalcomp}.
Whence the result.%
\end{proof}

\subsection{Generalization to \texorpdfstring{$\tau$}{tau}-sheaves}\label{ss:andpure}

A \emph{$\tau$-sheaf} over $\nAK$ is a left $\nAK\{\tau\}$-module $M$ such that
\begin{itemize}
\item
$M$ is locally free of finite rank over $\nAK$,

\item
the structure morphism of $M$ is injective.
\end{itemize}
These objects were introduced by Drinfeld and christened by Gardeyn
\cite[Section~1.1]{gardeyn-anstruct}.

One can show that $M_\infty = \nBKz\otimes_{\nAK} M$ is always an isocrystal.
A $\tau$-sheaf $M$ is called \emph{pure} if the isocrystal $M_\infty$ is pure
\cite[Definition~0.15]{gardeyn-thesis}. 
The effective $A$-motives of \cite{hartl-juschka} form a subcategory of $\tau$-sheaves and the notion of purity
for them is the same, see
\cite[Definitions 3.1, 3.5]{hartl-juschka}
and \cite[Sections 1.2, 1.9]{anderson}.

Let $M$ be a pure $\tau$-sheaf.
Making a choice of a pure and simple $\nBRz{\Fqbar}$-isocrystal $N$ of the same slope as $M_\infty$
we get the Tate module $\nTinf{M}$.
For a pure effective $A$-motive $M$ the rank of $\nTinf{M}$ over $\nDN = \End N$ is equal to
the greatest common divisor of the dimension and the rank of $M$.

%
%
%
%
%
%
%
%
%


\subsection{Models of the motive}
From now on we work with a fixed discrete valuation $\nval\colon K^\times \twoheadrightarrow \bZ$.
Let $\nOK$ be its ring of integers.
As before we fix a Drinfeld $A$-module $E$ of rank $r$ over $\Spec K$
and denote by $M$ its motive.

The $K$-vector space $\Lie_E(K)$ is one-dimensional,
and so the action of $A$ on it induces a homomorphism $\iota_E\colon A \to K$, the \emph{characteristic morphism} of~$E$.
Throughout the rest of this paper we make the following assumption:
\begin{align*}
\fbox{\ensuremath{\iota_E(A) \subset \nOK.}}
\end{align*}
We are then in the setting of reduction theory of Drinfeld
modules, cf.~\cite[\S4.1, first sentence]{goss}.

\begin{rmk}
One says that the Drinfeld module $E$ has \emph{finite residual characteristic} 
if $\iota_E(A) \subset \nOK$ as above. The aforementioned residual characteristic
is the preimage of the maximal ideal of $\nOK$ under $\iota_E$.
If $\iota_E(A) \not\subset\nOK$ then one says that the \emph{residual characterisctic
of $E$ is infinite.}
The analogous situation for abelian varieties is when
the base field $K$ is $\mathbb{R}$ or $\mathbb{C}$.
%
\end{rmk}

As before we set
$\nBzl = K\otimes_{\nOK}\nBz$.
The natural morphism $\nAR{\nOK} \to \nBz$ extends uniquely
to a morphism $\nAR{K} \to \nBzl$.

\begin{dfn}%
$\rMloc = \nBzl \otimes_{\nAR{K}} M$.%
\end{dfn}

\begin{prp}\label{mloctaubun}%
The module $\rMloc$ is a $\nBzl$-isocrystal.%
\end{prp}%
This is false unless $\iota_E(A) \subset \nOK$.
\begin{proof}[Proof of Proposition~\ref{mloctaubun}]%
Let $a \in A$ be an element which is transcendental over $\Fq$ and let
$\alpha = \iota_E(a)$ be its image in $\nOK$.
Consider the element $x = 1 \otimes a - \alpha \otimes 1$ of $\nAR{\nOK}$.
According to Proposition~\ref{motlie} this element acts by zero
on the cokernel of the structure morphism of $M$.
It is enough to show that $x$ is invertible in $\nBzl$.

Let $\nvalinf\colon F_\infty^\times \to \bZ$ be the normalized valuation.
Since $a \in A$ is transcendental over $\Fq$ it follows that $\nvalinf(a) < 0$. As a consequence the series
$\sum_{n\geqslant 0} \alpha^n \otimes a^{-n}$
converges to an  element $y \in \nBz$. By construction $xy = 1 \otimes a$
and the result follows since $a$ is invertible in $F_\infty$.
\end{proof}

Next we review the notion
of a \emph{model of the motive $M$} which was introduced by Gardeyn in \cite{gardeyn-anstruct} and \cite{gardeyn-thesis}.
Let $\nresf$ be the residue field of the valuation $v\colon K^\times \twoheadrightarrow \bZ$.

\begin{dfn}
A \emph{model of the motive $M$ over $\nOK$}
is a left $\nAR{\nOK}\{\tau\}$-submodule $\roM$
such that
\begin{itemize}
\item
$\roM$ is a locally free $\nAR{\nOK}$-module of finite type,

\item
$K\otimes_{\nOK}\roM = M$.
\end{itemize}
A model $\roM$ is \emph{good} if the structure morphism of 
$\nAR{\nresf}\otimes_{\nAR{\nOK}} \roM$ 
is injective.
\end{dfn}

A good model is unique if it exists \cite[Proposition~2.13]{gardeyn-anstruct}.
Gardeyn proved that the motive $M$ admits a good model 
if and only if the Drinfeld module $E$ has good reduction 
\cite[Theorem~8.1]{gardeyn-goodred}.
His argument depends on the good reduction criterion of Takahashi~\cite{takahashi}.
As an alternative we shall show directly
that the good model of $M$ is necessarily a motive of a Drinfeld module over
$\Spec\nOK$ extending $E$,
see Theorem~\ref{gardeyn}.

\begin{dfn}%
Given a model $\roM$ we set
$\roM_\infty = \nBz\otimes_{\nAR{\nOK}} \roM$.%
\end{dfn}

This is a left $\nBz\{\tau\}$-module 
which is locally free over $\nBz$.
It is not necessarily an isocrystal
since the structure morphism $\nsigma^\ast \roM_\infty \hookrightarrow \roM_\infty$
may not be surjective.

\begin{prp}\label{dmgoodmod}%
A model $\roM$ is good if and only if $\roM_\infty$ is an isocrystal.%
\end{prp}

\begin{proof}%
Let $\rrM = \nAR{\nresf} \otimes_{\nAR{\nOK}} \roM$
and let $\rrMinf = \nBRz{\nresf} \otimes_{\nAR{\nresf}} \rrM$.
Denote by $\roMa{f}$, $\rrMinfa{f}$ and $\rrMa{f}$ the structure morphisms
of $\roM_\infty$, $\rrMinf$
and $\rrM$ respectively.
We claim that the following properties
are equivalent:
\begin{enumerate}%
\renewcommand{\theenumi}{\roman{enumi}}%
\item\label{dmgoodmod-red}%
$\rrMa{f}$ is injective.%

\item\label{dmgoodmod-infred}%
$\rrMinfa{f}$ is an isomorphism.

\item\label{dmgoodmod-inf}%
$\roMa{f}$ is an isomorphism.
\end{enumerate}%

\eqref{dmgoodmod-red} $\Rightarrow$ \eqref{dmgoodmod-infred}.
The $\nAR{K}$-module $M$ is locally free 
of constant rank.
Since the map $\Spec \nAR{K} \to \Spec\nAR{\nOK}$
is bijective on connected components
it follows
that $\roM$ is a locally free $\nAR{\nOK}$-module of constant
rank. Hence the same is true of the $\nAR{\nresf}$-module $\rrM$
and the $\nBRz{\nresf}$-module $\rrMinf$.
The base change of $\rrMa{f}$ to $\nBRz{\nresf}$ is naturally isomorphic to $\rrMinfa{f}$.
Since $\nBRz{\nresf}$ is flat over $\nAR{\nresf}$ it follows
that $\rrMinfa{f}$ is injective.
As $\nBRz{\nresf}$ is a finite product of fields
and $\rrMinf$ is locally free of constant rank
we conclude that $\rrMinfa{f}$ is also surjective.%

\eqref{dmgoodmod-infred} $\Rightarrow$ \eqref{dmgoodmod-red}.
We have a commutative square
\begin{equation*}
\xymatrix{
\nsigma^\ast \rrMinf \ar[r] & \rrMinf \\
\nsigma^\ast \rrM \ar[r] \ar[u] & \rrM \ar[u]
}
\end{equation*}
The vertical arrows are injective since the image of $\Spec\nBkz$
contains all the generic points of $\nAR{\nresf}$.
So if the top arrow is an isomorphism
then the bottom arrow must be injective.%

\eqref{dmgoodmod-infred} $\Rightarrow$ \eqref{dmgoodmod-inf}.
Pick a uniformizer $\zeta\in K$.
The kernel of the reduction map
$\nBz \to \nBRz{\nresf}$ is $\zeta\nBz$.
Let $\nBzc$ be the completion of $\nBz$ at $\zeta\nBz$.
Nakayama's lemma implies that
the base change of $\roMa{f}$ to $\nBzc$
is an isomorphism.
The base change of $\roMa{f}$ to $\nBzl$ is the structure morphism of $\rMloc$.
This is an isomorphism by Proposition~\ref{mloctaubun}.
Since $\Spec\nBzl$ and $\Spec\nBzc$ form an fpqc covering
of $\Spec\nBz$ it follows
that $\roMa{f}$ is an isomorphism.
The implication
\eqref{dmgoodmod-inf} $\Rightarrow$ \eqref{dmgoodmod-infred}
is clear.%
\end{proof}

\pagebreak
\begin{thm}[Gardeyn]\label{gardeyn}%
The Drinfeld module $E$ has good reduction if and only if its motive $M$ admits a good model.
Moreover the good model of $M$, if it exists, is necessarily a motive of a Drinfeld module
over $\Spec\nOK$ extending~$E$.%
\end{thm}%
\begin{proof}%
Suppose that $E$ is the generic fiber of a Drinfeld $A$-module
$\nDMOK$ over $\Spec\nOK$.
Let $\roM$ be the abelian group of $\Fq$-linear group scheme morphisms
from $\nDMOK$ to $\bG_{a,\nOK}$ over $\Spec\nOK$.
In analogy to $M$ this carries a canonical structure of a left $\nAR{\nOK}\{\tau\}$-module.
Proposition~3 of \cite{drinfeld-commrgs}
implies that $\roM$ is a locally free $\nAR{\nOK}$-module of constant
rank.
The construction of $\roM$ is compatible with the change
of $\nOK$. Hence 
\begin{itemize}
\item $\nMzl{\roM} = M$,

\item the structure morphism of $\nAR{\nresf}\otimes_{\nAR{\nOK}}\roM$
is injective.
\end{itemize}
The first claim follows since the generic fiber of $\nDMOK$
is $E$. The second claim follows from Proposition~\ref{motlie}
since the special fiber of $\nDMOK$ is a Drinfeld module.
Thus $\roM$ is a good model of $M$ and we are done.

Conversely, suppose that $M$ admits a good model.
Proposition~\ref{infslope} shows that $M_\infty$ is a pure isocrystal of rank $r$ and slope $-\frac{1}{r}$.
Applying Proposition~\ref{embed} to the dual $M_\infty^*$
we conclude that there is an increasing family
$\{\raMinf{n}\}_{n\in\bZ}$ of lattices in $M_\infty$ such that
for all $n$
\begin{enumerate}%
\renewcommand{\theenumi}{\arabic{enumi}a}
\item\label{gardeyn0-step}%
$\raMinf{n+1} = \nlati{}{\raMinf{n}}$,

\item%
$\raMinf{n+r\ninfdeg} = z^{-1} \raMinf{n}$ where $z \in F_\infty$ is a uniformizer,

\item\label{gardeyn-krank}%
$\dim_K \raMinf{n+1}/\raMinf{n} = 1$.
\end{enumerate}

Let $\roM$ be the good model of $M$ and set
$\raoMinf_n = \raMinf{n} \cap \roM_\infty$.
We claim that $\{\raoMinf_n\}_{n\in\bZ}$ is an increasing family
of lattices in $\roM_\infty$
such that
for all $n$
\begin{enumerate}%
\renewcommand{\theenumi}{\arabic{enumi}b}
\item\label{gardeyn-step}%
$\raoMinf_{n+1} = \nlati{}{\raoMinf_n}$,

\item\label{gardeyn-shift}%
$\raoMinf_{n+r\ninfdeg} = z^{-1}\raoMinf_n$

\item\label{gardeyn-qrank}%
$\raoMinf_{n+1}/\raoMinf_n$ is a free $\nOK$-module of rank $1$.%
\end{enumerate}

Applying Lemma~\ref{abunlat} and Lemma~\ref{bunlat} we conclude that $\raoMinf_n$
is the unique lattice in $\roM_\infty$ which generates $\raMinf{n}$ over $\nBK$.
The structure morphism of $\roM_\infty$ is an isomorphism by
Proposition~\ref{dmgoodmod}.
Hence $\nlati{}{\raoMinf_n}$ is a lattice in $\roM_\infty$.
The unicity property of $\raoMinf_n$ then implies
\eqref{gardeyn-step} and \eqref{gardeyn-shift}.
It follows from \eqref{gardeyn-shift} that
the quotient $\raoMinf_{n+1}/\raoMinf_n$ is a submodule
of a finitely generated free $\nOK$-module $\raoMinf_{n+r\ninfdeg}/\raoMinf_n$ =
$z^{-1} \raoMinf_n / \raoMinf_n$.
In particular this quotient is finitely generated free.
Property~\eqref{gardeyn-krank} then implies 
\eqref{gardeyn-qrank}.

Our next task is to build out of $\roM$ and the family $\{\raoMinf_n\}$
a Drinfeld module over $\Spec\nOK$ with the generic fiber $E$.
We use the following notation:%
\begin{itemize}%
\item $X = \Spec\nOK\times_{\Fq} \ncurve$,
\item $\nsigma\colon X \to X$ is the product of the $q$-Frobenius on $\nOK$
and the identity on $\ncurve$,
\item $Z = \Spec R \times_{\Fq} \{\infty\}$,
\item $\cO(\infty)$ is the line bundle corresponding to $Z$.%
\end{itemize}

The scheme $\Spec\nB$ is the completion of $X$ along $Z$.
The complement of $Z$ in $X$ is $\Spec(\nOK\otimes_{\Fq} A)$
and $\Spec\nBz = (X\backslash Z) \times_X \Spec\nB$.
So glueing $\roM$ to $\raoMinf_n$ along $\roM_\infty$
we obtain a vector bundle $\cF_n$ of rank $r$ on $X$.

The bundles $\cF_n$ form an increasing family.
By property \eqref{gardeyn-step}
the multiplication by $\tau$ on $\roM$ and $\raoMinf_n$
defines a compatible family of morphisms $\nsigma^\ast\cF_n \to \cF_{n+1}$.
We claim that for all $n$
\begin{enumerate}%
\renewcommand{\theenumi}{\roman{enumi}}%
\item\label{gardeynf-qiso}%
the induced map $\nsigma^\ast(\cF_n/\cF_{n-1}) \to \cF_{n+1}/\cF_n$
is an isomorphism,

\item\label{gardeynf-shift}%
$\cF_{n+r\ninfdeg} = \cF_n(\infty)$,%

\item\label{gardeynf-qrank}%
$\uH^0(X,\,\cF_n/\cF_{n-1})$ is a free $\nOK$-module of rank $1$.
\end{enumerate}%
Indeed Property \eqref{gardeynf-qiso} follows from \eqref{gardeyn-step} and Snake Lemma,
\eqref{gardeyn-shift} implies \eqref{gardeynf-shift},
and
\eqref{gardeyn-qrank} implies \eqref{gardeynf-qrank}.
Every quotient $\cF_n/\cF_{n-1}$ is supported on the affine subscheme $Z \subset X$.
So Property \eqref{gardeynf-qrank} implies that
$\chi(\cF_n) = \chi(\cF_{n-1}) + 1$ for all $n$.
Shifting the indices we can ensure that
\begin{enumerate}%
\setcounter{enumi}{3}%
\renewcommand{\theenumi}{\roman{enumi}}%
\item\label{gardeynf-eu}%
$\chi(\cF_{-1}) = 0$.
\end{enumerate}%
According to \cite[Proposition~3]{drinfeld-commrgs}
the family of vector bundles $\{\cF_n\}$
and morphisms $\{\sigma^\ast\cF_n \to \cF_{n+1}\}$
satisfying \eqref{gardeynf-qiso} -- \eqref{gardeynf-eu}
determines a Drinfeld module $\nDMOK$ over $\nOK$ with the motive $\roM$.
This construction is natural in $\nOK$ and so the motive
of the generic fiber of $\nDMOK$ is $M$.
The functor from Drinfeld modules
to motives is fully faithful \cite[Theorem~1]{anderson}.
Hence the generic fiber of $\nDMOK$ is $E$.%
\end{proof}

%

\subsection{The good reduction criterion}%

We keep the following notation:
\begin{itemize}
\item
$K$ is a field over $\Fq$ equipped
with a valuation $v\colon K^\times \twoheadrightarrow \bZ$,

\item
$\nOK$ is the ring of integers of $v$,

\item
$E$ is a Drinfeld $A$-module of rank $r$ over $\Spec K$,

\item
$M$ is the motive of $E$.
\end{itemize}
We also keep the assumption that the residual
characteristic of $E$ is finite with respect to $v$.

%
%
%
\begin{lem}\label{makemod}%
Let $P \subset M$ be a left $\nAR{\nOK}\{\tau\}$-submodule such that $\nMzl{P} = M$
and let $P_\infty = \nBz\otimes_{\nAR{\nOK}} P$.
If $P$ is finitely generated over $\nAR{\nOK}$
then there is a model $\roM\subset M$ containing $P$ such that
$\roM_\infty = P_\infty$.%
\end{lem}%
\begin{proof}%
Let $P^{**}$ be the double dual of $P$ over $\nAR{\nOK}$
and let $M^{**}$ be the double dual of $M$ over $\nAR{K}$.
Since $M$ is locally free the natural map $M \to M^{**}$ is an isomorphism.
Let $\roM$ be the preimage of $P^{**}$ under this map.
By construction $\roM$ contains $P$ and
is a reflexive $\nAR{\nOK}$-module.
Since $\nAR{\nOK}$ is a regular ring of dimension $2$
it follows that $\roM$ is locally free \cite[Corollary~1.4]{hart}.

Double duals of finitely generated modules are functorial and compatible
with arbitrary base change.
In particular $\taulin_P$ gives rise to a morphism $\nsigma^\ast (P^{**}) \to P^{**}$
which is compatible with the morphism
$\nsigma^\ast (M^{**}) \to M^{**}$ induced by $\taulin_M$.
Hence $\roM$ is a left $\nAR{\nOK}\{\tau\}$-submodule of $M$.
Next, $\nMzl{P} = M$ 
so the compatibility of double duals with base change
implies that $\roM$ is a model of $M$.

By construction $P_\infty$ is a submodule of $M_\infty$ and so is torsion-free.
As $\nBz$ is a finite product of regular rings of dimension $1$
it follows that $P_\infty$ is locally free.
Therefore the natural map $P_\infty \to P^{**}_\infty$ is an isomorphism.
The fact that $\roM_\infty = P_\infty$ then follows from the
compatibility of double duals with base change.%
\end{proof}

\begin{prp}\label{metagood}%
Let $\nOK \to \nOL$ be a finite local morphism of discrete valuation rings
and let $\nL$ be the fraction field of $\nOL$.
Suppose that $\nL\otimes_K M$ has a good model over~$\nOL$.
If $M_\infty$ has good reduction over~$\nOK$ then 
$M$ has a good model over $\nOK$.%
\end{prp}%
\begin{proof}%
We claim that it is enough to construct a left $\nAR{\nOK}\{\tau\}$-submodule $P \subset M$
with the following properties:
\begin{enumerate}
\item\label{premodlat}%
$P$ is finitely generated over $\nAR{\nOK}$ and $\nMzl{P} = M$,

\item\label{premodgood}%
$P_\infty = \nBz \otimes_{\nAR{\nOK}} P$ is an isocrystal.
\end{enumerate}%
Indeed Lemma~\ref{makemod} shows that there is a model $\roM \subset M$
such that $\roM_\infty = P_\infty$ is an isocrystal.
Proposition~\ref{dmgoodmod} implies that the model $\roM$ is good
and the claim follows.

Let us construct a submodule $P$ satisfying \eqref{premodlat} and \eqref{premodgood}.
Let $\roML$ be the good model of $\rML = \nL\otimes_K M$ over~$\nOL$.
%
The left $\nAR{\nOK}\{\tau\}$-module $P = M \cap \roML$
is of finite type over $\nAR{\nOK}$ since $\roML$ is finitely generated
over $\nAR{\nOL}$ and the morphism $\nAR{\nOK}\to\nAR{\nOL}$ is finite.
The equality $K\otimes_{\nOK} \roML = \rML$ implies that $\nMzl{P} = M$.
So $P$ has property~\eqref{premodlat}.

Before we continue with the proof of~\eqref{premodgood} let us
introduce the following notation:
\begin{equation*}
\nBLzl = \nMLzl{\nBLz}, \quad
\rMLloc = \nBLzl \otimes_{\nAR{\nL}} \rML, \quad
\roMLinf = \nBLz \otimes_{\nAR{\nOL}} \roML.
\end{equation*}
Since $\nBz$ is flat over $\nAR{\nOK}$ we conclude that
\begin{equation*}
P_\infty = 
\ker\big( \rMloc \oplus \roMLinf \xrightarrow{\:\textup{difference}\:} \rMLloc\big).
\end{equation*}
Let $\roM_\infty$ be the good model of $M_\infty$.
Recall that $\rMloc = \nBzl \otimes_{\nAR{\nOK}} M$.
Proposition~\ref{modcmp} shows that
$\rMloc = K\otimes_{\nOK}\roM_\infty$ and 
$\rMLloc = \nL\otimes_{\nOK} \roM_\infty$.
Similarly $\nBRz{\nOL}\otimes_{\nBz}\roM_\infty = \roMLinf$ by Proposition~\ref{dmmoduniq}.
So
\begin{equation*}
P_\infty = \ker\big(\nBzl \oplus \nBLz \xrightarrow{\:\textup{difference}\:} \nBLzl\big) \otimes_{\nBz} \roM_\infty.
\end{equation*}
Since $\nOL$ is finitely generated free over $\nOK$
it follows that $\nBR{\nOL} = \nOL\otimes_{\nOK} \nBR{\nOK}$.
Hence $\nBzl \cap \nBLz = \nBz$ and we conclude that $P_\infty = \roM_\infty$ is an isocrystal.
Thus $P$ has property~\eqref{premodgood} and the claim follows.%
\end{proof}

\begin{lem}\label{main}%
Assume that $A = \Fq[t]$. 
If the isocrystal $M_\infty$ has good reduction then the Drinfeld module $E$ has good reduction.%
\end{lem}%
\begin{proof}%
We are free to assume that $K$ is $\nval$-adically complete.
By Theorem~\ref{gardeyn} it is enough to prove that $M$ has a good model.
In view of Proposition~\ref{metagood} we are free to replace $K$ by a finite separable extension.
%
So it is enough to treat the case when $E$ has stable reduction.

As in \cite{gardeyn-anmor} we denote by $\nEK$ the ring of entire power series
\begin{equation*}
\sum_{n\geqslant 0} \alpha_n t^n, \quad\alpha_n \in K, \quad
\lim_{n \to \infty} |\alpha_n|^{\frac{1}{n}} = 0.%
\end{equation*}%
We equip $\nEK$ with an endomorphism 
$\nsigma\colon \sum_n \alpha_n t^n \mapsto
\sum_n \alpha_n^q t^n$.

Set $M^{\textup{an}} = \nEK \otimes_{\nAR{K}} M$.
Gardeyn \cite[Theorem 1.2]{gardeyn-anmor}
constructed a short exact sequence of left $\nEK\{\tau\}$-modules
\begin{equation}\label{stablefilt}\tag{\ensuremath{\ast}}
0 \to N^{\nan} \to M^{\nan} \to M_1^{\nan} \to 0
\end{equation}
with the following properties:
\begin{itemize}
\item
$M_1^{\nan}$ is locally free over $\nEK$ and its rank
is equal to the stable rank of $E$,

\item
$N^{\nan}$ is potentially trivial:
there is a finite separable extension $L/K$
and an isomorphism of left $\nER{L}\{\tau\}$-modules
$\nER{L}\otimes_{\nEK} N \xrightarrow{\isosign} (\unit_{\nER{L}})^{\oplus n}$.
\end{itemize}
Thanks to Proposition~\ref{metagood} we are free to assume that $L = K$.

Let $\nDK$ be the ring of Definition~\ref{defdk}
with $z = t^{-1}$.
The natural inclusion $\nEK \subset \nDK$ commutes with $\nsigma$.
The base change of \eqref{stablefilt} to $\nDK$ gives a sequence
\begin{equation*}
0 \to \cN \to \cM \to \cM_1 \to 0.
\end{equation*}
This is exact since the $\nEK$-module $M_1^{\nan}$ is flat.

By construction the $\nsigma$-bundle $\cM$ arises from the isocrystal $\rMloc$.
Let $\roM_\infty$ be the $\nBz$-model of $M_\infty$.
Proposition~\ref{modcmp} shows that $K\otimes_{\nOK} \roM_\infty = \rMloc$.
Hence $\cM = \nDK\otimes_{\nBz}\roM_\infty$.
Moreover $\roM_\infty$ is pure of slope $-\frac{1}{r}$ by Proposition~\ref{dmredpure}.

At the same time the $\nsigma$-bundle $\cN$ arises from a $\nBz$-isocrystal
$\roN$ which is a direct sum of copies of $\unit$. This isocrystal
is a fortiori pure of slope~$0$.
Applying Theorem~\ref{hpzero} we conclude that the morphism $\cN \to \cM$
arises from a morphism $\roN \to \roM_\infty$.

The isocrystals $\roN$, $\roM_\infty$ are pure of different slopes $0$, $-\frac{1}{r}$.
So Proposition~\ref{slopehom} shows that $\Hom(\roN,\roM_\infty) = 0$.
Since the morphism $\cN \to \cM$ is injective we conclude that $\cN = 0$.
Thus the rank of $E$ is equal to the stable rank and so $E$ has good reduction.%
%
\end{proof}


\begin{thm}\label{crit}%
Let $K$ be a field over $\Fq$ equipped with a valuation
$v\colon K^\times \twoheadrightarrow \bZ$ and let $\nOK$
be its ring of integers.
Let $E$ be a Drinfeld module over $\Spec K$ with the motive~$M$.
Assume that the residual characteristic of $E$ is finite
with respect to~$v$.
Then the following are equivalent:
\begin{enumerate}
\item\label{e-goodred}%
The Drinfeld module $E$ has good reduction over $\Spec\nOK$.
 
\item\label{minf-goodred}%
The isocrystal $M_\infty$ has good reduction over $\nOK$.

\item\label{tinf-ur}%
The Tate module $\nTinf{M}$ is unramified at the valuation $\nval$.
\end{enumerate}
\end{thm}%
\begin{proof}%
The implication
\eqref{e-goodred}~$\Rightarrow$~\eqref{minf-goodred}
is a consequence of Theorem~\ref{gardeyn}
and Proposition~\ref{dmgoodmod}.
The equivalence of \eqref{minf-goodred} and \eqref{tinf-ur} follows by Theorem~\ref{tatered}.
It remains to prove \eqref{minf-goodred} $\Rightarrow$ \eqref{e-goodred}.
Pick an element $t\in A$ which is transcendental over $\Fq$.
It is enough to show that $E$ has good reduction as a Drinfeld $\Fq[t]$-module.
The local field of $\Fq[t]$ at infinity is $\rFt{\Fq}$.
The analogs of rings $\nBKz$ and $\nBz$ for the coefficient field $\rFt{\Fq}$
are $\rFt{K}$ and $\rFt{\nOK}$ respectively.
The natural morphisms
\begin{equation*}
\rFt{K}\otimes_{K[t]} \nAR{K} \xrightarrow{\isosign} \nBKz, \quad
\rFt{\nOK}\otimes_{\nOK[t]} \nAR{\nOK} \xrightarrow{\isosign} \nBz.
\end{equation*}
are isomorphisms since
$F_\infty = \rFt{\Fq}\otimes_{\Fq[t]} A$.

Let $\roM_\infty \subset M_\infty$ be the good model.
Note the following facts:
\begin{itemize}
\item
The motive of the Drinfeld $\Fq[t]$-module $E$ is $M$
with the action of $K[t]\{\tau\}$ given by restriction of scalars
from $\nAR{K}\{\tau\}$.

\item $M_\infty = \rFt{K}\otimes_{K[t]} M$.

\item $\roM_\infty$ is a good model of the $\rFt{K}$-isocrystal $M_\infty$.%
\end{itemize}%
We thus reduce to the case $A = \Fq[t]$ where the claim follows from Lemma~\ref{main}.%
\end{proof}

\end{document}